%% file: ManuscriptMay2022.tex
\documentclass[a4paper, 11pt]{article}
\usepackage{preamble}

\begin{document}
\onehalfspacing

\title{Cutting Plane Approaches for the Robust Kidney Exchange Problem}
\author{Danny Blom\thanks{Corresponding author: d.a.m.p.blom@tue.nl}}
\author{Christopher Hojny}
\author{Bart Smeulders}
\affil{%
    Department of Mathematics and Computer Science, Eindhoven University of Technology,
    Groene Loper 5, 5612 AZ Eindhoven, The Netherlands
}

\date{\today}

\maketitle

\input{./ManuscriptMay2022/abstract}

\input{./ManuscriptMay2022/Introduction}

\input{./ManuscriptMay2022/RecoursePolicies}

\input{./ManuscriptMay2022/AlgorithmicFramework}

\input{./ManuscriptMay2022/Formulations}

\input{./ManuscriptMay2022/ComputationalResults}

\input{./ManuscriptMay2022/Conclusion}


\bibliographystyle{plainnat}
\bibliography{./ManuscriptMay2022/KEP}

\clearpage
\appendix

\input{./ManuscriptMay2022/Appendix.tex}

\end{document}

%% file: ManuscriptMay2022/abstract.tex
\begin{abstract}
\noindent Renal patients which have a willing but incompatible donor can decide to participate in a kidney exchange program (KEP). 
The goal of a KEP is to identify sets of such incompatible pairs that can exchange donors, leading to compatible transplants for each recipient.
There is significant uncertainty involved in this process, as planned transplants may be canceled for a plethora of reasons.
It is therefore crucial to take into account failures while planning exchanges.
In this paper, we consider a robust variant of this problem with recourse studied in the literature that takes into account vertex failures,~\ie~withdrawing donors and/or recipients.
This problem belongs to the class of defender-attacker-defender (DAD) models.
We propose a cutting plane method for solving the attacker-defender subproblem based on two commonly used mixed-integer programming formulations for kidney exchange.
Our results imply a running time improvement of one order of magnitude compared to the state-of-the-art.
Moreover, our cutting plane methods can solve a large number of previously unsolved instances.
Furthermore, we propose a new practical policy for recourse in KEPs and show that the robust optimization problem concerning this policy is tractable for small to mid-size KEPs in practice.\\[1ex]

\end{abstract}

%% file: ManuscriptMay2022/Introduction.tex
\section{Introduction}\label{sec:introduction}

In the final stage of chronic kidney disease, patients suffer from end-stage renal disease.
The most preferred treatment option for this disease is a kidney transplant.
However, the supply of healthy kidneys from deceased donors does not adequately meet the needs of these patients.
An alternative are living donations, e.g., by a relative or friend who is willing to donate one of their two kidneys to a specific patient.
This is possible since only one healthy kidney is necessary to ensure sufficient kidney function.
Transplants from living donors also offer better long-term outcomes for the recipient compared to deceased donor grafts, with limited risk to the donor, see, e.g., \cite{Davis2005}.
However, to allow for donations, recipients must be medically compatible with the donor.
Incompatibilities can be caused, among other reasons, by conflicting blood or tissue type.
Recipients might thus be unable to receive a transplant from their intended donor.

To overcome incompatibilities between a patient and its related donor,~\citet{Rapaport1986} introduced kidney exchange programs (KEPs), which have been further popularized by a series of seminal papers in the field (see~\cite{Roth2004,Roth2005,Roth2007}).
In a KEP, a set of incompatible patient-donor pairs is given, and every donor is willing to give one of their kidneys to any patient as long as their paired recipient receives a transplant in return from some other donor.
In the simplest case, one is thus looking for two incompatible donor-recipient pairs~$(p_1, p_2)$ such that the kidney of~$p_1$'s donor is compatible with the recipient of~$p_2$ and vice versa.
Once such a match is identified, the patients from~$p_1$ and~$p_2$ can receive the kidney from the other pair's donor.
We call such an exchange of donors a \emph{2-way exchange} as it includes two patient-donor pairs.
Of course, this idea can be generalized to \emph{$k$-way exchanges}, where the exchange of donors occurs in a cyclic way between~$k$ patient-donor pairs.
In many real-life KEPs, an upper bound~$\cyclen$ is used for the size of a cycle, see~\cite{Biro2019}.
As transplants within an exchange are usually performed simultaneously to remove the risk of donors reneging once their paired patient has received a kidney, long exchanges are logistically challenging and thus avoided.

A kidney exchange program can also include \emph{non-directed donors} (NDDs) as studied by~\cite{Morrissey2005} and~\cite{Roth2006}.
These are donors who are willing to donate to any recipient, while not having a paired patient and thus not requiring a return transplant.
An NDD can start a \emph{chain of transplants}, in which the donor of the $i$-th pair in the chain donates a kidney to the recipient of the $(i+1)$-th pair.
The donor of the last pair in the chain can donate their kidney to a recipient on the deceased donor waiting list, or becomes a bridge donor, functioning as a non-directed donor in a future KEP run.
In practice, it is often also assumed that the maximum chain length is bounded by some integer~$\chainlen$.

The basic KEP model can be regarded as a directed graph $G = (V,A)$ whose vertices $v \in V$ correspond to recipient-donor pairs and non-directed donors, and an arc $(i,j) \in A$ corresponds to compatibility between the donor of $i$ and the recipient of $j$.
Notice that each feasible $k$-way exchange corresponds to a cycle in this graph $G$.
Therefore, we refer to $k$-way exchanges as \emph{$k$-cycles}.
A feasible KEP solution is then a set of pairwise vertex-disjoint cycles and chains in~$G$ of lengths respecting the upper bounds.
A common objective in kidney exchange problems is to maximize the number of transplants that can be realized.
Figure~\ref{fig:example_kep} illustrates an example of a kidney exchange program and a feasible KEP solution.

\begin{figure}[htbp]
\centering
\includegraphics[scale = 0.9]{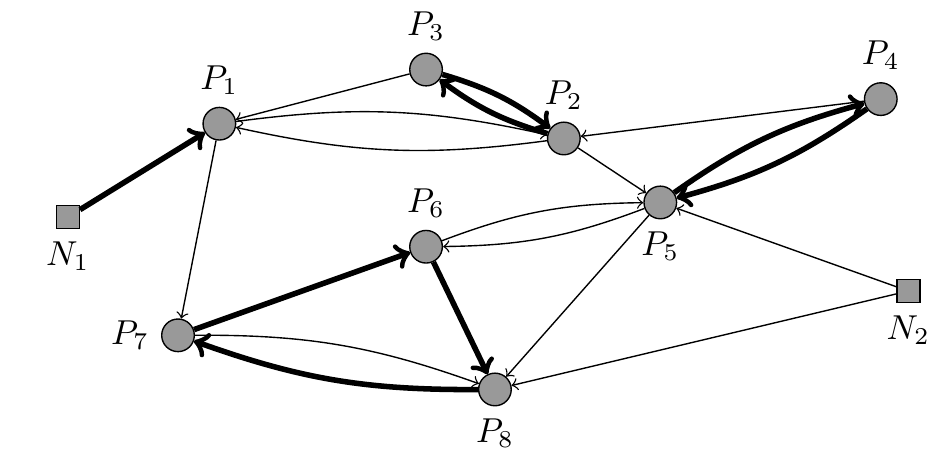}
\caption{Example of a kidney exchange program. Each vertex corresponds to a recipient-donor pair ($P_1,\dots,P_8$) or a non-directed donor ($N_1, N_2$). Arcs represent possible transplants.
  The bold arcs depict a feasible solution, with one chain of length one, two 2-cycles and one 3-cycle, implying eight transplants in total.\label{fig:example_kep}}
\end{figure}

Basic KEP models assume that all identified matches proceed to transplant.
In practice however, this is usually not the case.
For example,~\cite{Dickerson2019} report that \SI{93}{\percent} of proposed matches failed in the initial years of the UNOS program. 
In the NHS Living Kidney Sharing Scheme, \SI{43}{\percent} of identified matches did not proceed to transplant in the period from 2006 until 2021~\citep{NHS2021Report}.

Such failures can occur for a variety of reasons, \eg~when additional medical incompatibilities are identified during the time in between the matching run and the actual surgeries.
Furthermore, recipients and donors may decide to (temporarily) withdraw from KEPs due to health issues.
There exists a growing literature taking into account failures, which we discuss in more detail in Section \ref{subsec:literature_review}. 
We focus in particular on~\citet{Glorie20}, who introduced a robust optimization approach that addresses such donor (vertex) failures and aims to optimize a minimum guarantee for recipients that were selected for transplant in advance of the failure event.
They consider a three-stage model that belongs to the class of defender-attacker-defender (DAD) models, as optimizing worst-case behavior can be interpreted as playing a sequential game against an informed adversary.

In the first stage, the KEP owner (defender) identifies a feasible kidney exchange solution, which we refer to as the \emph{initial solution}.

Subsequently, it is anticipated that up to a specified number $B$ of pairs and / or NDDs might withdraw from the KEP.
In terms of DAD models, this corresponds to an adversary player selecting those pairs and NDDs that disrupt the initial solution as much as possible, therefore we refer to these withdrawals as an \emph{attack pattern} or \emph{attack}. 

In the third and final stage, a new kidney exchange solution is identified using only the remaining pairs and NDDs.
This kidney exchange solution is called the \emph{recourse solution} and it is selected such that it includes the largest number of patients also included in the initial solution. 
Additional restrictions on the set of possible recourse solution are captured through recourse policies.

The objective of the model is to select the initial solution that maximizes the worst-case number, with respect to all attacks, of patients involved in both the initial and recourse solution.
The idea behind this objective is to diminish the risk that recipients selected for transplant in the initial stage will end up without a transplant due to donor withdrawals.

The paper is structured as follows.
We start by indicating in~\Cref{subsec:our_contributions} our main contributions and review some relevant literature in~\Cref{subsec:literature_review}.
In~\Cref{sec:recourse_policies}, we consider two policies for recourse for the third stage of the defender-attacker-defender model.
The algorithmic framework for solving the attacker-defender subproblem is laid out in~\Cref{sec:algorithmic_framework}.
The implementation of this framework with respect to the cycle-chain and position-indexed chain edge formulations is made more concrete in~\Cref{sec:formulations}.
We finish by presenting our computational efforts in~\Cref{sec:computational_results} and a conclusion in~\Cref{sec:conclusion}.

\subsection{Our contributions}\label{subsec:our_contributions}
This paper will build on the work of \cite{Glorie20}, and the algorithmic literature on (D)AD models. Our main contributions are the following.
\begin{itemize}
    \item We introduce a new policy for recourse in kidney exchange programs, the \emph{Fix Successful Exchanges (FSE)} policy, where cycles and chains of the initial solution that are not affected by the subsequent attack will also be selected for the recourse solution. 
    This recourse policy further protects recipients involved in a selected cycle or chain of the initial solution, and has additional logistical benefits.
    We further motivate the FSE policy in~\Cref{sec:recourse_policies}. 
    Our computational results show that imposing the FSE policy for our recourse options leads to only a small loss in the attainable objective value.
    \item We make use of the position-indexed chain-edge formulation (PICEF) for kidney exchange that is compact with respect to the set of chains.
    We show that the resulting formulation of the second stage is significantly stronger than those of extended formulations such as the cycle-chain formulation. 
    Specifically, we show the strength of interdiction cuts depends on the underlying formulation, provided we consider the form proposed by \cite{Fischetti2019}.
    \item We propose a new approach based on cutting planes to solving the second-stage problem,~\ie the interdiction problem faced by the attacker. Our computational results show that this new approach outperforms the branch-and-bound algorithm of~\cite{Glorie20} by an order of magnitude.
    This improvement can be observed both with regards to the average computation times and the number of instances that could be solved within the time limit of one hour.
    
\end{itemize}

\subsection{Literature review}\label{subsec:literature_review}
In this section, we position our paper in the literature on failure in kidney exchanges and discuss some formulations for the basic kidney exchange problem which we adapt to our setting. Finally, we mention relevant literature on (Defender)-Attacker-Deference models.

Failure in kidney exchange was first considered by~\cite{Dickerson2019}. They study a setting with equal failure probabilities and equal value for each transplant and maximize the expected number of transplants. 
\cite{Bidhkori2020} extend this model to a setting with inhomogeneous arc existence probabilities and gives a compact formulation.
Since stochastic models require accurate estimations for arc existence probabilities, an alternative model based on robust optimization is proposed in~\cite{McElfresh2019}. 
In addition to handling the possibility of failures a priori, an additional decision can be made once the failures are observed, called the \emph{recourse decision}.
In a recourse decision, the initially proposed solution can be adapted subject to some constraints, which might allow for a significant percentage of additional transplants to be realized, see, e.g.,~\cite{Bray15}.
\cite{Alvelos2019} and \cite{Klimentova16} consider recourse in a stochastic setting. They initially select subsets of recipient-donor pairs. After failures are revealed, transplant cycles are selected by only making use of transplants between pairs within the same initial subset. Initial selections are made so as to maximize the expected number of transplants. \cite{Smeulders2022} studies a setting where initially a set of potential transplants, limited by a budget constraint, are tested. Again, the goal is to select these arcs so as to maximize the expected number of transplants after failure is revealed.
\cite{Glorie20} consider recourse in a robust setting, as described in the introduction.

Most algorithms related to kidney exchange programs are based on integer programming (IP).
An intuitive IP formulation is the \emph{cycle-chain formulation (CC)} due to~\cite{Roth2007}, which uses a binary variable for each cycle and each chain, and constraints for each pair and each NDD $j$, stating that only one cycle or chain can be chosen that includes $j$.
The number of variables in these formulations grow rapidly as the cycle and chain limits increas, and with increasing numbers of participating pairs and NDDs increases.
A common technique for solving the cycle-chain model is branch-and-price, see~\cite{Glorie2014} and~\cite{plaut2016fast}.
Another fundamental formulation is the \emph{edge formulation} due to~\cite{Abraham07}.
Instead of cycle and chain variables, this formulation uses binary variables $y_{ij}$ for each $(i,j) \in A$, that indicate whether the donor of $i$ donates a kidney to the recipient of $j$, leading to a polynomial number of variables.
However, the number of constraints is exponential in the number of pairs and NDDs.
\cite{Constantino13} introduced the first \emph{compact formulations}, i.e.,
models with a polynomial number of variables and constraints.
The drawback of this formulation is that its LP relaxation is generally weaker than the one of the CC formulation.
\cite{Dickerson2016} introduced the \emph{position-indexed cycle edge formulation (PICEF)}, which is a hybrid variant of different models:
as the CC formulation, it uses binary variables for each cycle.
But instead of using variables for each chain, chains are modeled by variables that express the position of an individual transplant within a chain. The resulting formulation is thus compact concerning the chains.
\citet{Dickerson2016} show that the LP relaxation of PICEF is strictly weaker than that of CC. However, we will use the PICEF formulation and show its advantages in the robust setting we consider.

As we study a worst-case setting, following \cite{Glorie20}, the failures can be modeled as deliberate attacks of an optimizing adversary.
\citet{Riascos2022} have studied a generalization of this model based on inhomogeneous failures of both vertices and arcs.
This makes the model a specific example of a general class of optimization models known as the class of \emph{defender-attacker-defender} (DAD) models.
Problems involving network interdiction are often intuitively modeled by a DAD model, many of which are surveyed in~\cite{Smith2020}.
The second and third stages of the model we study can separately be seen as an attacker-defender model, and specifically an interdiction game with monotonicity as studied,~\eg~by \cite{Fischetti2019}. 
We will indeed model these stages as an interdiction game and make use of their general framework.


%% file: ManuscriptMay2022/RecoursePolicies.tex
\section{Recourse Policies}\label{sec:recourse_policies}
In this section, we consider two policies for recourse in kidney exchange programs: the \emph{Full Recourse (FR)} policy proposed by~\citet{Glorie20} and the novel \emph{Fix Successful Exchanges (FSE)} policy.

\subsection{Full recourse}\label{subsec:full_recourse}
In the Full Recourse (FR) policy, no additional restriction is imposed on the set of recourse solutions.
As a result, any feasible kidney exchange solution that does not involve withdrawing donors is also a feasible recourse solution.
One clear advantage of FR is that it provides the highest degree of flexibility in selecting recourse solutions.
Nevertheless, the full recourse policy might leave individual recipients dissatisfied.
Namely, recipients that are selected for a (successful) exchange with respect to our initial solution are not guaranteed to receive a transplant according to the recourse solution.
Moreover, FR requires logistical flexibility by transplant centers organizing a kidney exchange.
In many cases, an exchange means that kidneys are transported from one facility to be delivered to the facility where the recipient is treated.
This freedom of selecting any possible kidney exchange with respect to the Full Recourse setting might cause issues as some donor kidneys are located far away from the intended destination,~\ie~the recipient's facility.

\subsection{Fix Successful Exchanges}\label{subsec:fix_successful_exchanges}
With these issues in mind, we now introduce a recourse policy that addresses these drawbacks and aims at minimizing such disadvantages.
We propose the \emph{Fix Successful Exchanges (FSE)} policy, which imposes additional constraints on the set of feasible recourse solutions.

Under the FSE policy, any cycle that is both (i) included in the initial solution and (ii) not affected by any donor withdrawals, is required to be included in the recourse solution.
Chains are (partially) enforced up to and including the final pair in the chain before the first pair with a withdrawing donor.
In the recourse step, we do not allow the extension of the enforced (partial) chains with alternative transplants.
In practice, the matching algorithms for kidney exchange are run frequently and by allowing the extension of enforced chains, we restrict the possible exchanges for the next matching run.
In other words, we fix the successful (parts of) initial exchanges to be included in the recourse solution.
Furthermore, any cycle and chain involving the remaining pairs and NDDs could be added to the recourse solution.

This recourse policy has significant advantages from an organizational and logistical point of view.
Cycles and chains without failures can immediately proceed to transplant, thus reducing the waiting time for the involved recipients.
Under full recourse, the recipients involved in such successful exchanges are not guaranteed to receive a transplant, and pairs might get reassigned to new exchanges.
The FSE policy thus reduces uncertainty for recipients and transplant centers alike.



The drawback of the FSE policy is that the set of feasible recourse solutions is restricted.
Consequently, it might occur that in the final recourse solution concerning FSE, the number of transplanted recipients is strictly smaller than it would be the case with FR.
Nevertheless, we believe that the negative impact of deploying the FSE policy will be limited.
Our computational results in~\Cref{sec:computational_results} show that for most instances, the attainable objective value for the FSE setting either coincides or is just one less compared to the objective value in the Full Recourse setting.
%
The FSE policy provides a solution concept that both benefits the donor and recipient experience while reducing the complexity of the logistics involved in long-distance kidney exchanges.

%% file: ManuscriptMay2022/AlgorithmicFramework.tex
\section{Algorithmic Framework}\label{sec:algorithmic_framework}

In this section, we reconsider the defender-attacker-defender model for robust kidney exchange proposed by~\citet{Glorie20} and algorithmic concepts used for solving this model.
We start by providing some preliminary knowledge on modelling basic kidney exchange programs.
Secondly, we describe the DAD model in a general sense, and concretize it for both recourse policies discussed in Section~\ref{sec:recourse_policies}.
Finally, we describe a cutting plane approach for solving the adversary problem.

\subsection{Preliminaries}\label{subsec:preliminaries}
The basic KEP model considers a \emph{compatibility graph} $G  = (V,A)$.
The vertex set $V$ consists of a set~$P$ of incompatible donor-recipient pairs and a set $N$ of non-directed donors, \ie~$V = P \cup N$.
Each arc $(i, j) \in A$ indicates that the donor of vertex $i \in P \cup N$ is compatible with the recipient of vertex $j \in P$.
As each arc corresponds to transplanting a donor kidney to a recipient, we will also refer to arcs as \emph{transplants}.
We denote by $\C_{\cyclen}$ the set of (directed) \emph{cycles} in $G$ with at most~$K$ arcs.
A directed path originating from some non-directed donor $n \in N$ is called a (directed)
\emph{chain}, where the length of the chain is given by the number of arcs
therein.
The set of all chains in~$G$ of length at most~$\chainlen$ is denoted by~$\D_{\chainlen}$.
We refer to $\E_{\cyclen, \chainlen} = \C_{\cyclen} \cup \D_{\chainlen}$ as the set of feasible \emph{exchanges}, or simply $\E$ if $K$ and $L$ are clear from the context.

Given an exchange $e \in \E$, we denote by $V(e)$ and $A(e)$ the set of vertices and arcs involved in $e$.
We denote by $\E^j$ the set of exchanges involving $j$, as well as $\C_{\cyclen}^j$ and $\D_{\chainlen}^j$ for cycles and chains involving $j$, respectively.
Furthermore, we denote by $\F$ the set of feasible kidney exchange solutions on $G$, \ie
\[
    \F \define \set{ X \subseteq \E \given V(e) \cap V(e') = \emptyset \text{ for all } e, e' \in X,\ e \ne e'}.
\]
We denote by $\F^{U}$ the subset of feasible kidney exchange solutions restricted to the subgraph $G[U]$ induced by the vertex subset $U \subset V$,~\ie
\[ 
    \F^{U} \define \F \cap \set{ X \subseteq \E \given V(e) \subseteq U \text{ for all } e \in X }.
\]

\subsection{Robust kidney exchange model}\label{subsec:robust_kidney_exchange_model}
The model introduced by~\citet{Glorie20} considers a variant of kidney exchange programs, in which donor withdrawal is taken into account in a worst-case setting.
This problem can be interpreted as a two-player game,
in which the KEP owner (defender) plays against an optimizing adversary (attacker).
In this setting, the attacker is not considered to be an actual decision maker, but rather a concept for modelling the worst-case realization of donor withdrawals.
Given an initial solution proposed by the defender and an attack proposed by the adversary, the defender's goal is to identify a feasible recourse solution for which the maximum number of patients that were also involved in some exchange of the initial solution receive a transplant.
In other words, this model belongs to the class of defender-attacker-defender models, since first the KEP owner makes a ``here-and-now'' decision (selects an initial solution), then the adversary reacts (selects an attack pattern) and the KEP owner finally adapts their solution based on the attack it has just observed (selects a recourse or ``wait-and-see'' decision).
The high-level aim of the defender is to coordinate their initial solution such that the objective value of the recourse problem is maximized in the case of the worst-case attacker's decision.
Consequently, solving the underlying robust optimization problem for the defender consists of solving a trilevel optimization problem.
In the following, we provide the details of this trilevel problem.

Let us denote by $\X$, $\U$, and $\Y_{\Pi}$ respectively the feasible sets of each of the three subsequent decisions to be made in the trilevel problem. $\Pi \in \set{\fr, \fse}$ captures
the recourse policy under consideration.

Note that the set of recourse solutions depends on preceding decisions, as both the initial solution and the observed second-stage attack restrict which exchanges should or cannot be selected in the recourse solution.
Therefore, we write $\Y_{\Pi}(\vecx,\vecu)$ to denote feasible recourse solution given an initial solution $\vecx \in \X$ and attack $\vecu \in \U$.

Using the above notation, the robust optimization problem for a fixed recourse policy~$\Pi \in \set{\fr, \fse}$ can be modeled as a defender-attacker-defender model
\begin{equation}\label{eq:DAD_model_general_form}
    \max_{\vecx \in \X} \min_{\vecu \in \U} \max_{\vecy \in \Y_{\Pi}(\vecx,\vecu)} f(\vecx, \vecu, \vecy),
\end{equation}
where $f\colon \{ (\vecx, \vecu, \vecy) \in \X \times \U \times \Y_{\Pi} \mid \vecy \in \Y_\Pi(\vecx, \vecu)\} \to \Z_{\ge 0}$ is an objective function dependent on the recourse policy~$\Pi$.
In the remainder of this section, we specify the exact structure of the sets~$\X$, $\U$, and~$\Y_\Pi$ as well as the objective function~$f$.

\paragraph{Initial solutions}~\\
In the first stage---the \emph{initial stage}---no additional constraints are imposed on initial solutions. The set $\X$ of feasible initial solutions is defined as

\begin{equation}\label{eq:feasible_region_stage_1}
    \X \define \set{ \vecx \in \set{0,1}^{m_1} \given \vecx \text{ encodes a feasible initial solution } X \in \F }.
\end{equation}

Here, we use $m_1$ to denote the length of the encoding.
Typically, this depends on the choice of the mathematical programming formulation at hand.
Below, we will briefly discuss different formulations and provide their details in Section~\ref{sec:formulations}.
With a slight abuse of notation, we will refer to both the encoding $\vecx \in \X$ and its associated set of exchanges $X \in \F$ as \emph{initial solutions}.
Furthermore, we denote by $P(\vecx) \subseteq P$ the set of pairs---specifically, the set of recipients---involved in an exchange corresponding to the initial solution $\vecx \in \X$.

One intuitive option to encode a solution $X \in \F$ is to consider a binary vector $\vecx = (x_e)_{e \in \E}$ having an entry for each possible exchange $e \in \E$ such that $x_e = 1$ if $e \in X$, and $x_e = 0$ otherwise.
The classical mixed-integer programming (MIP) formulation for kidney exchange programs using this solution encoding is called the \emph{cycle(-chain) formulation},~\eg~\citet{Roth2007} and~\citet{manlove12},~as this formulation includes a binary variable for each possible cycle and chain.

Alternatively, kidney exchange solutions can be encoded by binary vectors with indices given by the arcs of the underlying compatibility graph $G$.
In its most basic form, such an encoding $\vecx = (x_{ij})_{(i,j)\in A}$ of a feasible KEP solution $X \in \F$ is given by $x_{ij} = 1$ if the recipient of vertex $j \in P$ receives a kidney donation from the donor of vertex $i \in P \cup N$ with respect to $X$, and $x_{ij} = 0$ otherwise.
Commonly used MIP formulations for kidney exchange using this type of solution encodings are called \emph{edge formulations}.
Examples of edge formulations are the path-based formulation given by~\citet{Abraham07}, and the position-indexed edge formulations of~\citet{Dickerson2016} which also consider the position of an arc within a chain.

\paragraph{Attacks}~\\
For the second stage---the \emph{attacker's stage}---we consider an attacker's budget $B \in \mathbb{Z}_{>0}$ on the number of pairs and NDDs that can be ``attacked'' (corresponding to donor withdrawals).
The set $\U$ of feasible attacks is given by

\begin{equation}\label{eq:feasible_region_stage_2}
    \U \define \set{ \vecu \in \bins^{V} \given \sum_{j \in V} u_j \le B},
\end{equation}
\ie~each attack $\vecu \in \U$ corresponds to a subset of at most $B$ donors that withdraw from the program.
Given this definition of $\U$, for any vertex $j \in P \cup N$, we consider $j$ to be \emph{attacked} if $u_j = 1$ and \emph{not attacked} if $u_j = 0$.
We denote by $V(\vecx, \vecu)$ the set of attacked vertices given initial solution $\vecx \in \X$ and attack $\vecu \in \U$,~\ie~$j \in V(\vecx, \vecu)$ if and only if $u_j = 1$.
Although $V(\vecx, \vecu)$ does not explicitly depend on the initial solution $\vecx$, we use this notation with two arguments to emphasize that attacks correspond to second-stage decisions.

\paragraph{Recourse solutions}~\\
In the final stage of the model---the \emph{recourse stage}---the defender is given the opportunity to adapt the initial solution based on the observed attacked donors.
Given encodings $\vecx \in \X$ and $\vecu \in \U$, corresponding to an initial solution $X \in \F$ and subset $U = V(\vecx, \vecu) \subset V$ of attacked vertices, let $\F_{\Pi}(X, U) \subset \F^{V \setminus U}$ denote the subset of feasible recourse solutions on the induced subgraph $G[V \setminus U]$ respecting policy $\Pi \in \set{\fr, \fse}$.
Hence, we can define
\begin{equation}
    \Y_{\Pi}(\vecx, \vecu) \define \set{ \vecy \in \bins^{m_1} \given \vecy \text{ encodes a feasible recourse solution $Y \in \F_{\Pi}(X,U)$ } }.
\end{equation}
We will use the notation $P(\vecx, \vecu, \vecy) \subseteq P$ to denote the set of pairs---specifically, the set of recipients---involved in an exchange of the recourse solution $\vecy \in \Y_{\Pi}(\vecx, \vecu)$.
We now formalize the notions of $\F_{\Pi}(X,U)$ for $\Pi = \fr$ and $\Pi =\fse$ respectively as
\begin{itemize}
    \item Full Recourse:
    \[ \F_{\fr}(X,U) \define \F^{V\setminus U},\]
    \ie~the set of feasible recourse solutions respecting the full recourse policy does not depend on the initial solution $X \in \F$, but is only restricted by the set $U$ of attacked vertices.
    \item Fix Successful Exchanges:
    \[ \F_{\fse}(X,U) \define \set{Y \in \F \given (e \in X \wedge V(e) \subseteq V \setminus U) \Rightarrow e \in Y}.\]
\end{itemize}
Finally, remark that we can now write the objective function introduced by~\citet{Glorie20} as
\[
    f(\vecx, \vecu, \vecy) \define |P(\vecx) \cap P(\vecx, \vecu, \vecy)|,
\] 
\ie~the objective function is the number of patients that are both (i) involved in an exchange of the initial solution $\vecx \in \X$ and (ii) involved in an exchange of the recourse solution $y \in \Y_{\Pi}(\vecx, \vecu)$ after observing attack $\vecu \in \U$.
In the next section, we will formalize how the feasible sets of all three stages of the model could be captured in two (mixed-)integer programming formulations commonly used for modeling kidney exchange problems.

\subsection{Stage 1: column-and-constraint generation (C\&CG) algorithm}\label{subsec:column_and_constraint_generation_algorithm}
One approach for solving such defender-attacker-defender models is through first reformulating the model as a single-level optimization problem.
The robust kidney exchange problem with recourse---given recourse policy $\Pi$---can then be described in general as

\begin{maxi!}
    {\vecx, \vecy^{\U} = (\vecy^{\vecu})_{\vecu\in \U}, Z}{ Z\label{obj:robust_kep_general} }{\label{eq:robust_kep_general}}{z_{\Pi}(\U)\define}
    \addConstraint{Z}{\le |P(\vecx) \cap P(\vecx, \vecu, \vecy^{\vecu})|,\label{con:robust_kep_general:minimize_over_attacks}}{\qquad \vecu \in \U}
    \addConstraint{\vecy^{\vecu}}{\in \Y_{\Pi}(\vecx, \vecu),\label{con:robust_kep_general:y_feasible}}{\qquad \vecu \in \U}
    \addConstraint{\vecx}{\in \X\label{con:robust_kep_general:x_feasible}}{}
    \addConstraint{Z}{\ge 0.\label{con:robust_kep_general:Z_feasible}}{}
\end{maxi!}
We show now that $z_{\Pi}(\U)$ provides a valid formulation for the trilevel optimization problem.
Objective~\eqref{obj:robust_kep_general} introduces a dummy variable $Z$ that will be used to capture the worst-case behaviour.
Constraints~\eqref{con:robust_kep_general:y_feasible} and~\eqref{con:robust_kep_general:x_feasible} ensure that $\vecx$ and $\vecy^{\vecu}$ are feasible initial and recourse solution under attack $\vecu \in \U$, respectively.
It then follows from Constraints~\eqref{con:robust_kep_general:minimize_over_attacks} that the value of $Z$ in an optimal solution will be equal to 
\[ Z \define \min_{\vecu \in \U} |P(\vecx) \cap P(\vecx, \vecu, \vecy^{\vecu})|, \]
where $\vecx \in \X$ is an optimal initial solution.

The advantage of transforming a trilevel optimization model into a single-level optimization model is that it allows for employing techniques known from literature, such as algorithms for solving mixed-integer programs (MIPs).
We will show in Section~\ref{sec:formulations} that for common encodings of kidney exchange solutions, this model can be written as a mixed-integer linear program.
Nevertheless, the linearization of a multi-level optimization model comes at the cost of creating extremely large MIPs.
As each attack~$\vecu \in \U$ appears as a variable index and constraint index in Formulation~\eqref{eq:robust_kep_general}, both the number of variables and the number of constraints are extremely large.
Today's MIP solvers are not able to deal with such large formulations, hence we cannot simply provide this full description to a MIP solver.
Nevertheless, the structure of this MIP is such that the constraint matrix consists of blocks per attack $\vecu \in \U$ that are coupled by means of the column corresponding to $Z$.

One commonly used approach for solving MIP models of this structure is the \emph{column-and-constraint generation algorithm (C\&CG)}, as described below.
We refer the reader to~\citet{Zeng2013} for a description of the column-and-constraint generation algorithm in general two-stage robust optimization models.
In our specific setting, the C\&CG algorithm works as follows:
\begin{enumerate}
\item Let $\bar{\U} \subset \U$ be a small subset of feasible attacks.\label{ccg:step1}
\item Solve the restricted problem $z_\Pi(\bar{\U})$. Let $(\bar{\vecx}, \bar{\vecy}^{\bar{\U}}, \bar{Z})$ be an optimal solution.\label{ccg:step2}
\item Solve the bilevel subproblem given optimal initial solution $\bar{\vecx} \in \X$:
  \begin{equation}\label{eq:bilevel_subproblem}
    s_\Pi(\bar{\vecx}) \define \min_{\vecu \in \U} \max_{\vecy^{\vecu} \in \Y_{\Pi}(\bar{\vecx},\vecu)} |P(\bar{\vecx}) \cap P(\bar{\vecx}, \vecu, \vecy^{\vecu})|.
  \end{equation}
  Let $(\bar{\vecu}, \bar{\vecy}^{\bar{\vecu}})$ be an optimal solution to $s_{\Pi}(\bar{\vecx})$. \label{ccg:step3}
\item If $s_\Pi(\bar{\vecx}) < \bar{Z}$, let $\bar{\U} = \bar{\U} \cup \set{\bar{\vecu}}$ and go to Step~\ref{ccg:step2}. Otherwise, return $\bar{Z}$ as then $z_{\Pi}(\bar{\U}) = z_{\Pi}(\U)$. \label{ccg:step4}
\end{enumerate}
From a high-level perspective, one solves the problem while only accounting for a restricted set $\bar{\U}$ of feasible attacks. The initial solution $\bar{\vecx}$ is then tested against other attacks to see if there exists an attack~${\bar{\vecu} \in \U \setminus \bar{\U}}$ that is not yet considered, such that the optimal recourse solution $\bar{\vecy}^{\bar{\vecu}}$ is such that it violates the current solution of the restricted problem, i.e.,
\[ |P(\bar{\vecx}) \cap P(\bar{\vecx}, \bar{\vecu}, \bar{\vecy}^{\bar{\vecu}})| < \bar{Z}.\]
If this is true, it is clear that the current set of attacks $\bar{\U}$ is not enough to describe worst-case behaviour and we will need to add $\bar{\vecu}$ to it.
As $|\U|$ is finite, this algorithm will eventually terminate such that the problem is solved to optimality.

\subsection{Stage 2: cutting plane algorithm}\label{subsec:algorithmic_framework:stage2}
\citet{Glorie20} proposed a branch-and-bound type algorithm for solving $s_\Pi(\vecx)$, where branching is performed on the $\vecu$-variables representing attacks.
In a nutshell, this gives rise to a sequence of subproblems, in which some vertices are fixed to be attacked ($u_j=1$) or not attacked ($u_j = 0$). 
For each subproblem, a lower bound is computed by considering the initial solution $\vecx$ and the variable fixings, and extending the attack in a greedy fashion as if no recourse options exist.
Aside from that, the best recourse solution $\vecy \in \Y_{\Pi}(\vecx, \vecu)$ with respect to this greedily extended attack $\vecu$ and policy $\Pi$ is computed.
As $u$ is not necessarily optimal to $s_{\Pi}(\vecx)$, this gives us an upper bound to the optimal value.
The column-and-constraint generation method, using this branch-and-bound approach for generating columns and constraints, was tested in the computational experiments of~\citet{Glorie20}.
It was shown that for realistic instances with 50 incompatible donor-recipient pairs and attack budget $B \in \{1,2,3,4\}$, this method was able to solve the entire defender-attacker-defender model within the time limit of one hour, but struggled with instances of 100 pairs due to large sizes of the branch-and-bound trees involved.

In this paper, we propose an alternative algorithm based on cutting planes for solving the bilevel subproblem $s_{\Pi}(\bar{\vecx})$ with the aim to provide a method that is able to solve larger instances.
This is done through linearizing the model by first projecting out the variables of the lower-level problem, namely the variables indicating the recourse solution $\vecy^{\vecu} \in \Y_{\Pi}(\vecx,\vecu)$.
Additional constraints are imposed to ensure that the restrictions of the recourse policy are properly addressed.

We propose a reformulation of $s_\Pi(\vecx)$ based on \emph{optimal value functions}, see~\eg~\citet{Dempe2015}.
The optimal value function of this formulation can be written as

\begin{equation}\label{eq:subproblem_general:value_function}
    \phi_\Pi(\vecx, \vecu) \define \max_{\vecy^{\vecu} \in \Y_{\Pi}(\vecx, \vecu)} |P(\vecx) \cap P(\vecx, \vecu, \vecy^{\vecu})|.
\end{equation}
The issue with this optimal value function is that the feasible set $\Y_{\Pi}(\vecx, \vecu)$ depends on the attack $\vecu$, which means that the feasible set depends on a preceding decision of the attacker. 
In order to be able to reformulate the problem as a single-level mixed integer program, we first need to formalize an optimal value function for which the feasible set does not depend on preceding decisions.
For each initial solution $\vecx$ and attack $\vecu$, there exists a vector $\vecz \in \bins^{m_1}$ that satisfies the following property;
for any feasible recourse solution $\vecy^{\vecu} \in \Y_{\Pi}(\vecx, \vecu)$, there exists a kidney exchange solution $\hat{\vecx} \in \X$ such that $\vecy^{\vecu} = \vecz \odot \hat{\vecx}$, where $\odot$ denotes the entry-wise product of two vectors.
Informally speaking, the entries of $\vecz$ that are equal to zero correspond to a forbidden part of a feasible recourse solution,~\eg~an entry related to an arc containing an attacked vertex.
In other words, each optimal recourse solution is obtained from some maximal KEP solution after removing the exchanges with at least one vertex attacked by $\vecu$.
This is true as the kidney exchange polytope $\conv(\X)$ has the \emph{monotonicity} property,~\ie~for any feasible $\vecx \in \conv(\X)$ and any $\vecx'$ with $0 \le \vecx' \le \vecx$, we have that also $\vecx' \in \conv(\X)$.
A more general description of this type of reformulation for bilevel optimization problems where the underlying polytope has the monotonicity property can be found in~\cite{Fischetti2019}.

Furthermore, we need to consider additional restrictions on $\vecz$ to ensure that the recourse policy $\Pi$ is respected.
For the recourse policies $\Pi$ we consider, this can be captured through a system of linear inequalities $C_\Pi \vecz \ge d_\Pi(\vecx, \vecu)$.
The optimal value function can thus be rewritten as 

\begin{equation}\label{eq:subproblem_general:value_function:rewrite}
    \phi_{\Pi}(\vecx, \vecu) \define \max_{\vecz \in \bins^{m_1}} \max_{\hat{\vecx} \in \X} |P(\vecx) \cap P(\vecx, \vecu, \vecz \odot \hat{\vecx})|.
\end{equation}

In that way, the bilevel optimization problem $s_{\Pi}(\vecx)$ can be reformulated as a single-level optimization problem, as the feasible sets do not depend on preceding decisions anymore.
The inner maximization problem can then be reformulated by adding constraints for each feasible KEP solution $\hat{\vecx} \in \X$.
As a result, we get the following single-level reformulation for $s_{\Pi}(\vecx)$.

\begin{mini!}
    {\vecz, \vecu, Z}{ Z }{\label{obj:subproblem_general}}{s_{\Pi}(\vecx; \X)\define}
    \addConstraint{Z}{\ge |P(\vecx) \cap P(\vecx, \vecu, \vecz \odot \hat{\vecx})|,\label{con:subproblem_general:interdiction_cuts}}{\qquad \hat{\vecx} \in \X}
    \addConstraint{C_{\Pi} \vecz}{\ge d_{\Pi}(\vecx, \vecu)\label{con:subproblem_general:forcing_cuts}}{}
    \addConstraint{\vecu}{\in \U\label{con:subproblem_general:u_feasible}}{}
    \addConstraint{\vecz}{\ge 0\label{con:subproblem_general:z_feasible}}{}
    \addConstraint{Z}{\in \R.\label{con:subproblem_general:Z_feasible}}{}
\end{mini!}

One drawback of this reformulation is that projecting out the lower-level variables $\vecy \in \Y_{\Pi}(\vecx, \vecu)$ comes at the expense of an exponential number of constraints of type~\eqref{con:subproblem_general:interdiction_cuts}.
Following \citet{Fischetti2019}, who introduced this type of  constraints, we refer to these as \emph{interdiction cuts}, as the $\vecz$-variables can interdict exchanges that cannot occur in a recourse solution based on the restrictions of $\Y_{\Pi}(\vecx, \vecu)$ determined by initial solution $\vecx$, attack $\vecu$, and recourse policy $\Pi$. 

As kidney exchange instances consist of rather sparse graphs, the size $m_1$ of solution encodings is typically relatively small compared to the number of constraints, supporting the use of a cutting plane method to solve this problem.
We start by considering the restricted problem $s_\Pi(\vecx; \bar{\X})$ based on initial solution $\vecx$ and a small subset $\bar{\X} \subset \X$ of feasible kidney exchange solutions.
We can describe the algorithm as follows: 
\begin{enumerate}
	\item Solve the problem $s_\Pi(\vecx; \bar{\X})$,~\ie~with Constraints~\eqref{con:subproblem_general:interdiction_cuts} restricted to $\hat{\vecx} \in \bar{\X}$, with optimal solution~$(\bar{\vecz}, \bar{\vecu}, \bar{Z})$.\label{cpa:step1}
	\item Solve the separation problem
		\begin{equation}\label{eq:subproblem_general:separation_problem}
            R_{\Pi}(\vecx, \bar{\vecu}) \define \max_{\vecy \in \Y_{\Pi}(\vecx, \bar{\vecu})} |P(\vecx) \cap P(\vecx, \bar{\vecu}, \vecy)| \left( = \max_{\vecy \in \Y_{\Pi}(\vecx, \bar{\vecu})} w(\vecx, \vecu, \vecy) \right).
		\end{equation}
		This is just a weighted kidney exchange problem on the induced subgraph $G[V \setminus V(\vecx, \bar{\vecu})]$.
		Additional constraints might be imposed---based on recourse policy $\Pi$---that fix variables to zero and one, thus simplifying the problem. 
		The efficiency of the cutting plane algorithm is based on the empirical observation that kidney exchange problems are easy in practice although NP-hard in theory.
		\label{cpa:step2}
	\item If $R_{\Pi}(\vecx, \bar{\vecu}) < z_{\Pi}(\bar{\U})$, that means $\bar{\vecu}$ needs to be added to $\bar{\U}$.
	In other words, $\bar{\vecu}$ corresponds to both columns and constraints that need to be added to the restricted master problem.
	Otherwise, let~$\bar{\vecy} \in \argmax R_{\Pi}(\vecx, \bar{\vecu})$.
	If $R_\Pi(\vecx, \bar{\vecu}) > \bar{Z}$, let $\bar{\X} = \bar{\X} \cup \set{\vecy}$ (we add an interdiction cut to the model) and go to Step~\ref{cpa:step1}. 
	Otherwise, it holds that $s_\Pi(\vecx; \bar{\X}) = R_\Pi(\vecx, \bar{\vecu})$, meaning that the entire trilevel model is solved to optimality. \label{cpa:step3}
\end{enumerate}

\noindent Note that the feasible region of the separation problem~\eqref{eq:subproblem_general:separation_problem} coincides with the set $\Y_{\Pi}(\vecx, \bar{\vecu})$ of feasible recourse solutions with respect to policy $\Pi$.
This does however mean that we cannot consider any exchanges with at least one attacked vertex.
Typically, an optimal solution $\vecy^{*} \in \Y_{\Pi}(\vecx, \bar{\vecu})$ will correspond to a non-maximal kidney exchange solution.
In other words, there exists some different solution~$\vecy' \in \X \setminus \set{\vecy^{*}}$ with $\vecy' \ge \vecy^{*}$, i.e., the solution $\vecy'$ consists of the exchanges of an optimal recourse solution $\vecy^{*}$ together with at least one attacked exchange.
However, the interdiction cut corresponding to $\vecy'$ is strictly stronger than the one corresponding to $\vecy^{*}$. Hence, it is desirable that we can actually separate the stronger type of cuts.
We will refer to these stronger cuts as \emph{lifted interdiction cuts}, as they can be obtained from the original separated cuts by lifting coefficients for the attacked exchanges $e$ for which $y'_{e} = 1$ and $y^{*}_{e} = 0$.

In order to separate lifted interdiction cuts, we need to consider a modified variant $R'_{\Pi}(\vecx, \bar{\vecu})$ of the separation problem. Based on the approach of~\citet{Fischetti2019},
the modified separation problem can be formulated as

\begin{maxi!}
    {\hat{\vecx} \in \X}{|P(\vecx) \cap P(\vecx, \bar{\vecu}, \vecy(\hat{\vecx},\bar{\vecu}))|\label{obj:separation_problem_lifting}}{}{R'_{\Pi}(\vecx, \bar{\vecu}) \define}
    \addConstraint{\vecy(\hat{\vecx},\bar{\vecu})}{\text{ encodes all non-attacked exchanges of } \hat{\vecx} \text{ given } \bar{\vecu}\label{con:separation_problem_lifting:feasible_recourse_y}}{}
    \addConstraint{\vecy(\hat{\vecx},\bar{\vecu})}{\in \argmax_{\bar{\vecy} \in \Y_{\Pi}(\vecx, \bar{\vecu})} |P(\vecx) \cap P(\vecx, \bar{\vecu}, \bar{\vecy})|.\label{con:separation_problem_lifting:optimal_recourse_y}}{}
\end{maxi!}

\noindent In this model, Constraint~\eqref{con:separation_problem_lifting:feasible_recourse_y} ensures that $\vecy(\hat{\vecx}, \vecu)$ does not include cycles with at least one attacked vertex.
Together with Constraints~\eqref{con:separation_problem_lifting:optimal_recourse_y}, it follows that $\vecy(\hat{\vecx}, \vecu)$ is feasible optimal to the original separation problem $R_\Pi(\vecx, \bar{\vecu})$. 
Objective~\eqref{obj:separation_problem_lifting} ensures that we identify a maximal KEP solution $\hat{\vecx}$. 
Notice that the infeasibilities of this maximal KEP solution $\hat{\vecx}$ are dealt with in the subproblem through Constraints~\eqref{con:subproblem_general:forcing_cuts} on~$\vecz$.
Therefore, we can now obtain stronger interdiction cuts for the subproblem without cutting off feasible solutions.
One example of how deploying this modified separation problem can lead to significantly stronger cuts is depicted in Figure~\ref{fig:lifting}.
Remark that the example shows that the modified separation problem leads to strictly stronger cuts, for both recourse policies $\fr$ and $\fse$.


\begin{figure}[htb]
\centering
\includegraphics[scale = 0.9]{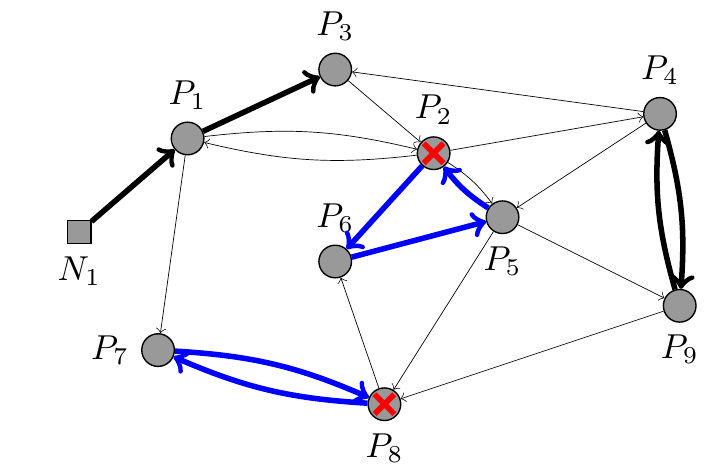}
\caption{(colored print) Given initial solution $\vecx \in \X$ (black and blue thick arcs) and attack $\vecu \in \U$ with $V(\vecx, \vecu) = \set{ P_2, P_8 }$, the optimal solution to the separation problem $R_{\Pi}(\vecx, \vecu)$ is given by the black thick arcs, while it coincides with $\vecx$ for the modified separation problem $R'_{\Pi}(\vecx, \vecu)$. \label{fig:lifting}}
\end{figure}

We will show that the modified separation problem is not significantly harder to solve than the original version. The reason for this is that we can rewrite $R'_{\Pi}(\vecx, \vecu)$ as a weighted kidney exchange problem.
Let us denote by $|\E(\hat{\vecx})|$ the number of exchanges (cycles and chains) encoded by $\hat{\vecx} \in \X$.
One can check that the objective of the modified separation problem can be reformulated as
\begin{equation}\label{eq:subproblem_general:separation_problem_lifting_as_weighted_KEP}
	R'_\Pi(\vecx, \vecu) \define \max_{\hat{\vecx} \in \X} |P(\vecx) \cap P(\vecx, \vecu, \vecy(\hat{\vecx},\vecu))|\cdot|V|+|\E(\hat{\vecx})|.
\end{equation}
When using~\eqref{eq:subproblem_general:separation_problem_lifting_as_weighted_KEP} as the objective function, we first prioritize finding a maximal kidney exchange solution $\hat{\vecx} \in \X$ for which the corresponding recourse solution $\vecy(\hat{\vecx}, \vecu)$ is optimal under $\vecu$.
Secondly, if multiple optima exist, ties are broken by selecting the optimal solution with the largest number of exchanges (which therefore must be a maximal solution).

An additional beneficial consequence is that it favors solutions with a large number of $2$-cycles and short chains.
This makes the cuts even stronger, as interdicting short exchanges (attacking one of its vertices) has a smaller impact.
In Section~\ref{sec:computational_results}, insights are provided on the impact of
deploying the original and modified separation problems on the cutting
plane algorithm.


%% file: ManuscriptMay2022/Formulations.tex
\section{Formulations}\label{sec:formulations}

In this section, we provide concrete formulations for the models described in Section~\ref{sec:algorithmic_framework}.
We start by providing the formulations based on the cycle-chain (CC) formulations for basic kidney exchange problems for both recourse policies in Subsection~\ref{subsec:CC}.
Secondly, we do the same for the position-indexed chain-edge formulation (PICEF) in Subsection~\ref{subsec:PICEF}.
In particular, we formalize that each of the three decisions to be made can be optimized by solving single-level mixed-integer linear programming (MILP) models. 

\subsection{Cycle-chain formulation}\label{subsec:CC}

\subsubsection{CC: Master problem}\label{subsubsec:CC:master_problem}
Firstly, we reconsider the CC-based formulation for $z_{FR}(\U)$, which was proposed by~\cite{Glorie20}.
It considers a vector $\vecx = (x_e)_{e \in \E}$ of binary variables for each exchange $e \in \E$ indicating whether or not $e$ is selected for the initial solution.
Furthermore, we consider a vector $\vecy = (y_{e}^\vecu)_{e \in E, \vecu \in \U}$ of binary variables for each exchange $e \in \E$ and each attack $\vecu \in \U$, indicating whether or not $e \in E$ is selected for the recourse solution under $\vecu$.
The vector of variables $\vecz = (z_{j}^{\vecu})_{j \in P, \vecu \in \U}$ indicates for each pair $j \in P$ and each attack $\vecu \in \U$ whether or not the recipient of pair $j \in P$ receives a transplant with respect to both the initial and recourse solution under $\vecu$.
If we reconsider Model~\eqref{eq:robust_kep_general}, Constraints~\eqref{con:robust_kep_general:minimize_over_attacks} are modelled by Constraints~\eqref{con:z_FR:CC:overlapping_pairs}-\eqref{con:z_FR:CC:recourse_pairs}.
The feasibility of initial and recourse solutions is guaranteed by Constraints~\eqref{con:z_FR:CC:initial_packing} and~\eqref{con:z_FR:CC:x_variables}, and Constraints~\eqref{con:z_FR:CC:recourse_packing} and~\eqref{con:z_FR:CC:y_variables} respectively.
\begin{maxi!}<b>
  {\vecx, \vecy, \vecz, Z}{Z\label{obj:z_FR:CC}}{\label{eq:z_FR:CC}}{z_{\fr}(\U)\define}
  \addConstraint{Z - \sum_{j \in P} z_{j}^\vecu}{\le 0,\label{con:z_FR:CC:overlapping_pairs}}{\qquad \vecu \in \U}
  \addConstraint{z_{j}^\vecu - \sum_{e \in \E^j} x_e }{\le 0,\label{con:z_FR:CC:initial_pairs}}{\qquad \vecu \in \U, j \in P}
  \addConstraint{z_{j}^\vecu - \sum_{e \in \E^j} y_{e}^\vecu}{\le 0,\label{con:z_FR:CC:recourse_pairs}}{\qquad \vecu \in \U, j \in P}
  \addConstraint{\sum_{e \in \E^j} x_e }{\le 1,\label{con:z_FR:CC:initial_packing}}{\qquad j \in P \cup N}
  \addConstraint{\sum_{e \in \E^j} y_{e}^\vecu}{\le 1 - u_j,\label{con:z_FR:CC:recourse_packing}}{\qquad \vecu \in \U, j \in P \cup N}
  \addConstraint{\vecx}{\in \bins^{\E}\label{con:z_FR:CC:x_variables}}
  \addConstraint{\vecy^\vecu}{\in\bins^{\E},\label{con:z_FR:CC:y_variables}}{\qquad \vecu \in \U}
  \addConstraint{\vecz^\vecu}{\ge 0,\label{con:z_FR:CC:z_variables}}{\qquad \vecu \in \U}
  \addConstraint{Z}{\ge 0.\label{con:z_FR:CC:Z_variable}}
\end{maxi!}

In order to obtain the CC-based formulation for $z_{\fse}(\U)$, we need to modify Constraints \eqref{con:z_FR:CC:recourse_pairs} and~\eqref{con:z_FR:CC:recourse_packing}.
Let us denote by $\E(\vecu) \define C_\cyclen(\vecu) \cup \D_{\chainlen}(\vecu)$ the set of exchanges with no interdicted vertices with respect to $\vecu \in \U$,~\ie~
\[
    \E(\vecu) \define \set{ e \in \E \mid \sum_{i \in V(e)} u_i = 0 } 
\]
where $\C_{\cyclen}(\vecu)$ and $\D_{\chainlen}(\vecu)$ are the analogous notions for cycles and chains.
We also consider a superscript for each $j \in V \cup N$ to restrict ourselves to the cycles and chains involving $j$.
Furthermore, we define for each chain $d = \set{v_{0}, v_{1}, \ldots, v_{q}} \in \D_{\chainlen}$ with $q = |A(d)| \le \chainlen$, and each vertex $j \in V(d)$ the subchain 
\[
    d^{\to j} \define
                \begin{cases}
                    \set{ v_{0}, v_{1} }, & \text{ if } j = v_{0}\\
                    \set{ v_{0}, \ldots, j }, & \text{ otherwise.}
                \end{cases}
\]
In other words, for each pair $j \in P$, we get that $d^{\to j}$ is the subchain of $d$ whose final arc goes into $j$, whereas for each non-directed donor $j \in N$ it consists of the first arc only.
We can interpret it as the smallest nonempty subchain of $d$ containing vertex $j$.
Now, we define for each $j \in P \cup N$ the set
\[ \I^j(\vecu) \define \C_{\cyclen}(\vecu) \cup \set{d \in \D_{\chainlen} \mid d^{\to j} \in \D_{\chainlen}(\vecu)}, \]
\ie~the set of exchanges that result in an enforced (partial) exchange involving $j$ in the recourse solution under $\vecu \in \U$, provided it was selected for the initial solution.

Hence, we obtain a CC-based formulation for $z_{FSE}(\U)$ by replacing Constraints~\eqref{con:z_FR:CC:recourse_pairs} and \eqref{con:z_FR:CC:recourse_packing} with the following constraints:
\begin{subequations}
    \begin{align}
        z_{j}^{\vecu} - \sum_{e \in \I^j(\vecu)} x_e - \sum_{e \in \E^j(\vecu)} y_{e}^{\vecu} &\le 0, & & \qquad \vecu \in \U, j \in P, \label{con:z_FSE:CC:recourse_pairs} \\ 
        \sum_{e \in \I^j(\vecu)} x_e + \sum_{e \in \E^j(\vecu)} y_{e}^{\vecu} &\le 1 - u_j,& &\qquad \vecu \in \U, j \in P \cup N. \label{con:z_FSE:CC:recourse_packing}
    \end{align}
\end{subequations}
Notice that in comparison with the $\fr$ setting, the interpretation of $\vecy^\vecu$ has changed in the $\fse$ setting.
In the $\fr$ setting, the vector $\vecy^{\vecu}$ encodes all the exchanges of the recourse solution, whereas this vector only encodes exchanges of the recourse solution that were not (partially) enforced in the $\fse$ setting.
In Constraints~\eqref{con:z_FSE:CC:recourse_pairs} and~\eqref{con:z_FSE:CC:recourse_packing}, this is reflected by an additional term with variables $x_e$ for each $e \in \I^j(\vecu)$, indicating (partially) enforced initial exchanges in which $j \in P \cup N$ is involved.

\subsubsection{CC: Attacker-defender subproblem}\label{subsubsec:CC:attacker_defender_subproblem}
We now consider MILP reformulations of the subproblem $s_{\Pi}(\vecx)$ faced by the ``attacker'' in the second-stage of the trilevel model, given the initial solution $\vecx \in \X$.
Notice that whenever an attack $\vecu \in \U$ has been realized, a recourse solution $\vecy^{\vecu} \in \Y_{\Pi}(\vecx, \vecu)$ will be selected that does not consider exchanges that have an attacked vertex.
For each exchange $e \in \E$, we define the weight $w_e(\vecx)$ as
\[ w_e(\vecx) \define |P(\vecx) \cap V(e)|. \]
Furthermore, let us denote by $\C_{\cyclen}(\vecx)$ and $\D_{\chainlen}(\vecx)$ the sets of initial cycles and initial chains respectively.
Following the framework described in Subsection~\ref{subsec:algorithmic_framework:stage2}, we obtain the following formulation for the bilevel subproblem.

\begin{mini!}<b>
  {\vecz, \vecu, Z}{Z\label{obj:s_FR:CC}}{\label{eq:s_FR:CC}}{s_{\fr}(\vecx; \X)\define}
  \addConstraint{Z}{\ge \sum_{e \in S} w_e(\vecx) z_e,\label{con:s_FR:CC:interdiction_cuts}}{\qquad S \in \X}
  \addConstraint{z_e}{\ge 1 - \sum_{j \in V(e)} u_j,\label{con:s_FR:CC:lb1_z}}{\qquad e \in \E}
  \addConstraint{\vecz}{\ge 0\label{con:s_FR:CC:lb2_z}}
  \addConstraint{\vecu}{\in \U\label{con:s_FR:CC:attack_budget}}
  \addConstraint{Z}{\ge 0.\label{con:s_FR:CC:dummy_variable}}    
\end{mini!}

Here, Constraints~\eqref{con:s_FR:CC:interdiction_cuts} are a concrete description of Constraints~\eqref{con:subproblem_general:interdiction_cuts} of the general framework.
Moreover, Constraints~\eqref{con:s_FR:CC:lb1_z} and~\eqref{con:s_FR:CC:lb2_z} correspond to the system of linear inequalities $C_{\fr}z \ge d_{\fr}(\vecx, \vecu)$ given in~\eqref{con:subproblem_general:forcing_cuts}.
Notice that for this model, optimal solutions minimize the value of the $z_e$ variables, with Constraints~\eqref{con:s_FR:CC:lb1_z} and~\eqref{con:s_FR:CC:lb2_z} providing a lower bound.

In order to adapt it for the setting of the $\fse$ policy, we need to make a distinction between enforceable and non-enforceable exchanges first.
To that end, we define for each chain $d \in \D_{\chainlen}$ the set of \emph{subchains} of $d$ as
\[
    \sub(d) \define \set{ d^{\to j} \mid j \in P(d) }.
\]
Furthermore, let $\C_{\cyclen}(\vecx)$ and $\D_{\chainlen}(\vecx)$ denote the sets of initial cycles and initial chains, respectively.
The set of \emph{enforceable exchanges} can then be defined as
\[ 
    \enf(\vecx) \define \C_{\cyclen}(\vecx) \cup
                       \bigcup_{d \in \D_{\chainlen}(\vecx)} \sub(d),
\]
\ie~the collection of all initial cycles and all subchains of the initial chains selected for exchanges that could be enforced given initial solution $\vecx \in \X$.
It depends on the realization of the true attack $\vecu \in \U$ which of the exchanges in $\enf(\vecx)$ will be enforced.
Finally, we obtain the formulation $s_{\fse}(\vecx; \X)$ by 
replacing Constraints~\eqref{con:s_FR:CC:lb1_z} with the following linear inequalities:
\begin{subequations}
    \begin{align}
        z_e &\ge 1 - \sum_{j \in V(e)} u_j, && e \in \enf(\vecx)\label{con:s_FSE:CC:lb1_z}\\
        z_e &\ge 1- \sum_{j \in V(e)} u_j - \sum_{\substack{\tilde{e} \in \enf(\vecx):\\ \tilde{e} \cap e \ne \emptyset}} z_{\tilde{e}}, && e \in \E \setminus \enf(\vecx)\label{con:s_FSE:CC:lb2_z}\\
        z_e &\le 1 - u_j, && e \in \E, j \in V(e)\label{con:s_FSE:CC:ub1_z}
    \end{align}
\end{subequations}
Notice that Constraints~\eqref{con:s_FSE:CC:lb1_z} is still the same for exchanges $e \in \enf(\vecx)$,~\ie~for the exchanges that are enforceable.
However, Constraints~\eqref{con:s_FSE:CC:lb2_z} indicate that if $e \notin \enf(\vecx)$, it can only be selected for a feasible recourse solution whenever none of the enforced exchanges shares a vertex with $e$.
In this setting, it could be beneficial to set $z_e = 1$ for an attacked exchange, to allow $z_{e'} = 0$ for a higher weight exchange through Constraints~\eqref{con:s_FSE:CC:lb2_z}, therefore we need to add Constraints~\eqref{con:s_FSE:CC:ub1_z} to ensure correct values for the $z_e$-variables.
This correctly models the impact of the FSE recourse policy, as it prioritizes enforceable exchanges.

\subsubsection{CC: Finding optimal recourse solutions}\label{subsubsec:CC:recourse_problem}
Notice that the formulation of the attacker-defender subproblem $s_{\Pi}(\vecx; \X)$ contains an exponential number of Constraints~\eqref{con:s_FR:CC:interdiction_cuts}.
We will refer to these constraints as \emph{interdiction cuts}, as these constraints implicitly model for each feasible kidney exchange solution $S \in \X$ which exchanges are not interdicted with donor withdrawals.
The algorithm proposed in Subsection~\ref{subsec:algorithmic_framework:stage2} relies on an efficient method for finding violated constraints.
Below we describe the corresponding separation problem.
Given the initial solution $\vecx \in \X$ and attack $\vecu \in \U$, we define the weights $w_e(\vecx, \vecu) = w_e(\vecx)$ for each exchange $e \in \E$.
This allows us to write the problem $R_{\fr}(\vecx, \vecu)$ of finding the optimal recourse solution $\vecy^\vecu \in \Y_{\fr}(\vecx, \vecu)$ as a single-level mixed-integer program, namely
\begin{maxi!}<b>
  {\vecy}{\sum_{e \in \E} w_e(\vecx, \vecu) y_{e}^\vecu\label{obj:R_FR:CC}}{\label{eq:R_FR:CC}}{R_{\fr}(\vecx, \vecu)\define}
  \addConstraint{\sum_{e \in \E^j} y_{e}^\vecu}{\le 1 -u_j,\label{con:R_FR:CC:packing}}{\qquad j \in P \cup N}
  \addConstraint{\vecy^\vecu}{\in \bins^{\E}.}
\end{maxi!}

Constraints~\eqref{con:R_FR:CC:packing} impose that when vertex $j \in P \cup N$ is interdicted, then none of the exchanges~${e \in \E^j}$ can be selected for the recourse solution. Therefore, any optimal solution to Problem~\eqref{eq:R_FR:CC} is also a feasible recourse solution under $\vecu$.
In other words, Problem~\eqref{eq:R_FR:CC} is equivalent to the problem of finding a weighted kidney exchange solution on the subgraph $G[V \setminus V(\vecx, \vecu)]$.
We consider a modified version~$R'_{\fr}(\vecx, \vecu)$ of the separation problem that ``lifts'' optimal solutions for~\eqref{eq:R_FR:CC} to maximal kidney exchange solutions on the entire compatibility graph $G$, with the aim of separating stronger interdiction cuts~\eqref{con:s_FR:CC:interdiction_cuts}.
We can obtained this modified separation problem by making two modifications.
First, we need to replace objective coefficients $w_e(\vecx, \vecu)$ by alternative coefficients 
\begin{align*}
    w'_e(\vecx, \vecu) 
        \define 
        \begin{cases}
            w_e(\vecx, \vecu)|V|+1, & \text{ if } e \in \E^j(\vecu)\\
            1, & \text{ otherwise.}
        \end{cases}
\end{align*}
and secondly, we need to relax the right hand side of Constraints~\eqref{con:R_FR:CC:packing} to 1 such that also interdicted exchanges could be selected in the solution.
The weights are defined in order to prioritize non-interdicted exchanges of high original weight $w_e(\vecx, \vecu)$, and actually ensure that the set of non-interdicted exchanges in the optimal solution indicate an optimal solution for the original separation problem $R_{\fr}(\vecx, \vecu)$.
Ties are broken by selecting the solution with the largest number of exchanges (the +1 term).
Our experiments have shown that utilizing this modified separation problem leads to strictly stronger interdiction cuts at no significant expense in terms of computation time.

Furthermore, we can obtain the formulations for $R_{\fse}(\vecx, \vecu)$ and $R'_{\fse}(\vecx, \vecu)$ by considering the exchanges $e \in \enf(\vecx)$ that are actually enforced.
For each initial chain $d \in \D_{\chainlen}(\vecx)$, we now only consider its longest subchain with no interdicted vertices to be enforced. For these exchanges, we fix the variable $y_{e}^\vecu$ to 1.

\subsection{Position-indexed chain-edge formulation}\label{subsec:PICEF}
We also provide formulations based on the position-indexed chain-edge formulation (PICEF) introduced by~\cite {Dickerson2016}.
PICEF can be regarded as a hybrid variant of the cycle-chain formulation and edge formulation.
Its main advantage is that it avoids full enumeration of all chains in the compatibility graph, as the number of chains might be exponential in $\chainlen$.
Instead, it considers a formulation of chains in terms of arc variables with an additional index for their position in a chain.
For basic kidney exchange problems, this is of particular interest when the maximum chain length $\chainlen$ increases.

Similar to CC, it considers a vector of binary variables $\vecx = (x_{c})_{c \in \C_{\cyclen}}$ indicating whether or not each cycle $c \in \C_{\cyclen}$ is selected for the initial solution.
The chains~$d$ are described by a collection of variables $\xi_{ij,\ell}$, for each arc $(i,j) \in A(d)$ and feasible position in a chain $\ell \in \L(i,j) \subseteq \set{1,\dots, |A(d)|}$ in $d$. 
Here, arcs going out of the non-directed donor have position index 1, the subsequent arc has index 3, etc.
Hence, an arc $(i,j)$ with $i \in N$ can only have index 1,~\ie~$\L(i,j) = \set{1}$.
Conversely, any~$(i,j) \in A$ with $i \in P$ can have any list $\L(i,j) \subseteq \L \setminus \set{1}$ of possible indices excluding 1.
We refer to the paper by~\citet{Dickerson2016} for techniques for reducing the sizes of these lists $\L(i,j)$, as this leads to formulations with (significantly) fewer variables.

We denote by $\P = \set{ (i,j,\ell) \in A \times \L \mid \ell \in \L(i,j) }$ the set of position-indexed arcs that can be used as building blocks for chains in $\D_{\chainlen}$.
The formulation based on PICEF will thus take into account a vector of binary variables $\vecxi = (\xi_{ij,\ell})_{(i,j,\ell) \in \P}$ indicating whether or not each position-indexed arc $(i,j,\ell) \in \P$ is selected for the initial solution.
We denote the entire encoding of initial solution by $\vecx' = (\vecx, \vecxi)$.

Furthermore, we let $\P^j$ denote the set of position-indexed arcs involving vertex $j \in P \cup N$,~\ie~
\[
    \P^j \define \set{ (u, v, \ell) \in \P \mid j = u \text{ or } j = v}. 
\]

\subsubsection{PICEF: Master problem}\label{subsubsec:PICEF:master_problem}
Similar to the initial solution encoding $\vecx'$, we reconsider for each attack~$\vecu \in \U$ the vector of binary variables $\vecy^{\vecu} = (y_c^{\vecu})_{c \in \C_{\cyclen}}$ describing whether or not cycle $c \in \C_{\cyclen}$ is selected for the recourse solution under~$\vecu \in \U$, and introduce the vector of binary variables $\vecpsi^\vecu = (\psi_{ij,\ell}^{\vecu})_{(i,j,\ell)\in\P}$ to describe chains selected for the recourse solution under $\vecu \in \U$. 
Hence, we denote by $\vecy' = (\vecy^{\vecu}, \vecpsi^{\vecu})_{\vecu \in \U}$ the entire encoding of recourse solutions under all possible attacks.
Similar to the CC version, we consider the vector of binary variables $\vecz = (z_{j}^\vecu)_{j \in P, \vecu \in \U}$ to indicate whether or not pair $j \in P$ is involved in an initial exchange and an exchange in the recourse solution under $\vecu \in \U$.

{\small
    \begin{maxi!}<b>
    { \vecx', \vecy', \vecz, Z }{Z\label{obj:z_FR:PICEF}}{\label{eq:z_FR:PICEF}}{z_{FR}(\U)\define}
    \addConstraint{Z - \sum_{j \in P} z_{j}^\vecu}{\le 0,\label{con:z_FR:PICEF:overlapping_pairs}}{\qquad \vecu \in \U}
    \addConstraint{z_{j}^\vecu - \sum_{c \in \C_{\cyclen}^j} x_{c} - \sum_{(i,j,\ell) \in \P^j} \xi_{ij,\ell}}{\le 0,\label{con:z_FR:PICEF:initial_pairs}}{\qquad \vecu \in \U, j \in P}
    \addConstraint{z_{j}^\vecu - \sum_{c \in \C_{\cyclen}^j} y_c^{\vecu} - \sum_{(i,j,\ell) \in \P^j} \psi^{\vecu}_{ij, \ell}}{\le 0,\label{con:z_FR:PICEF:recourse_pairs}}{\qquad \vecu \in \U, j \in P}
    \addConstraint{\sum_{c \in \C_{\cyclen}^j} x_{c} + \sum_{(i,j,\ell) \in \P^j} \xi_{ij,\ell} }{\le 1,\label{con:z_FR:PICEF:initial_packing_pairs}}{\qquad j \in P}
    \addConstraint{\sum_{c \in \C_{\cyclen}^j} y_c^{\vecu} + \sum_{(i,j,\ell) \in \P^j} \psi_{ij,\ell}^{\vecu} }{\le 1 - u_j,\label{con:z_FR:PICEF:recourse_packing_pairs}}{\qquad \vecu \in \U, j \in P}
    \addConstraint{\sum_{(j,i,1) \in \P^j} \xi_{ji,1} }{\le 1, \label{con:z_FR:PICEF:initial_packing_NDDs}}{\qquad j \in N}
    \addConstraint{\sum_{(j,i,1) \in \P^j} \psi^{\vecu}_{ji,1} }{\le 1 - u_j, \label{con:z_FR:PICEF:recourse_packing_NDDs}}{\qquad \vecu \in \U, j \in N}
    \addConstraint{\sum_{(j,i,\ell+1) \in \P^j} \xi_{ji, \ell+1} - \sum_{(i,j,\ell) \in \P^j} \xi_{ij,\ell}}{\le 0,\label{con:z_FR:PICEF:initial_precedence}}{\qquad j \in P, \ell \in \L \setminus \set{L}}
    \addConstraint{\sum_{(j,i,\ell+1) \in \P^j} \psi^{\vecu}_{ji,\ell+1} - \sum_{(i,j,\ell) \in \P^j} \psi^{\vecu}_{ij,\ell}}{\le 0,\label{con:z_FR:PICEF:recourse_precedence}}{\qquad \vecu \in \U, j \in P, \ell \in \L \setminus \set{L}}
    \addConstraint{\vecx}{\in \bins^{\C_{\cyclen}} \label{con:z_FR:PICEF:x_variables}}
    \addConstraint{\vecxi_{ij,\ell}}{\in \bins,\label{con:z_FR:PICEF:xi_variables}}{\qquad (i,j,\ell) \in \P}
    \addConstraint{\vecz^\vecu}{\ge 0,\label{con:z_FR:PICEF:z_variables}}{\qquad \vecu \in \U}
    \addConstraint{\vecy^\vecu}{\in \bins^{\C_{\cyclen}},\label{con:z_FR:PICEF:y_variables}}{\qquad \vecu \in \U}
    \addConstraint{\vecpsi_{ij,\ell}^\vecu}{\in \bins,\label{con:z_FR:PICEF:psi_variables}}{\qquad (i,j,\ell) \in \P, \vecu \in \U}
    \addConstraint{Z}{\ge 0.\label{con:z_FR:PICEF:Z_variable}}
    \end{maxi!}
}

Here, notice that Constraints~\eqref{con:z_FR:PICEF:initial_pairs}-\eqref{con:z_FR:PICEF:recourse_packing_pairs} are obtained from Constraints~\eqref{con:z_FR:CC:initial_pairs}-\eqref{con:z_FR:CC:recourse_packing} of the CC-based master problem by replacing, for each~$j \in P$, the variables $x_{d}$ for each chain $d \in \D^j_\chainlen$ with the variables~$\xi_{i,j,\ell}$ for all position-indexed arcs $(i,j,\ell) \in \P^j$.
 Constraints~\eqref{con:z_FR:PICEF:initial_packing_NDDs} and~\eqref{con:z_FR:PICEF:recourse_packing_NDDs} indicate that for each~$j \in N$, at most one of the arcs going out of $j$ can be selected for the initial and recourse solution.
Finally, Constraints~\eqref{con:z_FR:PICEF:initial_precedence} and~\eqref{con:z_FR:PICEF:recourse_precedence} are \emph{precedence constraints} for the initial and recourse solutions respectively, indicating that for each $j \in P$ and $\ell \in \L \setminus \set{\chainlen}$, one can only select an arc going out of $j$ with index~$\ell + 1$ if an arc going into $j$ is selected with index $\ell$.

A number of modifications are required to arrive at the PICEF-based formulation of the master problem~$z_{\fse}(\U)$ with respect to the FSE recourse policy.
For an initial solution $\vecx' = (\vecx, \vecxi)$, let us denote by~$\C_{\cyclen}(\vecx)$ and $\P(\vecxi)$ ($\D_\chainlen(\vecxi)$) the set of cycles and the set of position-indexed arcs (chains) induced by initial solution $\vecx'$.
Similarly, we can define the set $\enf(\vecx')$ of enforceable exchanges by
\[
    \enf(\vecx') \define \C_\cyclen(\vecx) \cup \bigcup_{d \in \D_\chainlen(\vecxi)} \sub(d)\ \left( = \enf(\vecx) \cup \enf(\vecxi) \right).
\]
Notice that given an attack $\vecu \in \U$, it is not directly clear if a position-indexed arc $(i,j,\ell) \in \P$ with~$\xi_{ij,\ell} = 1$ should be enforced in a recourse solution.
Let $d \in \D_\chainlen(\vecxi)$ be the unique chain containing the arc $(i,j)$ at position index $\ell$.
Then, the position-indexed arc variable for $(i,j,\ell) \in \P$ is enforced whenever the subchain $d^{\to j}$ contains no interdicted vertices.
Hence, we would like to track for each arc $(i,j) \in A$ whether or not 
\begin{enumerate}
    \item[(i)] it is a part of an initial chain $d \in \D_\chainlen(\vecxi)$,~\ie~$\xi_{ij,\ell} = 1$ for some $\ell \in \L(i,j)$ and
    \item[(ii)] $d^{\to j} \in \D_{\chainlen}(\vecu)$.
\end{enumerate}
To this end, we introduce binary variables $\beta_{ij}^{\vecu} \in \bins$ for each arc $(i,j) \in A$ and each attack $\vecu \in \U$.
This can be modelled by adding the following constraints:
\begin{subequations}\label{eq:z_FSE:PICEF:enforcement_constraints}
  \begin{align}
    \beta_{ij}^{\vecu} &\le \sum_{\ell \in \L(i,j)} \xi_{ij,\ell}, & & \vecu \in \U, (i,j) \in A\label{con:z_FSE:PICEF:initial}\\
    \beta_{ij}^{\vecu} &\ge \xi_{ij,1} - u_i - u_j, & & \vecu \in \U, (i,j) \in A: i \in N\label{con:z_FSE:PICEF:lb_arc_NDDs}\\
    \beta_{ij}^{\vecu} &\ge \sum_{\ell \in \L(i,j)} \xi_{ij,\ell} + \sum_{(k,i) \in A} \beta_{ki}^{\vecu} - 1 - u_j, & &\vecu \in \U, (i,j) \in A: i \in P\label{con:z_FSE:PICEF:lb_arc_pairs}\\
    \sum_{(i,j) \in A} \beta_{ij}^{\vecu} &\le \sum_{(k,i) \in A} \beta_{ki}^{\vecu}, & &\vecu \in \U, i \in P\label{con:z_FSE:PICEF:beta_precedence}\\
    \beta_{ij}^{\vecu} &\le 1 - u_i, & & \vecu \in \U, i \in P \cup N\label{con:z_FSE:PICEF:ub_src}\\
    \beta_{ij}^{\vecu} &\le 1 - u_j, & &\vecu \in \U, j \in P\label{con:z_FSE:PICEF:ub_dst}
  \end{align}
\end{subequations}
Condition (i) is required by Constraints~\eqref{con:z_FSE:PICEF:initial}.
 Constraints~\eqref{con:z_FSE:PICEF:lb_arc_NDDs} to~\eqref{con:z_FSE:PICEF:ub_dst}, together with the precedence constraints~\eqref{con:z_FSE:PICEF:beta_precedence}, are used for enforcing arc $(i,j) \in A$ to be used in the recourse solution under $\vecu$ if~$(i,j)$ is part of an initial chain $d \in \D_{\chainlen}(\vecxi)$ with $d^{\to j} \in \D_{\chainlen}(\vecu)$.
This definition allows us to rewriting Constraints~\eqref{con:z_FR:PICEF:recourse_pairs},~\eqref{con:z_FR:PICEF:recourse_packing_pairs} and~\eqref{con:z_FR:PICEF:recourse_packing_NDDs} of $z_{\fr}(\U)$ as

\begin{subequations}\label{eq:z_FSE:PICEF:adapting_recourse_constraints}
\begin{align}
    z_{j}^{\vecu} - \sum_{c \in \C_{\cyclen}^j(\vecu)} x_{c} - \sum_{(i,j) \in A} \beta_{ij}^{\vecu} - \sum_{c \in \C_{\cyclen}^j} y_c^{\vecu} - \sum_{(i, j, \ell) \in \P^j} \psi^{\vecu}_{ij, \ell} &\le 0, &\vecu \in \U, j \in P\\
    \sum_{c \in \C_{\cyclen}^j(\vecu)} x_{c} + \sum_{(i,j) \in A} \beta_{ij}^{\vecu} + \sum_{c \in \C_{\cyclen}^j} y_c^{\vecu} + \sum_{(i, j, \ell) \in \P^j} \psi^{\vecu}_{ij, \ell} &\le 1 - u_j, & \vecu \in \U, j \in P\\
    \sum_{(j,i) \in A} \left(\beta_{ji}^{\vecu} + \psi^{\vecu}_{ji,1}\right) &\le 1 - u_j, & \vecu \in \U, j \in N.
\end{align}
\end{subequations}
Adding Constraints~\eqref{eq:z_FSE:PICEF:enforcement_constraints} and modifying the recourse constraints to~\eqref{eq:z_FSE:PICEF:adapting_recourse_constraints} gives us the PICEF-based formulation for $z_{\fse}(\U)$.

\subsubsection{PICEF: Attacker-defender subproblem}\label{subsubsec:PICEF:attacker_defender_subproblem}
We provide a PICEF-based formulation for the subproblem $s_{\Pi}(\vecx'; \X)$ faced by the attacker.
Similarly to the CC-based formulation~\eqref{eq:s_FR:CC}, we have $\vecz$-variables for each cycle $c \in \C_{\cyclen}$ such that $z_{c} = 1$ if $c \in \C_{\cyclen}(\vecu)$, and $z_{c} = 0$ otherwise.
However, we cannot follow a similar approach with variables $\zeta_{ij,\ell}$ for position-indexed arcs $(i,j,\ell) \in \P$.
For each $(i,j,\ell) \in \P$ with $\ell \ge 2$, it depends on which chain $d \in \D_{\chainlen}$ is considered among those involving arc $(i,j)$ at position $\ell$, and in particular its respective subchain $d^{\to j}$.
For that reason, we cannot unambiguously set the variable $\zeta_{ij,\ell}$ to 0 or 1.
\Cref{fig:picef-ambiguity} illustrates this issue.

\begin{figure}[htbp]
    \centering
    \includegraphics[scale=1.2]{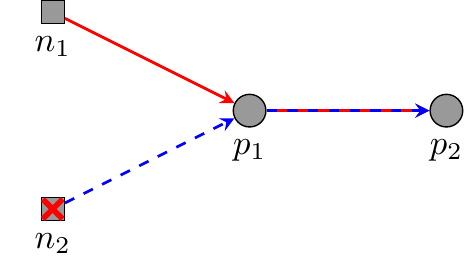}
    \caption{(colored print) Example of ambiguity with a position-indexed chain-edge variable.
    There exist two chains using $(p_{1}, p_{2})$ at position $2$. If the chain $\{n_{1}, p_{1}, p_{2}\}$ is considered, it is part of a feasible chain, whereas this is not the case for the chain $\{n_{2}, p_{1}, p_{2}\}$ due to $n_{2}$ being attacked.\label{fig:picef-ambiguity}}
\end{figure}

For the chains, we can circumvent this issue by instead introducing chain-indexed edge variables $\zeta_{ij}^d$ (notice that the chain index allows us to drop the position indices).
As we only require these variables for chains present in the recourse solutions, we do not need to enumerate all chains, preserving this benefit of PICEF over the CC formulation in terms of formulation compactness.
We define the cycle weight $w_{c}(\vecx')$ for each $c \in \C_{\cyclen}$ as
\begin{equation}\label{eq:s_FR:PICEF:cycle_weights}
    w_{c}(\vecx') = |P(\vecx') \cap V(c)|,
\end{equation}
and the arc weight $w_{ij}(\vecx')$ for each $(i,j) \in A$ as
\begin{equation}\label{eq:s_FR:PICEF:arc_weights}
    w_{ij}(\vecx') = |P(\vecx') \cap \set{j}|,
\end{equation}
The subproblem $s_{\fr}(\vecx'; \X)$ faced by the attacker can then 
be formulated as a MILP as

{\scriptsize
    \begin{mini!}<b>
        {\vecz, \veczeta, \vecu, Z}{Z\label{obj:s_FR:PICEF}}{\label{eq:s_FR:PICEF}}{s_{\fr}(\vecx'; \X)\define}
        \addConstraint{Z}{\ge \sum_{c \in \C_{\cyclen}(S)} w_c(\vecx') z_c + \sum_{d \in \D_{\chainlen}(S)} \sum_{(i,j) \in A(d)} w_{ij}(\vecx') \zeta_{ij}^{d},\label{con:s_FR:PICEF:interdiction_cuts}}{\qquad S \in \X}
        \addConstraint{z_c}{\ge 1 - \sum_{j \in V(c)} u_j,\label{con:s_FR:PICEF:lb_cycle}}{\qquad c \in \C_\cyclen}
        \addConstraint{\zeta_{ij}^d}{\ge 1 - \sum_{k \in V(d^{\to j})} u_k,\label{con:s_FR:PICEF:lb_posarc}}{\qquad d \in \D_{\chainlen}: (i,j) \in A(d)}
        \addConstraint{\vecu}{\in \U\label{con:s_FR:PICEF:attack_budget}}
        \addConstraint{\vecz}{\ge 0\label{con:s_FR:PICEF:z_variables}}
        \addConstraint{\veczeta = (\veczeta^d)_{d \in \D_{\chainlen}}}{\ge 0,\label{con:s_FR:PICEF:zeta_variables}}{\qquad d \in \D_{\chainlen}}
        \addConstraint{Z}{\ge 0.\label{con:s_FR:PICEF:dummy_variable}}
    \end{mini!}
}
Here, Constraints~\eqref{con:s_FR:PICEF:interdiction_cuts} again take the role of implementing the interdiction constraints~\eqref{con:subproblem_general:interdiction_cuts}.
The system of linear inequalities $C_{\fr}(\vecz, \veczeta) \ge d_{\fr}(\vecx', \vecu)$ is given by Constraints~\eqref{con:s_FR:PICEF:lb_cycle} to~\eqref{con:s_FR:PICEF:lb_posarc}.
Notice that for each fixed $\vecu \in \U$, Constraints~\eqref{con:s_FR:PICEF:lb_posarc} allows us to set the value of $\zeta_{ij}^d = 1$ if the subchain $d^{\to j}$ contains no interdicted vertices, and $\zeta_{ij}^d = 0$ otherwise.
Then, for each feasible kidney exchange solution $S \in \X$, the right hand side of Constraints~\eqref{con:s_FR:PICEF:interdiction_cuts} equals the weight of the cycles and (partial) chains containing no interdicted vertices with respect to the variables $\vecu$.

Below, we also provide the modified formulation $s_{\fse}(\vecx'; \X)$ with respect to the $\fse$ policy.
We provide the entire model, as it is quite different from the formulation $s_{\fr}(\vecx'; \X)$ for the full recourse setting.
We introduce variables $\vect = (t_j)_{j \in P \cup N}$ indicating whether or not vertex $j$ is involved in an enforced cycle or (partial) chain $e \in \enf(\vecx')$ given attack $\vecu \in \U$.
Hence, if $t_j = 1$,~\ie~$j$ is already involved in an enforced exchange, any other (non-enforced) exchange involving $j$ cannot be selected for the recourse solution.

{\scriptsize
  \begin{mini!}<b>
    {\vecz, \veczeta, \vect, \vecu, Z}{Z\label{obj:s_FSE:PICEF}}{\label{eq:s_FSE:PICEF}}{s_{\fse}(\vecx'; \X)\define}
    \addConstraint{Z}{\ge \sum_{c \in \C_{\cyclen}(S)} w_c(\vecx') z_c + \sum_{d \in \D_{\chainlen}(S)} \sum_{(i,j) \in A(d)} w_{ij}(\vecx')\zeta_{ij}^{d},\label{con:s_FSE:PICEF:interdiction_cuts}}{\quad S \in \X}
    \addConstraint{z_c}{\ge 1 - \sum_{j \in V(c)} u_j,\label{con:s_FSE:PICEF:lb_initial_cycle}}{\quad c \in \enf(\vecx)}
    \addConstraint{z_c}{\ge 1 - \sum_{j \in V(c)} (u_j + t_j),\label{con:s_FSE:PICEF:lb_noninitial_cycle}}{\quad c \in \C_\cyclen \setminus \enf(\vecx)}
    \addConstraint{z_c}{\le 1 - u_j,\label{con:s_FSE:PICEF:ub_initial_cycle}}{\quad c \in \enf(\vecx): j \in V(c)}
    \addConstraint{\zeta_{ij}^d}{\ge 1 - \sum_{k \in V(d^{\to j})} u_k,\label{con:s_FSE:PICEF:lb_initial_arc}}{\quad d \in \enf(\vecxi): (i,j) \in A(d)}
    \addConstraint{\zeta_{ij}^d}{\ge 1 - \sum_{k \in V(d^{\to j})} \left(u_k + t_k\right),\label{con:s_FSE:PICEF:lb_noninitial_arc}}{\quad d \in \D_\chainlen \setminus \enf(\vecxi): (i,j) \in A(d)}
    \addConstraint{\zeta_{ij}^d}{\le 1 - u_k,\label{con:s_FSE:PICEF:ub_initial_arc}}{\quad d \in \enf(\vecxi): (i,j) \in A(d), k \in V(d^{\to j})}
    \addConstraint{\sum_{j \in V(c)} t_j}{= |V(c)|z_c,\label{con:s_FSE:PICEF:covered_pair_cycle}}{\quad c \in \enf(\vecx)}
    \addConstraint{t_j}{= \zeta_{ij}^d,\label{con:s_FSE:PICEF:covered_pair_posarc}}{\quad d \in \enf(\vecxi): (i,j) \in A(d)}
    \addConstraint{t_i}{= \zeta_{ij}^d,\label{con:s_FSE:PICEF:covered_NDD}}{\quad d \in \enf(\vecxi): (i,j) \in A(d), i \in N}
    \addConstraint{\vecu}{\in \U\label{con:s_FSE:PICEF:attack_budget}}
    \addConstraint{\vecz,\vect}{\ge 0\label{con:s_FSE:PICEF:zt_variables}}
    \addConstraint{\veczeta = (\veczeta^{d})_{d \in \D_{\chainlen}}}{\ge 0\label{con:s_FSE:PICEF:zeta_variables}}
    \addConstraint{Z}{\ge 0.\label{con:s_FSE:PICEF:dummy_variable}}
  \end{mini!}   
}

Again, Constraints~\eqref{con:s_FSE:PICEF:interdiction_cuts} take the role of implementing~\eqref{con:subproblem_general:interdiction_cuts}.
The remaining constraints,~\ie Constraints~\eqref{con:s_FSE:PICEF:lb_initial_cycle}-\eqref{con:s_FSE:PICEF:covered_NDD} model the system of linear inequalities $C_{\fse}(\vecz, \veczeta, \vect) \ge d_{\fse}(\vecx', \vecu)$.
Here, for each vertex $j \in P \cup N$, the variable $t_j$ is correctly fixed through Constraints~\eqref{con:s_FSE:PICEF:covered_pair_cycle},~\eqref{con:s_FSE:PICEF:covered_pair_posarc} and~\eqref{con:s_FSE:PICEF:covered_NDD}.
These values are used in Constraints~\eqref{con:s_FSE:PICEF:lb_noninitial_cycle} and~\eqref{con:s_FSE:PICEF:lb_noninitial_arc} to ensure that $\vecz$- and $\veczeta$-variables of non-enforceable exchanges $e \notin \enf(\vecx')$ are fixed to zero whenever any exchange is enforced that has at least one common vertex with $e$.
For cycles and arcs not selected for the initial solution, these variable fixings are realized only if no enforced exchange it has nonempty intersection with was successful.

\subsubsection{PICEF: Finding optimal recourse solutions}\label{subsubsec:PICEF:recourse_problem}
Finally, we provide a PICEF-based MILP formulation for the separation problem $R_{\fr}(\vecx', \vecu)$ of finding optimal recourse solutions given initial solution $\vecx' \in \X$ and attack $\vecu \in \U$.
Two vectors of binary variables are used to describe cycles and position-indexed arcs selected for the recourse solution,~\ie~cycle variables~$\vecy^{\vecu} = (y_{c}^{\vecu})_{c \in \C_{\cyclen}}$ and position-indexed arc variables~$\vecpsi^{\vecu} = (\psi_{ij,\ell}^{\vecu})_{(i,j,\ell) \in \P}$.
We define the cycle weights~$w_c(\vecx', \vecu) = w_c(\vecx')$ for each cycle $c \in \C_{\cyclen}$ and arc weights $w_{ij}(\vecx', \vecu) = w_{ij}(\vecx')$ for each arc~$(i,j)\in A$.
The formulation of the separation problem $R_{\fr}(\vecx', \vecu)$ is given as follows.
{
    \begin{maxi!}<b>
        {\vecy^{\vecu}, \vecpsi^{\vecu}}{\sum_{c \in \C_{\cyclen}} w_c(\vecx') y_c^{\vecu} + \sum_{(i,j,\ell) \in \P} w_{ij}(\vecx') \psi_{ij,\ell}\label{obj:R_FR:PICEF}}{\label{eq:R_FR:PICEF}}{R_{\fr}(\vecx', \vecu)\define}
        \addConstraint{\sum_{c \in \C_{\cyclen}^j} y_c^{\vecu} + \sum_{(i,j,\ell) \in \P^j} \psi_{ij, \ell}^{\vecu}}{\le 1 - u_j,\label{con:R_FR:PICEF:packing_pairs}}{\qquad j \in P}
        \addConstraint{\sum_{(j,i,1) \in \P^j} \psi_{ji, 1}^{\vecu}}{\le 1 - u_j,\label{con:R_FR:PICEF:packing_NDDs}}{\qquad j \in N}
        \addConstraint{\sum_{(j,i,\ell+1) \in \P^j} \psi_{ji, \ell+1}^{\vecu} - \sum_{(i,j,\ell) \in \P^j} \psi_{ij, \ell}^{\vecu}}{\le 0,\label{con:R_FR:PICEF:precedence}}{\qquad j \in P, \ell \in \L \setminus \set{\chainlen}}
        \addConstraint{\vecy^{\vecu}}{\in \bins^{\C_{\cyclen}}}
        \addConstraint{\vecpsi^{\vecu}}{\in \bins^{\P}}
    \end{maxi!}
}

Notice that Problem~\eqref{eq:R_FR:PICEF} is again a weighted version of the basic PICEF-based kidney exchange model on the subgraph induced by $V \setminus V(\vecx', \vecu)$.
In order to obtain the modified version of the PICEF-based separation problem, also allowing non-feasible recourse solutions, we also need to consider which cycles and (partial) chains of solutions are feasible. 
In the case of cycles, this can be realized again through a modification of the cycle weights,~\ie
\[
    w'_{c}(\vecx', \vecu) \define
                                \begin{cases}
                                    w_{c}(\vecx')|V| + 1, & \text{ if } c \in \C_{\cyclen}(\vecu)\\
                                                       1, & \text{ otherwise.}
                                \end{cases}
\]
Again, as position-indexed arcs only hold local information of a chain, a similar modification does not work for position-indexed arc variables, as it might allow for infeasible chains of large weight.
To overcome the issue, we introduce a vector $\veceta^{\vecu} = (\eta_{ij,\ell}^{\vecu})_{(i,j,\ell)\in \P}$ of binary variables inducing a disjoint set of chains on the entire compatibility graph $G$,~\ie~not regarding any attacks.
We again consider $\vecpsi$-variables, now to model the enforced subchains under attack $\vecu$, with respect to the chains induced by $\veceta$. 
We can now disregard position indices, since these are already incorporated in the $\veceta$-variables. 
The arc weights $w'_{ij}(\vecx', \vecu)$ for all arcs $(i,j) \in A$ are then defined as
\[
    w'_{ij}(\vecx', \vecu) 
        \define
        \begin{cases}
            w_{ij}(\vecx')|V| + 1, & \text{ if } i \in N\\
            w_{ij}(\vecx')|V|, & \text{ otherwise.}
        \end{cases}
\]
We consider a parameter $\varepsilon > 0$ that takes an infinitesimally small value to model the tiebreaker.
Although the optimal recourse solution is given by the $\vecy$- and $\vecpsi$-variables, the terms for the chains in the cuts we generate are induced by the $\veceta$-variables. 
The idea here is to find a feasible kidney exchange solution with the largest number of arcs such that its enforced cycles and (sub)chains form an optimal recourse solution.
Given the choices of the objective coefficients for the $\vecy$- and $\vecpsi$-variables, we can fix $\epsilon$ to $1$.
This way, any optimal solution $\vecpsi$ to the modified recourse problem still induces an optimal solution to the original recourse problem if we restrict ourselves to the set of non-interdicted exchanges.

The modified recourse problem can then be written as
{\small
    \begin{maxi!}<b>
        {\vecy^{\vecu}, \vecpsi^{\vecu}, \veceta^{\vecu}}{\sum_{c \in \C_{\cyclen}} w'_c(\vecx', \vecu) y_c^{\vecu} + \sum_{(i,j,\ell) \in \P} w'_{ij}(\vecx', \vecu)\psi_{ij}^{\vecu} + \varepsilon \sum_{(i,j) \in A} \eta_{ij,\ell}^{\vecu}\label{obj:modif_recourse:PICEF}}{\label{eq:modif_recourse:PICEF}}{R'_{\fr}(\vecx', \vecu)\define}
        \addConstraint{\sum_{c \in \C_{\cyclen}^j} y_c^{\vecu} + \sum_{(i,j,\ell) \in \P} \eta_{ij,\ell}^{\vecu}}{\le 1,\label{con:modif_recourse:PICEF:packing_pairs}}{\qquad j \in P}
        \addConstraint{\sum_{(j,i) \in A} \eta_{ji, 1}^{\vecu}}{\le 1,\label{con:modif_recourse:PICEF:packing_NDDs}}{\qquad j \in N}
        \addConstraint{\sum_{(j,i,\ell+1) \in \P} \eta_{ji,\ell + 1}^{\vecu} - \sum_{(i,j,\ell)\in \P} \eta_{ij,\ell}^{\vecu}}{\le 0, \label{con:modif_recourse:PICEF:precedence}}{\qquad j \in P, \ell \in \L \setminus \set{\chainlen}}
        \addConstraint{\psi_{ij}^{\vecu} - \sum_{\ell \in \L(i,j)} \eta_{ij,\ell}^{\vecu}}{\le 0,\label{con:modif_recourse:PICEF:psi_ub}}{\qquad (i,j) \in A}
        \addConstraint{\sum_{(i,j) \in A} \psi_{ij}^{\vecu}}{\le 1 - u_{j},\label{con:modif_recourse:PICEF:psi_ub_pair}}{\qquad j \in P}
        \addConstraint{\sum_{(j,i) \in A} \psi_{ji}^{\vecu}}{\le 1 - u_{j},\label{con:modif_recourse:PICEF:psi_ub_NDD}}{\qquad j \in N}
        \addConstraint{\sum_{(j,i)\in A} \psi_{ji}^{\vecu} - \sum_{(i,j) \in A} \psi_{ij}^{\vecu}}{\le 0,\label{con:modif_recourse:PICEF:psi_precedence}}{\qquad j \in P}
        \addConstraint{\vecy^{\vecu}}{\in \bins^{\C_{\cyclen}}}
        \addConstraint{\vecpsi^{\vecu}}{\in \bins^{A}}
        \addConstraint{\veceta^{\vecu}}{\in \bins^{\P}}
    \end{maxi!}
}

In this model, Constraints~\eqref{con:modif_recourse:PICEF:packing_pairs}-\eqref{con:modif_recourse:PICEF:precedence} requires the solution $(\vecy^{\vecu}, \veceta^{\vecu})$ to induce a feasible kidney exchange solution on $G$.
The remaining constraints are used to restrict the values that each arc variable $\psi_{ij}^{\vecu}$ can take,~\ie~the variables that correspond to (sub)chains that are feasible for the recourse solution given $\vecu$.
Constraints~\eqref{con:modif_recourse:PICEF:psi_ub} impose that~$\psi_{ij}^{\vecu}$ can be set to 1 only if there exists some initial chain $d$ with $(i,j) \in A(d)$.
Furthermore, Constraints~\eqref{con:modif_recourse:PICEF:psi_ub_pair} and~\eqref{con:modif_recourse:PICEF:psi_ub_NDD} require that $\psi_{ij}^{\vecu} = 0$ if either $i$ or $j$ is interdicted.
Constraints~\eqref{con:modif_recourse:PICEF:psi_precedence} are precedence constraints,~\ie~the arcs $(i,j) \in A$ with $\psi_{ij}^{\vecu}$ indeed induce a packing of non-interdicted subchains.


Furthermore, we can obtain the formulations for $R_{\fse}(\vecx', \vecu)$ and $R'_{\fse}(\vecx', \vecu)$ from Formulation~\eqref{eq:R_FR:PICEF} and Formulation~\eqref{eq:modif_recourse:PICEF}, respectively, by fixing variables to 1 as follows:
\begin{align*}
    y_{c}^{\vecu} &= 1, & & \qquad c \in \C_{\cyclen}(\vecu): x_{c} = 1\\
    \psi_{ij}^{\vecu} &= 1, & & \qquad (i,j) \in A: \exists d \in \D_{\chainlen}(\vecxi) \text{ s.t. } (i,j) \in A(d),\ d^{\to j} \in \D_{\chainlen}(\vecu)
\end{align*}
In the next section, we will show computational results based on the algorithmic framework and formulations laid out before.

\subsection{Strength of formulations}
Having described the CC and PICEF based formulations, we now compare their strength, in particular for the second-level problems $s_{\Pi}(\vecx; \X)$. Recall that we solve these through a cut generation procedure. Most importantly, we solve the relaxations $s_{\Pi}(\vecx; \bar{\X})$, where a limited set $\bar{\X} \subset \X$ of kidney exchange solutions are considered. We show that in this case, the PICEF-based formulation is stronger. We use $Z^*_{\fr,PICEF}(\vecx', \bar{\X})$ to denote the value of the optimal solution to~\eqref{eq:s_FR:PICEF} and $Z^*_{\fr,CC}(\vecx, \bar{\X})$ to denote the optimal solution value of~\eqref{eq:s_FR:CC}.

\begin{theorem}\label{thm:FR:strength_comparison}
Given a graph $G$, set of exchanges $\E$, initial solution $\vecx$ and a set of feasible kidney exchange solutions $\bar{\X} \subset \X$, we have that $Z^*_{\fr,PICEF}(\vecx, \bar{\X}) \ge Z^*_{\fr,CC}(\vecx, \bar{\X})$. Furthermore, there exist instances for which $\bar{\X}$, $Z^*_{\fr,PICEF}(\vecx, \bar{\X}) > Z^*_{\fr,CC}(\vecx, \bar{\X})$.
\end{theorem}
\begin{proof}
To prove $Z^*_{\fr,PICEF}(\vecx, \bar{\X}) \ge Z^*_{\fr,CC}(\vecx, \bar{\X})$, we compare the right-hand sides of Constraints~\eqref{con:s_FR:CC:interdiction_cuts} and Constraints~\eqref{con:s_FR:PICEF:interdiction_cuts}, respectively. We claim that for a given recourse solution $S \in \X$ and any fixed attack \vecu, the right-hand side of the PICEF variant is at least as large as the right-hand side of the CC variant, \ie
\begin{align}
    \sum_{c \in \C(S)} w_c(\vecx') z_c + \sum_{d \in \D(S)} \sum_{(i,j) \in A(d)} w_{ij}(\vecx')\zeta_{ij}^{d} \ge \sum_{c \in \C(S)} w_c(\vecx') z_c + \sum_{d \in \D(S)} w_d(\vecx') z_d.
\end{align}
Since the cycle variables are handled identically in both formulations, we must only prove
\begin{align}
    \sum_{d \in \D(S)} \sum_{(i,j) \in A(d)} w_{ij}(\vecx')\zeta_{ij}^{d} \ge\sum_{d \in \D(S)} w_d(\vecx') z_d \label{eq: CC-PICEF Comparison}.
\end{align}
If for a given chain $d \in \D(S)$, no vertices $v \in V(d)$ are attacked, we have $z_d = 1$ (due to Constraints~\eqref{con:s_FR:CC:lb1_z}), as well as $\zeta_{ij}^{d} = 1$ for all~$(i,j) \in A(d)$ (due to Constraints~\eqref{con:s_FR:PICEF:lb_posarc}). 
By definition, $\sum_{(i,j) \in A(d)} w_{ij}(\vecx') = w_{d}(\vecx')$, thus chains that are not attacked contribute the same value to the right-hand sides in both formulations. 
If the chain is attacked at any point, then~${z_{d} = 0}$, and the chain does not contribute to the right-hand side of Constraints~\eqref{con:s_FR:CC:interdiction_cuts}. 
As~$w_{ij}(\vecx') \ge 0$ and $\zeta_{ij}^{d} \ge 0$ for all~$(i,j) \in A(d)$, Inequality~\eqref{eq: CC-PICEF Comparison} follows.

We now provide an instance where for $\bar{\X}$, $Z^*_{\fr,PICEF}(\vecx, \bar{\X}) > Z^*_{\fr,CC}(\vecx, \bar{\X})$. Consider the graph $G = (V,A)$ given in~\Cref{fig:cut-strength}, and attack budget $B = 1$.
The first stage solution consists only of the chain $(1,2,3,4)$,~\ie~$\vecx$ for CC and $\vecx'$ for PICEF. As a result, the value of every exchange in the recourse solutions is equal to the number of pairs it contains,~\ie~$w_{c}(\vecx) = w_{c}(\vecx') = |V(c)|$ for each cycle $c$, $w_{d}(\vecx) = |V(d)|$ for each chain $d$ and $w_{ij}(\vecx') = 1$ for each arc $(i,j) \in A$.
Let $\bar{\X} = \set{S_{1}, S_{2}}$ where $S_{1} = {(1,2,3,4)}$ and $S_{2} = {(3,4)}$.

For this instance, $Z^*_{\fr,CC}(\vecx, \bar{\X}) = 0$, which can be achieved by attacking pair $3$ or pair $4$. However, in the PICEF formulation, attacking pair $3$ or $4$ forces $\zeta_{1,2}^d = 1$, and $Z^*_{\fr,PICEF}(\vecx, \bar{\X}) \ge 1$ due to recourse solution $S_{1}$. Not attacking either of these pairs forces $z_{c} = 1$ for the cycle $(3,4)$, and $Z^*_{\fr,PICEF}(\vecx, \bar{\X}) \ge 2$. Consequently, for this instance $Z^*_{\fr,PICEF}(\vecx, \bar{\X}) > Z^*_{\fr,CC}(\vecx, \bar{\X})$.
\end{proof}

\begin{figure}[htbp]
    \centering
    \includegraphics[scale=1.2]{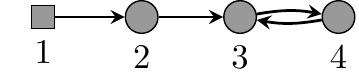}
    \caption{Example compatibility graph illustrating the strength of PICEF-based interdiction cuts compared to CC-based ones.\label{fig:cut-strength}}
\end{figure}

Notice that for single-stage kidney exchange programs, it was shown by~\cite{Dickerson2016} that the LP relaxation of the cycle-chain formulation is strictly stronger than the LP relaxation for PICEF.
However, our result implies that the PICEF-based formulation leads to stronger relaxations for the attacker-defender subproblem than the CC-based formulation.
A similar result holds for the formulations with respect to the FSE policy.

\begin{theorem}\label{thm:FSE:strength_comparison}
Given a graph $G$, set of exchanges $\E$, initial solution $\vecx$ and a set of feasible kidney exchange solutions $\bar{\X}$, $Z^*_{\fse,PICEF}(\vecx, \bar{\X}) \ge Z^*_{\fse,CC}(\vecx, \bar{\X})$. Furthermore, there exist instances for which~$\bar{\X}$, $Z^*_{\fse,PICEF}(\vecx, \bar{\X}) > Z^*_{\fse,CC}(\vecx, \bar{\X})$.
\end{theorem}

We omit the proof, as the argument is analogous to the Full Recourse case. The only complication are the fixed exchanges, which also allow $z_e$ or $\zeta_{ij}^d$ variables of non-attacked exchanges to be set to 0. However, these have identical effects to the attacks allowing these variables to be set to 0, and so the same arguments from the FR case hold. 




%% file: ManuscriptMay2022/ComputationalResults.tex
\section{Computational Results}\label{sec:computational_results}

The aim of this section is to evaluate the performance of our cutting plane approaches, which we will refer to as cut-CC and cut-PICEF for the implementation based on the cycle-chain and the position-indexed chain-edge formulations, respectively.
These cutting plane approaches are compared with the branch-and-bound approach by~\cite{Glorie20} on practical instances for the Full Recourse setting.
In this section, we investigate the following questions:
\medskip

{\noindent\bfseries Full Recourse:}
\begin{questions}
    \item \label{Q1} Are cut-CC and/or cut-PICEF computationally more efficient than
      the branch-and-bound approach, both with respect to computation times and the number of solved instances?
    \item \label{Q2} Does cut-PICEF outperform cut-CC or vice versa?
\end{questions}

{\noindent\bfseries FSE:}
\begin{questions}
    \item \label{Q3} Are cut-CC and/or cut-PICEF computationally tractable methods for solving the robust kidney exchange problem whenever recourse is allowed according to the FSE policy?
    \item \label{Q4} Does cut-PICEF outperform cut-CC or vice versa?
    \item \label{Q5} Are the robust objective values under the FSE setting comparable to the values for Full Recourse, \ie how large is the loss in the guaranteed number of transplants by imposing FSE constraints on recourse solutions?
\end{questions}

In the following, we first discuss our computational setup and used test sets
in \Cref{sec:comp-setup}.
Afterward, we present a comparison of the cutting plane algorithms against the branch-and-bound scheme in \Cref{sec:method-comparison-FR}.
Moreover, we explore the tractability of instances under the FSE policy by cut-CC and cut-PICEF in \Cref{sec:results-FSE}.
We finish by discussing the impact of our cut lifting techniques in~\Cref{sec:impact-lifting}.

\subsection{Test sets and computational setup}
\label{sec:comp-setup}

In our experiments, we have used the same test set as~\cite{Glorie20}, which is publicly available\footnote{\url{https://rdm.inesctec.pt/dataset/ii-2020-001}}. 
This test set consists of ninety graphs of which thirty contain~20, 50, and~100 vertices, respectively.
More characteristics of these graphs can be found in~\cite{Glorie20}.
From these graphs, we generate different instances of the full recourse robust kidney exchange problem by bounding the length~\cyclen
and~\chainlen of the considered cycles and chains, respectively.

To allow for a fair comparison, we have implemented both our new methods and the reference method in \solver{C/C++} using the mixed-integer programming framework \solver{SCIP~8.0.2} with \solver{SoPlex~6.0.2} as LP solver\footnote{Our implementation is publicly available at \url{https://github.com/DannyBlom/BendersRobustKEP/commit/93fce3fb}}, see~\cite{SCIP8.0}.
All computations were run on a Linux cluster with Intel Xeon Platinum 8260 \SI{2.4}{\GHz} processors.
The code was executed using a single thread.
The time limit of all computations is \SI{3600}{\second} per instance.

For cut-CC and the branch-and-bound method, we compute for each instance all cycles of length at most~\cyclen and chains of length at most~\chainlen present in its compatibility graph.
For cut-PICEF, we again compute the cycles and generate all position-indexed arcs that could potentially be part of a feasible chain,~\ie~each position-indexed arc $(i,j,\ell) \in \P$ for which there exists a path from an NDD to $i$ of length $\ell - 1$ not containing $j$.

In all methods, all cycle variables (and, if applicable, all chain or all position-indexed arc variables) are added to the initial model.
Column generation techniques exist for these models (see~\cite{Glorie2014} and~\cite{plaut2016fast} among others), but these are not efficient for the smaller chain and cycle sizes we study (\cite{Dickerson2016}).
In the branch-and-bound method, there are some degrees of freedom in selecting the branching strategy.
To be able to produce results as consistent as possible with the ones described in~\cite{Glorie20}, we have implemented the same strategies as discussed there.
A summary of these strategies can be found in \Cref{sec:appendix-bandb}.
When comparing with the branch-and-bound scheme of~\cite{Glorie20}, we consider the most basic implementation of cut-CC and cut-PICEF, without lifting.  
We also conducted experiments with the cut lifting methods described in~\Cref{subsubsec:CC:recourse_problem} and~\Cref{subsubsec:PICEF:recourse_problem} enabled, which we will discuss in~\Cref{sec:impact-lifting} in more detail.

\subsection{Comparison of solution methods for Full Recourse}
\label{sec:method-comparison-FR}
To empirically find answers to Questions~\ref{Q1} and~\ref{Q2}, we have conducted experiments with varying maximum cycle and chain lengths $\cyclen$ and $\chainlen$ for the compatibility graphs described in \Cref{sec:comp-setup}.
The maximum cycle length in our experiments takes value~$\cyclen \in \{3,4\}$; the maximum chain length is \mbox{$\chainlen \in \{2,3,4\}$}.
We allow to attack between~1 and~4 vertices per instance in our experiments,~\ie~the attack budget $B \in \set{1,2,3,4}$; attacks of arcs are not taken into account, which is consistent with the experiments by \cite{Glorie20}.

In comparison to the experiments described in \cite{Glorie20}, we consider the impact of varying cycle lengths, while they only considered~$\cyclen = 3$. We do include $\cyclen = 4$, since multiple European countries allow cycles of this size,~\eg~the Czech Republic and The Netherlands, see~\cite{biro2020}.
In order to be concise, we only report the tables for~$\cyclen = 3$ in the main text, as the computational results for~$\cyclen = 4$ depict a similar comparison between methods.
The results for~$\cyclen=4$ are presented in~\Cref{sec:appendix-compresults-K=4}.
The different chain lengths in our experiments are the same as in \cite{Glorie20} (note, however, that they measure the length of a chain by its number of vertices instead of arcs, i.e., the numerical values are shifted by~1 compared to their article).
\Cref{tab:cyc3-policy1} reports on our experiments for~$\cyclen = 3$ and different $\chainlen$ and $B$ with respect to the Full Recourse policy.

\input{./Tables/MethodComparison/comparison-cyc3-policy1.tex}

Before we proceed with a discussion of our experiments, we describe the
common structure of the presented tables.
All mean values are given in arithmetic mean with the only exception
being the average solving time, which is measured in shifted geometric mean
\[
  \prod_{i = 1}^n (t_i + s)^{1/n} - s,
\]
where we use a shift of~$s = 10$.
The reason for the latter is to reduce the impact of outliers and instances
with a very small running time.

Column ``\#vertices'' reports on the number of vertices in the tested graphs,
whereas column ``\#opt'' provides the number of instances
solved within the time limit.
The mean solution time per instance (in seconds) is given in ``time total''.
For cut-CC and cut-PICEF, the columns ``stage~2'' and ``stage~3'' present the average proportion of running time spent in solving attacker-defender subproblems and recourse problems, respectively.
For the branch-and-bound method, ``stage 2'' represents the problem of solving the attacker-defender subproblem with the branch-and-bound algorithm.
Note that all reported times are in seconds.
Column ``\#att.'' reports on the mean number of generated attacks per
solved instance; if no instance from a test set could be solved within the
time limit, the corresponding entry is ``---''.
Finally, column ``\#B\&B-nodes'' gives the mean number of
branch-and-bound nodes per solved instance for the
branch-and-bound method.
For cut-CC and cut-PICEF, the column ``\#sub'' shows the mean number of recourse problems that needed to be solved per instance solved within the time limit.

In our experiments for maximum cycle length~$\cyclen = 3$, we observe that all methods can solve all instances based on graphs with~20 vertices within a matter of seconds.
For graphs with~50 vertices, however, there exist instances that the branch-and-bound method fails to solve within the time limit, for~$\chainlen = 2$ and attack budget~$B = 4$; for larger maximum chain lengths,~\ie~$\chainlen \in \set{3,4}$, there already exist instances that could not solved for which the attack budget~$B$ equals~3 or~2, respectively.
In contrast to this, cut-CC could solve all instances with up to~50 vertices within the time limit whenever~$\chainlen \in \set{2,3}$, except for one instance.
For~$\chainlen = 4$, cut-CC can still solve over \SI{96}{\percent} of these instances within one hour.
For cut-PICEF, there are slightly more (4) instances with at most 50 vertices and $\chainlen \in \{2,3\}$ that could not be solved within the time limit.
However, for $\chainlen = 4$, almost \SI{99}{\percent} of these instances could still be solved within the time limit.
Furthermore, the instances that could not be solved were based on $B = 4$ exclusively.

For instances with~100 vertices in the compatibility graph, the branch-and-bound method is hardly able to solve any instance if $B \ge 3$ and~$\chainlen = 2$; 
for increasing maximum chain lengths~$\chainlen$, even instances with $B \in \{1,2\}$ become challenging for the branch-and-bound approach.
Whenever $\chainlen = 2$, the cut-CC approach is superior as it allows solving more instances than cut-PICEF for all possible attack budgets $B$.
The opposite holds whenever~$\chainlen \ge 3$.
The benefit of cut-PICEF over cut-CC and branch-and-bound is most apparent for~$\chainlen = 4$.
In this case, cut-PICEF succeeds in solving over two third of the number of instances with 100 vertices, whereas both cut-CC and branch-and-bound could only solve 4 and 3 out of the 120 instances, respectively.

Based on these experiments, we can answer Question~\ref{Q1} affirmatively as the cutting plane approaches are able to consistently solve more instances than the branch-and-bound method within the time limit.
Specifically, we have accomplished a running time improvement of one order of magnitude for most instance sets.
Furthermore, regarding Question~\ref{Q2}, comparing the two cutting plane methods with each other, there is a transition related to the choice of~$\chainlen$ in which cut-PICEF starts outperforming cut-CC.
Whenever the maximum chain length~$\chainlen$ is small, there exists an overhead in the formulation size for cut-PICEF causing it to be slower than cut-CC, whereas this comparison flips in case~$\chainlen$ increases. 
These conclusions are also supported by our experiments with maximum cycle length~$\cyclen = 4$, as can be seen in~\Cref{tab:cyc4-policy1}.
This behaviour is expected, due to the advantages in model size of PICEF with longer chain lengths. Furthermore, the advantage of PICEF's stronger cuts (Theorems ~\ref{thm:FR:strength_comparison} and~\ref{thm:FSE:strength_comparison}) for the attacker-defender subproblem will grow as the chain length increases.


Observe that this is also reflected in our computational results. For most parameter settings, the average number of subproblems solved per master iteration is higher for cut-CC compared to cut-PICEF, even when $\chainlen = 2$.
However, we also observe substantial variability; on an instance level, there are many examples of cut-CC requiring fewer iterations than cut-PICEF. This is possible as the recourse problems have multiple optima and the solvers for cut-CC and cut-PICEF might determine different recourse solutions,~\ie~leading to different interdiction cuts.

\begin{remark}
  When comparing the number of generated attacks for the branch-and-bound method with the numbers reported by~\cite{Glorie20}, we observe that our numbers are higher in general.
  Although we tried to implement the method as close as possible to the implementation of~\cite{Glorie20}, cf.~\Cref{sec:appendix-bandb}, there are still factors that might negatively impact the number of generated attacks in our implementation.
  Among others, the order of cycles/chains might be different in the two implementations, which might cause to find different optimal solutions in each iteration, and the MIP solver used in our experiments is different to the one used by~\cite{Glorie20}.
  Thus, since these factors impact both the branch-and-bound and cutting plane algorithms equally, our implementation does not give an advantage to either of these methods.
\end{remark}

\input{./Tables/MethodComparison/comparison-cyc3-policy2.tex}

\subsection{FSE: exploring tractability and efficiency}
\label{sec:results-FSE}

As the benefit of the cutting plane methods over branch-and-bound for the Full Recourse setting is evident, we have implemented the Fix Successful Exchanges (FSE) policy only for cut-CC and cut-PICEF.
The results of our computational experiments for FSE, again based on the most basic implementation without cut lifting, can be found in~\Cref{tab:cyc3-policy2}.


Again, we see that all of the instances based on graphs with~20 vertices can be solved within a matter of seconds, both for cut-CC and cut-PICEF.
However, for graphs with at least~50 vertices, both cut-CC and cut-PICEF take much longer to solve instances to optimality compared to the Full Recourse setting, especially for increasing attack budgets~$B$.
In particular, both cut-CC and cut-PICEF fail to solve all instances even if~$\chainlen = 2$ and $B = 1$.
Based on the numbers in~\Cref{tab:cyc3-policy2}, we can make a few observations on why imposing this recourse policy complicates the problem.
First of all, for the parameter settings for which none or only a fraction of the instances can be solved within the time limit, the percentage of time spent in solving the attacker-defender subproblem (stage 2) and the optimal recourse problem (stage 3) is remarkably small.
In other words, most of the time is spent solving the master problem.
Furthermore, compared to the Full Recourse setting, the number of attacks separated by solving attacker-defender subproblems is significantly larger on average, especially for the more challenging instances.
This implies that we end up with much larger master relaxations, which typically take significantly more time to be solved.
Apart from this, the mixed-integer programming formulations for all three stages are also much more complex due to the additional constraints imposed by the FSE policy.

Regarding the comparison of cut-CC and cut-PICEF, we again observe a transition as in the FR setting; for the most challenging instances with~$\chainlen \in \set{2, 3}$, cut-CC is superior to cut-PICEF: for~100 vertices and all different attack budgets $B \in \{1,2,3,4\}$, cut-CC is able to solve more instances within the time limit compared to cut-PICEF, whereas the opposite is true whenever~$\chainlen \geq 4$.
A similar comparison can be made on the average computation times over all solved instances.
The only difference here is that the transition occurs for a larger maximum chain length; cut-PICEF now starts to outperform cut-CC only whenever $\chainlen = 4$, whereas in the Full Recourse setting, this was the case for $\chainlen = 3$ already.

To answer Questions~\ref{Q3} and~\ref{Q4}, we can conclude that for small instances, both cut-CC and cut-PICEF are able to solve the problem within seconds.
Nevertheless, when the kidney exchange pools and/or attack budgets grow in size, both cut-CC and cut-PICEF struggle to solve a significant fraction of the instance sets within the time limit. 
The impact of increasing $\chainlen$ is more gradual for cut-PICEF, whereas for cut-CC, a clear transition is visible between $\chainlen = 2$ and $\chainlen = 3$, leading to a similar performance comparison between cut-CC and cut-PICEF as we observed for Full Recourse.
More sophisticated algorithmic techniques will be necessary to solve instances of real-life mid-size to large kidney exchange programs.

Furthermore, we have investigated the impact of the constraints of the FSE policy on the final master objective value of the instances, which is depicted in~\Cref{tab:fr-vs-fse}.
For this comparison, we can only include the instances that we were able to solve within the time limit for both the Full Recourse and the FSE setting.
We have aggregated our results for each kidney exchange pool size (20, 50, 100), as we believe it is one of the main instance features impacting the objective value in general.

\input{./Tables/FR-vs-FSE.tex}

Observe that for most instances, the objective value,~\ie~a guaranteed number of patients from the initial solution receiving a transplant, is rather similar for both recourse policies.
In particular, one can see that more than \SI{85}{\percent} for the instances with~20 vertices, the objective values of the Full Recourse and FSE settings coincide, whereas the objective values differ by more than one for less than~\SI{0.5}{\percent} of the instances.
Moreover, for any of the instances solved to optimality for both FR and FSE, the maximum difference in the objective values observed is 2.
We conclude that imposing the constraints on the recourse options with respect to fixing successful initial exchanges does not drastically decrease the guaranteed number of initial patients receiving a transplant given all possible attacks.
Consequently, imposing the FSE policy might be a feasible option in practice, if the advantages of reducing the logistical challenges outweighs the small, albeit non-negligible, loss of transplants in the worst case.

\subsection{Evaluation of the impact of lifting}
\label{sec:impact-lifting}
In the experiments discussed in~\Cref{sec:computational_results}, we have deliberately not considered cut lifting procedures to allow for a fair comparison of cut-CC and cut-PICEF with the branch-and-bound scheme.
As mentioned above, we have also conducted experiments to measure the
impact on the cutting plane approaches of the lifting schemes proposed in Sections~\ref{subsubsec:CC:recourse_problem} and~\ref{subsubsec:PICEF:recourse_problem}.
\Cref{tab:lifting-cyc3} shows the computational results for both cut-CC and cut-PICEF, and for Full Recourse and FSE, respectively.
We have restricted our results in this table to instances with 100 vertices, to focus on the hardest instances where algorithmic improvements are most necessary.
In the columns below ``with lifting'', we present the results where the modified recourse problems are used to separate interdiction cuts.
In the columns below ``without lifting'', we review the results of cut-CC and cut-PICEF from~\Cref{tab:cyc3-policy1} and~\Cref{tab:cyc3-policy2} for easier comparison.
Again, the attack budgets $B$ from 1 to 4 are used.

\input{./Tables/LiftingTables/combined-lifting-table-cyc3.tex}

For Full Recourse, both cut-CC and cut-PICEF benefit from cut lifting, as it has a positive impact on both the number of instances solved and the total computation time.
While the number of generated attacks per optimally solved instance is roughly the same, the number of subproblems that need to be solved is significantly smaller whenever cut lifting is enabled.
As a result, proving the optimality of solutions to the attacker-defender subproblem takes fewer iterations, which suggests cut lifting is of benefit for reducing computation times.

Interestingly, this only leads to a reduction in running time and an increased number of solved instances for Full Recourse.
Especially when cut-PICEF is employed, and for instances with larger maximum chain lengths, the number of solved instances increases substantially when cut lifting is enabled compared to the basic variant of cut-PICEF, as well as a computation time reduction of over \SI{50}{\percent} for the instances with attack budget~$B = 4$. 

For FSE, no significant benefit is visible when cut lifting is enabled or vice versa.
For both cut-CC and cut-PICEF, enabling cut lifting slightly improves the average computation times for most parameter settings, but it cannot be regarded as an improvement under all circumstances.
For chain length $\chainlen = 2$ and $B = 1$, enabling cut lifting decreases the number of instances solved within the time limit, although for larger $\chainlen$ and $B$, the opposite holds.
Moreover,~\Cref{tab:combined-aggrlifting-cyc3-policy1} shows an aggregated overview of the results of cut lifting.
As the positive impact of cut lifting was most pronounced for Full Recourse, we restrict ourselves to the results of this recourse policy.

\input{./Tables/AggregatedLifting/combined_aggr_lifting_table-cyc3-policy1.tex}

We conclude that under almost all parameter settings, cut lifting is a valuable feature for both cut-CC and cut-PICEF to obtain a performant algorithm for the Full Recourse policy, especially for less restrictive maximum lengths of cycles and chains in a kidney exchange solution.
This is consistent with similar experiments conducted in literature, such as the modifications of interdiction cuts as reported by~\cite{Fischetti2019}.


%% file: Tables/MethodComparison/comparison-cyc3-policy1.tex
\afterpage{%
  \clearpage%
  \begin{landscape}
    \begin{table}[htbp]
      \begin{tiny}
        
        \label{tab:cyc3-policy1}
        \begin{tabular*}{\linewidth}{@{}l@{\;\;\extracolsep{\fill}}rrrrrrrrrrrrrrrrr@{}}\toprule
           & \multicolumn{5}{c}{Branch-and-Bound} & \multicolumn{6}{c}{cut-CC} & \multicolumn{6}{c}{cut-PICEF}\\
          \cmidrule{2-6} \cmidrule{7-12} \cmidrule{13-18}
          & & \multicolumn{2}{c}{time} & & & & \multicolumn{3}{c}{time} & & & & \multicolumn{3}{c}{time} & & \\
          \cmidrule{3-4} \cmidrule{8-10} \cmidrule{14-16}
          \#vertices & \#opt & total & stage 2 & \#att. & \#B\&B-nodes & \#opt & total & stage 2 & stage 3 & \#att. & \#sub & \#opt & total & stage 2 & stage 3 & \#att. & \#sub \\
          \midrule
          \multicolumn{18}{@{}l}{chain length: 2}\\
          \multicolumn{18}{@{}l}{$B = 1$}\\
           20 &  \num{ 30} & \num{ 0.06} & \SI{30.55}{\percent} & \num{ 1.9} & \num{ 11.4}  &  \num{ 30} & \num{ 0.08} & \SI{15.67}{\percent} & \SI{19.56}{\percent} & \num{ 2.0} & \num{ 6.5}  &  \num{ 30} & \num{ 0.08} & \SI{21.96}{\percent} & \SI{15.64}{\percent} & \num{ 1.9} & \num{ 6.1} \\
           50 &  \num{ 30} & \num{ 3.25} & \SI{30.27}{\percent} & \num{ 3.4} & \num{ 37.0}  &  \num{ 30} & \num{ 2.88} & \SI{12.40}{\percent} & \SI{20.81}{\percent} & \num{ 3.4} & \num{ 17.1}  &  \num{ 30} & \num{ 3.18} & \SI{11.68}{\percent} & \SI{22.63}{\percent} & \num{ 3.1} & \num{ 14.9} \\
          100 &  \num{ 30} & \num{ 34.93} & \SI{34.44}{\percent} & \num{ 3.3} & \num{ 66.9}  &  \num{ 30} & \num{ 37.47} & \SI{7.98}{\percent} & \SI{25.18}{\percent} & \num{ 3.4} & \num{ 28.7}  &  \num{ 30} & \num{ 43.32} & \SI{8.55}{\percent} & \SI{31.72}{\percent} & \num{ 3.6} & \num{ 31.1} \\
          \multicolumn{18}{@{}l}{$B = 2$}\\
           20 &  \num{ 30} & \num{ 0.21} & \SI{55.94}{\percent} & \num{ 2.8} & \num{ 56.1}  &  \num{ 30} & \num{ 0.15} & \SI{30.15}{\percent} & \SI{16.11}{\percent} & \num{ 3.0} & \num{ 10.8}  &  \num{ 30} & \num{ 0.16} & \SI{35.88}{\percent} & \SI{13.95}{\percent} & \num{ 2.6} & \num{ 10.4} \\
           50 &  \num{ 30} & \num{ 17.25} & \SI{65.49}{\percent} & \num{ 5.5} & \num{ 519.8}  &  \num{ 30} & \num{ 8.42} & \SI{26.77}{\percent} & \SI{22.21}{\percent} & \num{ 4.8} & \num{ 60.8}  &  \num{ 30} & \num{ 9.54} & \SI{26.18}{\percent} & \SI{19.87}{\percent} & \num{ 4.8} & \num{ 52.3} \\
          100 &  \num{ 29} & \num{ 338.22} & \SI{77.93}{\percent} & \num{ 5.4} & \num{ 1752.7}  &  \num{ 30} & \num{ 102.18} & \SI{19.79}{\percent} & \SI{30.22}{\percent} & \num{ 4.9} & \num{ 119.5}  &  \num{ 30} & \num{ 137.89} & \SI{20.23}{\percent} & \SI{29.31}{\percent} & \num{ 5.5} & \num{ 120.2} \\
          \multicolumn{18}{@{}l}{$B = 3$}\\
           20 &  \num{ 30} & \num{ 0.61} & \SI{58.58}{\percent} & \num{ 3.6} & \num{ 213.9}  &  \num{ 30} & \num{ 0.18} & \SI{28.96}{\percent} & \SI{16.26}{\percent} & \num{ 3.2} & \num{ 10.9}  &  \num{ 30} & \num{ 0.16} & \SI{35.42}{\percent} & \SI{20.68}{\percent} & \num{ 3.0} & \num{ 10.8} \\
           50 &  \num{ 30} & \num{ 104.00} & \SI{85.81}{\percent} & \num{ 8.8} & \num{ 6067.7}  &  \num{ 30} & \num{ 17.82} & \SI{39.87}{\percent} & \SI{17.06}{\percent} & \num{ 6.1} & \num{ 126.6}  &  \num{ 30} & \num{ 25.55} & \SI{48.76}{\percent} & \SI{16.67}{\percent} & \num{ 6.1} & \num{ 138.5} \\
          100 &  \num{ 7} & \num{2682.88} & \SI{93.26}{\percent} & \num{ 8.4} & \num{12369.4}  &  \num{ 26} & \num{ 396.43} & \SI{37.39}{\percent} & \SI{28.33}{\percent} & \num{ 6.4} & \num{ 308.8}  &  \num{ 25} & \num{ 518.40} & \SI{40.29}{\percent} & \SI{25.17}{\percent} & \num{ 7.2} & \num{ 247.8} \\
          \multicolumn{18}{@{}l}{$B = 4$}\\
           20 &  \num{ 30} & \num{ 1.12} & \SI{62.20}{\percent} & \num{ 3.9} & \num{ 514.8}  &  \num{ 30} & \num{ 0.17} & \SI{32.37}{\percent} & \SI{14.32}{\percent} & \num{ 3.7} & \num{ 11.8}  &  \num{ 30} & \num{ 0.13} & \SI{33.35}{\percent} & \SI{17.83}{\percent} & \num{ 3.4} & \num{ 10.0} \\
           50 &  \num{ 28} & \num{ 492.07} & \SI{97.34}{\percent} & \num{10.0} & \num{43440.6}  &  \num{ 30} & \num{ 42.06} & \SI{61.14}{\percent} & \SI{13.83}{\percent} & \num{ 6.9} & \num{ 248.0}  &  \num{ 28} & \num{ 59.73} & \SI{57.61}{\percent} & \SI{10.67}{\percent} & \num{ 6.6} & \num{ 204.8} \\
          100 &  \num{ 0} & \num{3600.00} & \SI{91.21}{\percent} & --- & ---  &  \num{ 21} & \num{ 783.06} & \SI{50.85}{\percent} & \SI{19.76}{\percent} & \num{ 7.9} & \num{ 411.7}  &  \num{ 18} & \num{1254.98} & \SI{59.81}{\percent} & \SI{19.31}{\percent} & \num{ 8.4} & \num{ 396.8} \\
          \midrule
          \multicolumn{18}{@{}l}{chain length: 3}\\
          \multicolumn{18}{@{}l}{$B = 1$}\\
           20 &  \num{ 30} & \num{ 0.19} & \SI{32.40}{\percent} & \num{ 1.9} & \num{ 13.8}  &  \num{ 30} & \num{ 0.17} & \SI{25.70}{\percent} & \SI{16.20}{\percent} & \num{ 2.0} & \num{ 6.3}  &  \num{ 30} & \num{ 0.14} & \SI{30.30}{\percent} & \SI{19.43}{\percent} & \num{ 1.8} & \num{ 5.7} \\
           50 &  \num{ 30} & \num{ 18.20} & \SI{26.89}{\percent} & \num{ 3.6} & \num{ 44.3}  &  \num{ 30} & \num{ 14.81} & \SI{8.78}{\percent} & \SI{19.01}{\percent} & \num{ 3.3} & \num{ 18.9}  &  \num{ 30} & \num{ 5.57} & \SI{10.35}{\percent} & \SI{29.30}{\percent} & \num{ 3.0} & \num{ 13.7} \\
          100 &  \num{ 26} & \num{ 549.15} & \SI{40.93}{\percent} & \num{ 3.2} & \num{ 75.7}  &  \num{ 25} & \num{ 597.63} & \SI{6.26}{\percent} & \SI{28.83}{\percent} & \num{ 3.5} & \num{ 33.0}  &  \num{ 30} & \num{ 114.01} & \SI{8.76}{\percent} & \SI{36.75}{\percent} & \num{ 4.0} & \num{ 30.8} \\
          \multicolumn{18}{@{}l}{$B = 2$}\\
           20 &  \num{ 30} & \num{ 0.37} & \SI{50.04}{\percent} & \num{ 2.8} & \num{ 60.9}  &  \num{ 30} & \num{ 0.28} & \SI{30.41}{\percent} & \SI{17.93}{\percent} & \num{ 2.8} & \num{ 13.1}  &  \num{ 30} & \num{ 0.31} & \SI{22.48}{\percent} & \SI{19.08}{\percent} & \num{ 2.8} & \num{ 11.0} \\
           50 &  \num{ 30} & \num{ 75.92} & \SI{55.54}{\percent} & \num{ 6.2} & \num{ 651.4}  &  \num{ 30} & \num{ 32.13} & \SI{16.95}{\percent} & \SI{23.45}{\percent} & \num{ 5.1} & \num{ 68.2}  &  \num{ 30} & \num{ 18.69} & \SI{17.94}{\percent} & \SI{34.13}{\percent} & \num{ 4.6} & \num{ 55.5} \\
          100 &  \num{ 7} & \num{2630.62} & \SI{62.15}{\percent} & \num{ 5.7} & \num{ 1406.4}  &  \num{ 18} & \num{1415.26} & \SI{6.39}{\percent} & \SI{27.47}{\percent} & \num{ 5.9} & \num{ 89.3}  &  \num{ 29} & \num{ 303.46} & \SI{14.41}{\percent} & \SI{48.09}{\percent} & \num{ 4.9} & \num{ 120.0} \\
          \multicolumn{18}{@{}l}{$B = 3$}\\
           20 &  \num{ 30} & \num{ 1.03} & \SI{56.80}{\percent} & \num{ 3.5} & \num{ 251.5}  &  \num{ 30} & \num{ 0.43} & \SI{27.18}{\percent} & \SI{12.65}{\percent} & \num{ 3.6} & \num{ 12.7}  &  \num{ 30} & \num{ 0.28} & \SI{30.12}{\percent} & \SI{14.59}{\percent} & \num{ 3.1} & \num{ 11.1} \\
           50 &  \num{ 28} & \num{ 382.37} & \SI{83.40}{\percent} & \num{ 8.3} & \num{ 7286.9}  &  \num{ 30} & \num{ 52.07} & \SI{30.26}{\percent} & \SI{22.88}{\percent} & \num{ 6.2} & \num{ 153.4}  &  \num{ 30} & \num{ 42.75} & \SI{35.71}{\percent} & \SI{28.34}{\percent} & \num{ 5.5} & \num{ 135.1} \\
          100 &  \num{ 1} & \num{3510.61} & \SI{54.65}{\percent} & \num{ 8.0} & \num{11784.0}  &  \num{ 11} & \num{2093.21} & \SI{10.21}{\percent} & \SI{41.77}{\percent} & \num{ 5.7} & \num{ 200.7}  &  \num{ 21} & \num{1262.74} & \SI{25.70}{\percent} & \SI{54.60}{\percent} & \num{ 6.8} & \num{ 268.2} \\
          \multicolumn{18}{@{}l}{$B = 4$}\\
           20 &  \num{ 30} & \num{ 2.13} & \SI{65.14}{\percent} & \num{ 4.0} & \num{ 762.8}  &  \num{ 30} & \num{ 0.20} & \SI{29.54}{\percent} & \SI{13.61}{\percent} & \num{ 3.4} & \num{ 11.2}  &  \num{ 30} & \num{ 0.30} & \SI{38.98}{\percent} & \SI{10.52}{\percent} & \num{ 3.8} & \num{ 12.2} \\
           50 &  \num{ 16} & \num{1493.15} & \SI{95.55}{\percent} & \num{10.9} & \num{41147.7}  &  \num{ 29} & \num{ 103.31} & \SI{42.99}{\percent} & \SI{18.11}{\percent} & \num{ 7.2} & \num{ 265.9}  &  \num{ 28} & \num{ 75.70} & \SI{49.02}{\percent} & \SI{21.11}{\percent} & \num{ 7.4} & \num{ 185.8} \\
          100 &  \num{ 0} & \num{3600.00} & \SI{58.88}{\percent} & --- & ---  &  \num{ 6} & \num{2956.89} & \SI{15.72}{\percent} & \SI{33.07}{\percent} & \num{ 8.3} & \num{ 540.0}  &  \num{ 13} & \num{2140.89} & \SI{36.24}{\percent} & \SI{44.04}{\percent} & \num{ 8.9} & \num{ 364.5} \\
          \midrule
          \multicolumn{18}{@{}l}{chain length: 4}\\
          \multicolumn{18}{@{}l}{$B = 1$}\\
           20 &  \num{ 30} & \num{ 0.44} & \SI{19.38}{\percent} & \num{ 1.8} & \num{ 14.2}  &  \num{ 30} & \num{ 0.58} & \SI{24.81}{\percent} & \SI{11.99}{\percent} & \num{ 2.1} & \num{ 6.2}  &  \num{ 30} & \num{ 0.23} & \SI{24.25}{\percent} & \SI{23.81}{\percent} & \num{ 1.8} & \num{ 6.0} \\
           50 &  \num{ 30} & \num{ 142.77} & \SI{28.28}{\percent} & \num{ 3.0} & \num{ 51.9}  &  \num{ 30} & \num{ 196.04} & \SI{5.21}{\percent} & \SI{17.33}{\percent} & \num{ 3.8} & \num{ 23.4}  &  \num{ 30} & \num{ 16.06} & \SI{8.22}{\percent} & \SI{36.91}{\percent} & \num{ 3.1} & \num{ 18.5} \\
          100 &  \num{ 3} & \num{3042.85} & \SI{20.15}{\percent} & \num{ 2.3} & \num{ 79.7}  &  \num{ 3} & \num{3283.32} & \SI{10.86}{\percent} & \SI{6.29}{\percent} & \num{ 3.7} & \num{ 34.3}  &  \num{ 30} & \num{ 195.42} & \SI{7.60}{\percent} & \SI{53.41}{\percent} & \num{ 3.4} & \num{ 39.3} \\
          \multicolumn{18}{@{}l}{$B = 2$}\\
           20 &  \num{ 30} & \num{ 1.13} & \SI{45.74}{\percent} & \num{ 3.0} & \num{ 92.1}  &  \num{ 30} & \num{ 0.91} & \SI{24.61}{\percent} & \SI{12.69}{\percent} & \num{ 2.8} & \num{ 13.5}  &  \num{ 30} & \num{ 0.50} & \SI{24.01}{\percent} & \SI{12.87}{\percent} & \num{ 2.9} & \num{ 12.3} \\
           50 &  \num{ 25} & \num{ 505.81} & \SI{59.78}{\percent} & \num{ 5.6} & \num{ 726.4}  &  \num{ 29} & \num{ 321.98} & \SI{10.08}{\percent} & \SI{29.94}{\percent} & \num{ 4.9} & \num{ 95.7}  &  \num{ 30} & \num{ 41.36} & \SI{18.89}{\percent} & \SI{43.99}{\percent} & \num{ 4.2} & \num{ 77.2} \\
          100 &  \num{ 0} & \num{3600.00} & \SI{11.76}{\percent} & --- & ---  &  \num{ 1} & \num{3477.83} & \SI{9.58}{\percent} & \SI{6.12}{\percent} & \num{ 7.0} & \num{ 56.0}  &  \num{ 27} & \num{ 674.06} & \SI{14.31}{\percent} & \SI{56.58}{\percent} & \num{ 4.9} & \num{ 160.7} \\
          \multicolumn{18}{@{}l}{$B = 3$}\\
           20 &  \num{ 30} & \num{ 2.10} & \SI{54.68}{\percent} & \num{ 3.7} & \num{ 250.6}  &  \num{ 30} & \num{ 0.90} & \SI{20.47}{\percent} & \SI{13.84}{\percent} & \num{ 3.2} & \num{ 12.4}  &  \num{ 30} & \num{ 0.52} & \SI{37.43}{\percent} & \SI{13.49}{\percent} & \num{ 3.1} & \num{ 13.3} \\
           50 &  \num{ 16} & \num{1369.98} & \SI{80.02}{\percent} & \num{ 9.4} & \num{ 6808.6}  &  \num{ 27} & \num{ 423.88} & \SI{15.23}{\percent} & \SI{27.43}{\percent} & \num{ 6.6} & \num{ 164.7}  &  \num{ 30} & \num{ 79.13} & \SI{29.21}{\percent} & \SI{35.99}{\percent} & \num{ 6.0} & \num{ 153.9} \\
          100 &  \num{ 0} & \num{3600.00} & \SI{10.30}{\percent} & --- & ---  &  \num{ 0} & \num{3600.00} & \SI{4.71}{\percent} & \SI{5.72}{\percent} & --- & ---  &  \num{ 18} & \num{1908.69} & \SI{21.57}{\percent} & \SI{58.79}{\percent} & \num{ 6.9} & \num{ 249.3} \\
          \multicolumn{18}{@{}l}{$B = 4$}\\
           20 &  \num{ 30} & \num{ 3.88} & \SI{59.30}{\percent} & \num{ 4.5} & \num{ 959.4}  &  \num{ 30} & \num{ 0.84} & \SI{25.28}{\percent} & \SI{8.16}{\percent} & \num{ 3.6} & \num{ 13.3}  &  \num{ 30} & \num{ 0.39} & \SI{28.25}{\percent} & \SI{16.86}{\percent} & \num{ 3.7} & \num{ 11.7} \\
           50 &  \num{ 8} & \num{2442.50} & \SI{87.49}{\percent} & \num{10.4} & \num{38666.1}  &  \num{ 25} & \num{ 563.46} & \SI{22.20}{\percent} & \SI{25.68}{\percent} & \num{ 7.8} & \num{ 227.4}  &  \num{ 27} & \num{ 134.69} & \SI{42.51}{\percent} & \SI{25.01}{\percent} & \num{ 7.2} & \num{ 193.7} \\
          100 &  \num{ 0} & \num{3600.00} & \SI{9.46}{\percent} & --- & ---  &  \num{ 0} & \num{3600.00} & \SI{4.34}{\percent} & \SI{1.94}{\percent} & --- & ---  &  \num{ 6} & \num{3083.76} & \SI{32.60}{\percent} & \SI{50.54}{\percent} & \num{ 7.8} & \num{ 321.0} \\
          \bottomrule
        \end{tabular*}
        \caption{Comparison of the different methods for maximum cycle length 3 and different chain lengths, for the Full Recourse policy.}
      \end{tiny}
    \end{table}
  \end{landscape}
}

%% file: Tables/MethodComparison/comparison-cyc3-policy2.tex
  \begin{landscape}
\begin{table}[htbp]
  \begin{tiny}
    
    \label{tab:cyc3-policy2}
    \begin{tabular*}{\linewidth}{@{}l@{\;\;\extracolsep{\fill}}rrrrrrrrrrrr@{}}\toprule
       & \multicolumn{6}{c}{cut-CC} & \multicolumn{6}{c}{cut-PICEF}\\
      \cmidrule{2-7} \cmidrule{8-13}
      & & \multicolumn{3}{c}{time} & & & & \multicolumn{3}{c}{time} & & \\
      \cmidrule{3-5} \cmidrule{9-11}
      \#vertices & \#opt & total & stage 2 & stage 3 & \#att. & \#sub & \#opt & total & stage 2 & stage 3 & \#att. & \#sub \\
      \midrule
      \multicolumn{13}{@{}l}{chain length: 2}\\
      \multicolumn{13}{@{}l}{$B = 1$}\\
       20 &  \num{ 30} & \num{ 0.24} & \SI{31.49}{\percent} & \SI{4.67}{\percent} & \num{ 2.6} & \num{ 6.1}  &  \num{ 30} & \num{ 0.29} & \SI{13.69}{\percent} & \SI{10.64}{\percent} & \num{ 2.5} & \num{ 4.8} \\
       50 &  \num{ 30} & \num{ 8.23} & \SI{28.81}{\percent} & \SI{1.50}{\percent} & \num{ 5.2} & \num{ 17.0}  &  \num{ 29} & \num{ 33.57} & \SI{1.96}{\percent} & \SI{0.51}{\percent} & \num{ 4.5} & \num{ 10.6} \\
      100 &  \num{ 26} & \num{ 187.88} & \SI{30.48}{\percent} & \SI{0.54}{\percent} & \num{ 6.5} & \num{ 22.8}  &  \num{ 25} & \num{ 661.56} & \SI{0.34}{\percent} & \SI{0.05}{\percent} & \num{ 4.9} & \num{ 12.4} \\
      \multicolumn{13}{@{}l}{$B = 2$}\\
       20 &  \num{ 30} & \num{ 0.55} & \SI{33.79}{\percent} & \SI{10.00}{\percent} & \num{ 3.6} & \num{ 11.3}  &  \num{ 30} & \num{ 0.73} & \SI{17.90}{\percent} & \SI{10.76}{\percent} & \num{ 3.5} & \num{ 10.5} \\
       50 &  \num{ 29} & \num{ 49.68} & \SI{17.26}{\percent} & \SI{1.72}{\percent} & \num{ 9.4} & \num{ 65.4}  &  \num{ 27} & \num{ 125.42} & \SI{3.61}{\percent} & \SI{0.77}{\percent} & \num{ 8.1} & \num{ 41.9} \\
      100 &  \num{ 11} & \num{1733.28} & \SI{4.35}{\percent} & \SI{0.51}{\percent} & \num{11.3} & \num{ 102.7}  &  \num{ 9} & \num{2239.02} & \SI{0.15}{\percent} & \SI{0.05}{\percent} & \num{ 7.3} & \num{ 34.3} \\
      \multicolumn{13}{@{}l}{$B = 3$}\\
       20 &  \num{ 30} & \num{ 1.00} & \SI{47.74}{\percent} & \SI{14.46}{\percent} & \num{ 3.8} & \num{ 19.0}  &  \num{ 30} & \num{ 0.93} & \SI{24.09}{\percent} & \SI{15.57}{\percent} & \num{ 4.0} & \num{ 16.0} \\
       50 &  \num{ 27} & \num{ 235.93} & \SI{23.44}{\percent} & \SI{1.60}{\percent} & \num{14.3} & \num{ 157.2}  &  \num{ 22} & \num{ 401.65} & \SI{5.35}{\percent} & \SI{1.11}{\percent} & \num{11.3} & \num{ 91.6} \\
      100 &  \num{ 4} & \num{2808.18} & \SI{6.33}{\percent} & \SI{0.56}{\percent} & \num{14.5} & \num{ 159.0}  &  \num{ 2} & \num{3445.56} & \SI{0.10}{\percent} & \SI{2.59}{\percent} & \num{16.0} & \num{ 88.5} \\
      \multicolumn{13}{@{}l}{$B = 4$}\\
       20 &  \num{ 30} & \num{ 0.35} & \SI{42.87}{\percent} & \SI{13.45}{\percent} & \num{ 3.6} & \num{ 10.8}  &  \num{ 30} & \num{ 0.76} & \SI{38.74}{\percent} & \SI{13.75}{\percent} & \num{ 3.5} & \num{ 14.3} \\
       50 &  \num{ 21} & \num{ 596.78} & \SI{28.85}{\percent} & \SI{1.54}{\percent} & \num{17.0} & \num{ 358.1}  &  \num{ 17} & \num{ 830.07} & \SI{6.13}{\percent} & \SI{0.58}{\percent} & \num{15.3} & \num{ 185.9} \\
      100 &  \num{ 0} & \num{3600.00} & \SI{4.76}{\percent} & \SI{0.33}{\percent} & --- & ---  &  \num{ 0} & \num{3600.00} & \SI{0.09}{\percent} & \SI{0.03}{\percent} & --- & --- \\
      \midrule
      \multicolumn{13}{@{}l}{chain length: 3}\\
      \multicolumn{13}{@{}l}{$B = 1$}\\
       20 &  \num{ 30} & \num{ 0.86} & \SI{35.61}{\percent} & \SI{5.26}{\percent} & \num{ 2.7} & \num{ 6.9}  &  \num{ 30} & \num{ 1.50} & \SI{14.69}{\percent} & \SI{9.10}{\percent} & \num{ 2.9} & \num{ 5.8} \\
       50 &  \num{ 28} & \num{ 98.70} & \SI{23.32}{\percent} & \SI{0.66}{\percent} & \num{ 5.6} & \num{ 21.0}  &  \num{ 28} & \num{ 98.75} & \SI{0.94}{\percent} & \SI{0.15}{\percent} & \num{ 4.8} & \num{ 12.2} \\
      100 &  \num{ 13} & \num{2198.62} & \SI{32.36}{\percent} & \SI{0.69}{\percent} & \num{ 5.6} & \num{ 19.9}  &  \num{ 14} & \num{1169.76} & \SI{0.19}{\percent} & \SI{0.03}{\percent} & \num{ 3.9} & \num{ 8.6} \\
      \multicolumn{13}{@{}l}{$B = 2$}\\
       20 &  \num{ 30} & \num{ 2.18} & \SI{31.51}{\percent} & \SI{11.82}{\percent} & \num{ 4.2} & \num{ 14.9}  &  \num{ 30} & \num{ 1.87} & \SI{15.34}{\percent} & \SI{9.36}{\percent} & \num{ 3.8} & \num{ 12.4} \\
       50 &  \num{ 24} & \num{ 330.08} & \SI{16.62}{\percent} & \SI{1.16}{\percent} & \num{ 8.9} & \num{ 63.8}  &  \num{ 18} & \num{ 635.58} & \SI{1.04}{\percent} & \SI{0.32}{\percent} & \num{ 7.7} & \num{ 34.9} \\
      100 &  \num{ 2} & \num{3343.52} & \SI{9.90}{\percent} & \SI{0.54}{\percent} & \num{13.0} & \num{ 43.5}  &  \num{ 0} & \num{3600.00} & \SI{0.07}{\percent} & \SI{0.02}{\percent} & --- & --- \\
      \multicolumn{13}{@{}l}{$B = 3$}\\
       20 &  \num{ 30} & \num{ 2.43} & \SI{41.52}{\percent} & \SI{14.90}{\percent} & \num{ 4.0} & \num{ 19.3}  &  \num{ 30} & \num{ 2.60} & \SI{24.62}{\percent} & \SI{8.35}{\percent} & \num{ 4.3} & \num{ 18.7} \\
       50 &  \num{ 17} & \num{1136.03} & \SI{20.11}{\percent} & \SI{1.05}{\percent} & \num{14.8} & \num{ 156.4}  &  \num{ 15} & \num{1210.16} & \SI{0.63}{\percent} & \SI{0.21}{\percent} & \num{11.3} & \num{ 85.6} \\
      100 &  \num{ 2} & \num{3547.29} & \SI{12.34}{\percent} & \SI{0.57}{\percent} & \num{18.0} & \num{ 192.5}  &  \num{ 0} & \num{3600.00} & \SI{0.06}{\percent} & \SI{0.02}{\percent} & --- & --- \\
      \multicolumn{13}{@{}l}{$B = 4$}\\
       20 &  \num{ 30} & \num{ 1.81} & \SI{49.86}{\percent} & \SI{6.00}{\percent} & \num{ 4.0} & \num{ 19.1}  &  \num{ 30} & \num{ 1.04} & \SI{30.14}{\percent} & \SI{13.76}{\percent} & \num{ 3.3} & \num{ 14.0} \\
       50 &  \num{ 12} & \num{1491.58} & \SI{25.46}{\percent} & \SI{0.99}{\percent} & \num{15.2} & \num{ 209.5}  &  \num{ 12} & \num{1783.44} & \SI{0.71}{\percent} & \SI{0.23}{\percent} & \num{17.6} & \num{ 177.4} \\
      100 &  \num{ 1} & \num{3518.60} & \SI{9.54}{\percent} & \SI{0.23}{\percent} & \num{14.0} & \num{ 192.0}  &  \num{ 0} & \num{3600.00} & \SI{0.05}{\percent} & \SI{0.01}{\percent} & --- & --- \\
      \midrule
      \multicolumn{13}{@{}l}{chain length: 4}\\
      \multicolumn{13}{@{}l}{$B = 1$}\\
       20 &  \num{ 30} & \num{ 3.66} & \SI{36.13}{\percent} & \SI{4.14}{\percent} & \num{ 2.8} & \num{ 6.7}  &  \num{ 30} & \num{ 1.31} & \SI{11.75}{\percent} & \SI{5.58}{\percent} & \num{ 2.7} & \num{ 5.2} \\
       50 &  \num{ 20} & \num{ 692.00} & \SI{23.34}{\percent} & \SI{0.69}{\percent} & \num{ 4.7} & \num{ 14.7}  &  \num{ 23} & \num{ 210.85} & \SI{0.77}{\percent} & \SI{0.16}{\percent} & \num{ 4.3} & \num{ 9.5} \\
      100 &  \num{ 1} & \num{3566.63} & \SI{18.65}{\percent} & \SI{0.13}{\percent} & \num{ 5.0} & \num{ 11.0}  &  \num{ 6} & \num{2563.12} & \SI{0.07}{\percent} & \SI{0.01}{\percent} & \num{ 4.0} & \num{ 9.5} \\
      \multicolumn{13}{@{}l}{$B = 2$}\\
       20 &  \num{ 30} & \num{ 6.11} & \SI{30.63}{\percent} & \SI{10.35}{\percent} & \num{ 4.2} & \num{ 17.2}  &  \num{ 30} & \num{ 3.83} & \SI{17.52}{\percent} & \SI{7.20}{\percent} & \num{ 4.2} & \num{ 15.2} \\
       50 &  \num{ 9} & \num{1655.99} & \SI{13.84}{\percent} & \SI{0.65}{\percent} & \num{ 8.3} & \num{ 39.1}  &  \num{ 15} & \num{ 968.64} & \SI{0.48}{\percent} & \SI{0.15}{\percent} & \num{ 7.0} & \num{ 39.4} \\
      100 &  \num{ 0} & \num{3600.00} & \SI{12.00}{\percent} & \SI{0.06}{\percent} & --- & ---  &  \num{ 0} & \num{3600.00} & \SI{0.04}{\percent} & \SI{0.01}{\percent} & --- & --- \\
      \multicolumn{13}{@{}l}{$B = 3$}\\
       20 &  \num{ 30} & \num{ 4.81} & \SI{33.41}{\percent} & \SI{7.84}{\percent} & \num{ 3.9} & \num{ 16.1}  &  \num{ 30} & \num{ 3.67} & \SI{23.60}{\percent} & \SI{8.39}{\percent} & \num{ 4.6} & \num{ 23.0} \\
       50 &  \num{ 6} & \num{2149.74} & \SI{15.14}{\percent} & \SI{0.54}{\percent} & \num{13.3} & \num{ 86.3}  &  \num{ 9} & \num{1738.23} & \SI{0.70}{\percent} & \SI{0.11}{\percent} & \num{11.9} & \num{ 93.7} \\
      100 &  \num{ 0} & \num{3600.00} & \SI{15.93}{\percent} & \SI{0.04}{\percent} & --- & ---  &  \num{ 0} & \num{3600.00} & \SI{0.04}{\percent} & \SI{0.02}{\percent} & --- & --- \\
      \multicolumn{13}{@{}l}{$B = 4$}\\
       20 &  \num{ 30} & \num{ 3.07} & \SI{45.49}{\percent} & \SI{6.24}{\percent} & \num{ 3.5} & \num{ 14.6}  &  \num{ 30} & \num{ 2.24} & \SI{25.68}{\percent} & \SI{15.86}{\percent} & \num{ 4.1} & \num{ 18.2} \\
       50 &  \num{ 4} & \num{2640.83} & \SI{22.61}{\percent} & \SI{0.49}{\percent} & \num{17.8} & \num{ 133.0}  &  \num{ 9} & \num{1896.04} & \SI{0.69}{\percent} & \SI{1.74}{\percent} & \num{14.2} & \num{ 120.2} \\
      100 &  \num{ 0} & \num{3600.00} & \SI{17.15}{\percent} & \SI{0.05}{\percent} & --- & ---  &  \num{ 0} & \num{3600.00} & \SI{0.04}{\percent} & \SI{0.02}{\percent} & --- & --- \\
      \bottomrule
    \end{tabular*}
    \caption{Comparison of the different methods for maximum cycle length 3 and different chain lengths, for the FSE policy.}
  \end{tiny}
\end{table}
\end{landscape}

%% file: Tables/FR-vs-FSE.tex
\begin{table}
  \centering
  \begin{scriptsize}
    
    \label{tab:fr-vs-fse}
    \begin{tabular*}{0.5\textwidth}{@{}l@{\;\;\extracolsep{\fill}}rrr@{}}\toprule
          & $z^{*}_{\fr} - z^{*}_{\fse}$ & $\#$instances & Relative count \\
      \midrule
      \multicolumn{4}{@{}l}{$|V|=20$} \\
        \cmidrule{1-4}
          &  0 & \num{613} & \num{0.851}\\
          &  1 & \num{104} & \num{0.144}\\
          &  2 &  \num{3} & \num{0.004}\\
      \midrule
        Total & & \num{720}& \\
      \midrule
      \multicolumn{4}{@{}l}{$|V|=50$} \\
      \cmidrule{1-4}
          &  \num{0} & \num{318}& \num{0.748}\\
          &  \num{1} & \num{103}& \num{0.242}\\
          &  \num{2} &  \num{4} & \num{0.009}\\
      \midrule
        Total & & \num{425}& \\
      \midrule
      \multicolumn{4}{@{}l}{$|V|=100$} \\
      \cmidrule{1-4}
          & \num{0} & \num{89} & \num{0.947}\\
          &  \num{1} &  \num{4} & \num{0.043}\\
          &  \num{2} &  \num{1} & \num{0.011}\\
      \midrule
        Total & & \num{94} & \\
      \bottomrule
    \end{tabular*}
    \caption{Comparison of average objective values with FR and FSE for all instances that are solved by cut-CC and/or cut-PICEF for both recourse policies. Total number of instances per value of $|V|$ is 720,~\ie~$\cyclen\in\{3,4\}, \chainlen\in\{2,3,4\}$ and $B \in \{1,2,3,4\}$, 30 instances each.}
  \end{scriptsize}
\end{table}

%% file: Tables/LiftingTables/combined-lifting-table-cyc3.tex
\begin{table}
  \begin{tiny}
    \label{tab:lifting-cyc3}
    \begin{tabular*}{\textwidth}{@{}l@{\;\;\extracolsep{\fill}}rrrrrrrrrrrr@{}}\toprule
      {\bfseries cut-CC: FR} & \multicolumn{6}{c}{with lifting} & \multicolumn{6}{c}{without lifting} \\
      \cmidrule{2-7} \cmidrule{8-13}
      & & \multicolumn{3}{c}{time} & & & & \multicolumn{3}{c}{time} & & \\
      \cmidrule{3-5} \cmidrule{9-11}
       & \#opt & total & stage 2 & stage 3 & \#att. & \#sub & \#opt & total & stage 2 & stage 3 & \#att. & \#sub \\
      \midrule
      \multicolumn{13}{@{}l}{chain length: 2}\\
      $B = 1$ & \num{30} & \num{33.27} & \SI{7.13}{\percent} & \SI{17.38}{\percent} & \num{3.4} & \num{20.8}  &  \num{30} & \num{37.47} & \SI{7.98}{\percent} & \SI{25.18}{\percent} & \num{3.4} & \num{28.7} \\
      $B = 2$ & \num{30} & \num{74.27} & \SI{9.13}{\percent} & \SI{19.33}{\percent} & \num{5.3} & \num{57.2}  &  \num{30} & \num{102.18} & \SI{19.79}{\percent} & \SI{30.22}{\percent} & \num{4.9} & \num{119.5} \\
      $B = 3$ & \num{30} & \num{163.37} & \SI{20.28}{\percent} & \SI{23.41}{\percent} & \num{6.1} & \num{124.5}  &  \num{26} & \num{396.43} & \SI{37.39}{\percent} & \SI{28.33}{\percent} & \num{6.4} & \num{308.8} \\
      $B = 4$ & \num{30} & \num{258.16} & \SI{28.08}{\percent} & \SI{21.96}{\percent} & \num{7.5} & \num{215.3}  &  \num{21} & \num{783.06} & \SI{50.85}{\percent} & \SI{19.76}{\percent} & \num{7.9} & \num{411.7} \\
      \midrule
      \multicolumn{13}{@{}l}{chain length: 3}\\
      $B = 1$ & \num{28} & \num{540.66} & \SI{5.82}{\percent} & \SI{21.28}{\percent} & \num{3.7} & \num{26.6}  &  \num{25} & \num{597.63} & \SI{6.26}{\percent} & \SI{28.83}{\percent} & \num{3.5} & \num{33.0} \\
      $B = 2$ & \num{23} & \num{1175.08} & \SI{5.42}{\percent} & \SI{21.29}{\percent} & \num{5.4} & \num{55.8}  &  \num{18} & \num{1415.26} & \SI{6.39}{\percent} & \SI{27.47}{\percent} & \num{5.9} & \num{89.3} \\
      $B = 3$ & \num{20} & \num{1660.15} & \SI{7.33}{\percent} & \SI{28.94}{\percent} & \num{6.2} & \num{130.2}  &  \num{11} & \num{2093.21} & \SI{10.21}{\percent} & \SI{41.77}{\percent} & \num{5.7} & \num{200.7} \\
      $B = 4$ & \num{17} & \num{2155.48} & \SI{6.32}{\percent} & \SI{20.18}{\percent} & \num{8.1} & \num{181.9}  &  \num{6} & \num{2956.89} & \SI{15.72}{\percent} & \SI{33.07}{\percent} & \num{8.3} & \num{540.0} \\
      \midrule
      \multicolumn{13}{@{}l}{chain length: 4}\\
      $B = 1$ & \num{3} & \num{3159.06} & \SI{10.15}{\percent} & \SI{4.91}{\percent} & \num{3.7} & \num{16.0}  &  \num{3} & \num{3283.32} & \SI{10.86}{\percent} & \SI{6.29}{\percent} & \num{3.7} & \num{34.3} \\
      $B = 2$ & \num{2} & \num{3401.13} & \SI{8.40}{\percent} & \SI{7.42}{\percent} & \num{4.5} & \num{74.5}  &  \num{1} & \num{3477.83} & \SI{9.58}{\percent} & \SI{6.12}{\percent} & \num{7.0} & \num{56.0} \\
      $B = 3$ & \num{2} & \num{3400.58} & \SI{5.16}{\percent} & \SI{5.91}{\percent} & \num{4.0} & \num{86.0}  &  \num{0} & \num{3600.00} & \SI{4.71}{\percent} & \SI{5.72}{\percent} & --- & --- \\
      $B = 4$ & \num{1} & \num{3498.25} & \SI{4.75}{\percent} & \SI{1.70}{\percent} & \num{10.0} & \num{90.0}  &  \num{0} & \num{3600.00} & \SI{4.34}{\percent} & \SI{1.94}{\percent} & --- & --- \\
      \midrule
      {\bfseries cut-CC: FSE} & & \multicolumn{3}{c}{time} & & & & \multicolumn{3}{c}{time} & & \\
      \cmidrule{3-5} \cmidrule{9-11}
       & \#opt & total & stage 2 & stage 3 & \#att. & \#sub & \#opt & total & stage 2 & stage 3 & \#att. & \#sub \\
      \midrule
      \multicolumn{13}{@{}l}{chain length: 2}\\
      $B = 1$ &  \num{26} & \num{189.69} & \SI{30.38}{\percent} & \SI{0.60}{\percent} & \num{6.5} & \num{22.7}  &  \num{26} & \num{187.88} & \SI{30.48}{\percent} & \SI{0.54}{\percent} & \num{6.5} & \num{22.8} \\
      $B = 2$ &  \num{12} & \num{1490.95} & \SI{3.97}{\percent} & \SI{0.52}{\percent} & \num{11.1} & \num{50.2}  &  \num{11} & \num{1733.28} & \SI{4.35}{\percent} & \SI{0.51}{\percent} & \num{11.3} & \num{102.7} \\
      $B = 3$ &  \num{6} & \num{2901.80} & \SI{5.41}{\percent} & \SI{0.42}{\percent} & \num{16.2} & \num{227.8}  &  \num{4} & \num{2808.18} & \SI{6.33}{\percent} & \SI{0.56}{\percent} & \num{14.5} & \num{159.0} \\
      $B = 4$ &  \num{1} & \num{3510.64} & \SI{2.65}{\percent} & \SI{0.23}{\percent} & \num{13.0} & \num{307.0}  &  \num{0} & \num{3600.00} & \SI{4.76}{\percent} & \SI{0.33}{\percent} & --- & --- \\
      \midrule
      \multicolumn{13}{@{}l}{chain length: 3}\\
       $B = 1$ &  \num{14} & \num{2132.62} & \SI{32.61}{\percent} & \SI{0.70}{\percent} & \num{5.6} & \num{20.1}  &  \num{13} & \num{2198.62} & \SI{32.36}{\percent} & \SI{0.69}{\percent} & \num{5.6} & \num{19.9} \\
       $B = 2$ &  \num{2} & \num{3546.73} & \SI{11.28}{\percent} & \SI{0.54}{\percent} & \num{17.5} & \num{84.0}  &  \num{2} & \num{3343.52} & \SI{9.90}{\percent} & \SI{0.54}{\percent} & \num{13.0} & \num{43.5} \\
       $B = 3$ &  \num{0} & \num{3600.00} & \SI{12.27}{\percent} & \SI{0.52}{\percent} & --- & ---  &  \num{2} & \num{3547.29} & \SI{12.34}{\percent} & \SI{0.57}{\percent} & \num{18.0} & \num{192.5} \\
       $B = 4$ &  \num{0} & \num{3600.00} & \SI{9.51}{\percent} & \SI{0.19}{\percent} & --- & ---  &  \num{1} & \num{3518.60} & \SI{9.54}{\percent} & \SI{0.23}{\percent} & \num{14.0} & \num{192.0} \\
      \midrule
      \multicolumn{13}{@{}l}{chain length: 4}\\
       $B = 1$ &  \num{1} & \num{3561.68} & \SI{19.96}{\percent} & \SI{0.17}{\percent} & \num{6.0} & \num{20.0}  &  \num{1} & \num{3566.63} & \SI{18.65}{\percent} & \SI{0.13}{\percent} & \num{5.0} & \num{11.0} \\
       $B = 2$ &  \num{0} & \num{3600.00} & \SI{12.39}{\percent} & \SI{0.05}{\percent} & --- & ---  &  \num{0} & \num{3600.00} & \SI{12.00}{\percent} & \SI{0.06}{\percent} & --- & --- \\
       $B = 3$ &  \num{0} & \num{3600.00} & \SI{15.62}{\percent} & \SI{0.05}{\percent} & --- & ---  &  \num{0} & \num{3600.00} & \SI{15.93}{\percent} & \SI{0.04}{\percent} & --- & --- \\
       $B = 4$ &  \num{0} & \num{3600.00} & \SI{17.94}{\percent} & \SI{0.04}{\percent} & --- & ---  &  \num{0} & \num{3600.00} & \SI{17.15}{\percent} & \SI{0.05}{\percent} & --- & --- \\
      \midrule
      {\bfseries cut-PICEF: FR} & & \multicolumn{3}{c}{time} & & & & \multicolumn{3}{c}{time} & & \\
      \cmidrule{3-5} \cmidrule{9-11}
       & \#opt & total & stage 2 & stage 3 & \#att. & \#sub & \#opt & total & stage 2 & stage 3 & \#att. & \#sub \\
      \midrule
      \multicolumn{13}{@{}l}{chain length: 2}\\
      $B = 1$ & \num{30} & \num{62.88} & \SI{11.65}{\percent} & \SI{45.83}{\percent} & \num{3.3} & \num{21.7}  &  \num{30} & \num{43.32} & \SI{8.55}{\percent} & \SI{31.72}{\percent} & \num{3.6} & \num{31.1} \\
      $B = 2$ & \num{30} & \num{152.87} & \SI{11.69}{\percent} & \SI{48.87}{\percent} & \num{5.3} & \num{57.1}  &  \num{30} & \num{137.89} & \SI{20.23}{\percent} & \SI{29.31}{\percent} & \num{5.5} & \num{120.2} \\
      $B = 3$ & \num{30} & \num{319.64} & \SI{12.80}{\percent} & \SI{53.06}{\percent} & \num{6.4} & \num{94.5}  &  \num{25} & \num{518.40} & \SI{40.29}{\percent} & \SI{25.17}{\percent} & \num{7.2} & \num{247.8} \\
      $B = 4$ & \num{29} & \num{542.21} & \SI{17.06}{\percent} & \SI{55.70}{\percent} & \num{7.8} & \num{166.0}  &  \num{18} & \num{1254.98} & \SI{59.81}{\percent} & \SI{19.31}{\percent} & \num{8.4} & \num{396.8} \\
      \midrule
      \multicolumn{13}{@{}l}{chain length: 3}\\
       $B = 1$ & \num{30} & \num{119.61} & \SI{13.72}{\percent} & \SI{40.38}{\percent} & \num{3.9} & \num{22.8}  &  \num{30} & \num{114.01} & \SI{8.76}{\percent} & \SI{36.75}{\percent} & \num{4.0} & \num{30.8} \\
       $B = 2$ & \num{30} & \num{244.16} & \SI{9.25}{\percent} & \SI{55.42}{\percent} & \num{4.6} & \num{56.7}  &  \num{29} & \num{303.46} & \SI{14.41}{\percent} & \SI{48.09}{\percent} & \num{4.9} & \num{120.0} \\
       $B = 3$ & \num{27} & \num{752.24} & \SI{9.85}{\percent} & \SI{60.42}{\percent} & \num{6.4} & \num{95.8}  &  \num{21} & \num{1262.74} & \SI{25.70}{\percent} & \SI{54.60}{\percent} & \num{6.8} & \num{268.2} \\
       $B = 4$ & \num{27} & \num{911.99} & \SI{10.91}{\percent} & \SI{59.01}{\percent} & \num{7.5} & \num{141.6}  &  \num{13} & \num{2140.89} & \SI{36.24}{\percent} & \SI{44.04}{\percent} & \num{8.9} & \num{364.5} \\
      \midrule
      \multicolumn{13}{@{}l}{chain length: 4}\\
       $B = 1$ & \num{30} & \num{181.29} & \SI{10.58}{\percent} & \SI{49.64}{\percent} & \num{3.3} & \num{24.0}  &  \num{30} & \num{195.42} & \SI{7.60}{\percent} & \SI{53.41}{\percent} & \num{3.4} & \num{39.3} \\
       $B = 2$ & \num{30} & \num{409.29} & \SI{8.86}{\percent} & \SI{46.99}{\percent} & \num{5.3} & \num{46.7}  &  \num{27} & \num{674.06} & \SI{14.31}{\percent} & \SI{56.58}{\percent} & \num{4.9} & \num{160.7} \\
       $B = 3$ & \num{25} & \num{1165.20} & \SI{7.56}{\percent} & \SI{54.31}{\percent} & \num{6.6} & \num{83.9}  &  \num{18} & \num{1908.69} & \SI{21.57}{\percent} & \SI{58.79}{\percent} & \num{6.9} & \num{249.3} \\
       $B = 4$ & \num{21} & \num{1525.97} & \SI{10.41}{\percent} & \SI{63.18}{\percent} & \num{7.3} & \num{115.1}  &  \num{6} & \num{3083.76} & \SI{32.60}{\percent} & \SI{50.54}{\percent} & \num{7.8} & \num{321.0} \\
       \midrule
      {\bfseries cut-PICEF: FSE} & & \multicolumn{3}{c}{time} & & & & \multicolumn{3}{c}{time} & & \\
      \cmidrule{3-5} \cmidrule{9-11}
       & \#opt & total & stage 2 & stage 3 & \#att. & \#sub & \#opt & total & stage 2 & stage 3 & \#att. & \#sub \\
      \midrule
      \multicolumn{13}{@{}l}{chain length: 2}\\
      $B = 1$ &  \num{24} & \num{654.49} & \SI{0.36}{\percent} & \SI{0.11}{\percent} & \num{4.7} & \num{11.3}  &  \num{25} & \num{661.56} & \SI{0.34}{\percent} & \SI{0.05}{\percent} & \num{4.9} & \num{12.4} \\
      $B = 2$ &  \num{10} & \num{2199.59} & \SI{0.17}{\percent} & \SI{0.06}{\percent} & \num{6.8} & \num{23.1}  &  \num{9} & \num{2239.02} & \SI{0.15}{\percent} & \SI{0.05}{\percent} & \num{7.3} & \num{34.3} \\
      $B = 3$ &  \num{5} & \num{2765.90} & \SI{1.20}{\percent} & \SI{0.21}{\percent} & \num{8.4} & \num{176.6}  &  \num{2} & \num{3445.56} & \SI{0.10}{\percent} & \SI{2.59}{\percent} & \num{16.0} & \num{88.5} \\
      $B = 4$ &  \num{1} & \num{3568.64} & \SI{1.16}{\percent} & \SI{0.15}{\percent} & \num{7.0} & \num{68.0}  &  \num{0} & \num{3600.00} & \SI{0.09}{\percent} & \SI{0.03}{\percent} & --- & --- \\
      \midrule
      \multicolumn{13}{@{}l}{chain length: 3}\\
      $B = 1$ &  \num{15} & \num{1091.93} & \SI{0.40}{\percent} & \SI{0.06}{\percent} & \num{3.9} & \num{8.7}  &  \num{14} & \num{1169.76} & \SI{0.19}{\percent} & \SI{0.03}{\percent} & \num{3.9} & \num{8.6} \\
      $B = 2$ &  \num{1} & \num{3435.58} & \SI{0.08}{\percent} & \SI{0.06}{\percent} & \num{4.0} & \num{18.0}  &  \num{0} & \num{3600.00} & \SI{0.07}{\percent} & \SI{0.02}{\percent} & --- & --- \\
      $B = 3$ &  \num{0} & \num{3600.00} & \SI{0.07}{\percent} & \SI{0.06}{\percent} & --- & ---  &  \num{0} & \num{3600.00} & \SI{0.06}{\percent} & \SI{0.02}{\percent} & --- & --- \\
      $B = 4$ &  \num{0} & \num{3600.00} & \SI{0.09}{\percent} & \SI{0.10}{\percent} & --- & ---  &  \num{0} & \num{3600.00} & \SI{0.05}{\percent} & \SI{0.01}{\percent} & --- & --- \\
      \midrule
      \multicolumn{13}{@{}l}{chain length: 4}\\
      $B = 1$ &  \num{6} & \num{2415.62} & \SI{1.89}{\percent} & \SI{0.04}{\percent} & \num{4.0} & \num{9.0}  &  \num{6} & \num{2563.12} & \SI{0.07}{\percent} & \SI{0.01}{\percent} & \num{4.0} & \num{9.5} \\
      $B = 2$ &  \num{0} & \num{3600.00} & \SI{0.07}{\percent} & \SI{0.04}{\percent} & --- & ---  &  \num{0} & \num{3600.00} & \SI{0.04}{\percent} & \SI{0.01}{\percent} & --- & --- \\
      $B = 3$ &  \num{0} & \num{3600.00} & \SI{0.09}{\percent} & \SI{0.06}{\percent} & --- & ---  &  \num{0} & \num{3600.00} & \SI{0.04}{\percent} & \SI{0.02}{\percent} & --- & --- \\
      $B = 4$ &  \num{0} & \num{3600.00} & \SI{0.12}{\percent} & \SI{0.06}{\percent} & --- & ---  &  \num{0} & \num{3600.00} & \SI{0.04}{\percent} & \SI{0.02}{\percent} & --- & --- \\
      \bottomrule
    \end{tabular*}
    \caption{Impact of cut lifting restricted to instances with 100 pairs / NDDs, for cut-CC and cut-PICEF, given maximum cycle length $K = 3$ and different chain lengths, for Full Recourse (FR) and FSE}
  \end{tiny}
\end{table}

%% file: Tables/AggregatedLifting/combined_aggr_lifting_table-cyc3-policy1.tex
\begin{table}[htbp]
  \begin{scriptsize}
    
    \label{tab:combined-aggrlifting-cyc3-policy1}
    \begin{tabular*}{\textwidth}{@{}l@{\;\;\extracolsep{\fill}}rrrrrrrr@{}}\toprule
      {\bfseries Cut-CC} & \multicolumn{4}{c}{with lifting} & \multicolumn{4}{c}{without lifting}\\
      \cmidrule{2-5} \cmidrule{6-9}
      \#vertices & \#opt & time & \#att. & \#sub & \#opt & time & \#att. & \#sub \\
      \midrule
      \multicolumn{9}{@{}l}{chain length: 2}\\
       20 &  120 & 0.1 & 2.8 & 8.0 &  120 & 0.1 & 3.0 & 10.0\\
       50 &  120 & 7.3 & 5.1 & 41.7 &  120 & 17.8 & 5.3 & 113.1\\
      100 &  120 & 132.3 & 5.6 & 104.5 &  107 & 329.8 & 5.6 & 217.2\\
      \midrule
      \multicolumn{9}{@{}l}{chain length: 3}\\
       20 &  120 & 0.2 & 2.8 & 8.4 &  120 & 0.3 & 3.0 & 10.8\\
       50 &  120 & 29.6 & 5.4 & 45.4 &  119 & 50.6 & 5.5 & 126.6\\
      100 &  88 & 1382.8 & 5.9 & 98.6 &  60 & 1765.7 & 5.9 & 215.8\\
      \midrule
      \multicolumn{9}{@{}l}{chain length: 4}\\
       20 &  120 & 0.8 & 2.9 & 8.8 &  120 & 0.8 & 2.9 & 11.3\\
       50 &  118 & 247.9 & 5.4 & 50.7 &  111 & 376.3 & 5.8 & 127.8\\
      100 &  8 & 3364.8 & 5.5 & 66.6 &  4 & 3490.3 & 2.7 & 22.6\\
      \midrule
      {\bfseries Cut-PICEF} & \multicolumn{4}{c}{with lifting} & \multicolumn{4}{c}{without lifting}\\
      \cmidrule{2-5} \cmidrule{6-9}
      \#vertices & \#opt & time & \#att. & \#sub & \#opt & time & \#att. & \#sub \\
      \midrule
      \multicolumn{9}{@{}l}{chain length: 2}\\
       20 &  120 & 0.2 & 2.7 & 7.9 &  120 & 0.1 & 2.8 & 9.3\\
       50 &  120 & 11.6 & 5.3 & 35.4 &  118 & 24.5 & 5.2 & 102.6\\
      100 &  119 & 269.4 & 5.7 & 84.8 &  103 & 488.6 & 6.2 & 199.0\\
      \midrule
      \multicolumn{9}{@{}l}{chain length: 3}\\
       20 &  120 & 0.3 & 2.8 & 7.8 &  120 & 0.3 & 2.9 & 10.0\\
       50 &  120 & 27.4 & 4.9 & 40.8 &  118 & 35.7 & 5.1 & 97.5\\
      100 &  114 & 507.0 & 5.6 & 79.2 &  93 & 955.3 & 6.1 & 195.9\\
      \midrule
      \multicolumn{9}{@{}l}{chain length: 4}\\
       20 &  120 & 0.5 & 2.7 & 7.7 &  120 & 0.4 & 2.9 & 10.8\\
       50 &  120 & 48.7 & 4.9 & 43.0 &  117 & 67.8 & 5.1 & 110.8\\
      100 &  106 & 820.4 & 5.6 & 67.4 &  81 & 1465.5 & 5.8 & 192.6\\
      \bottomrule
    \end{tabular*}
    \caption{Comparison of the cut-CC method with and without lifting for maximum cycle length 3 and different chain lengths, for the Full Recourse policy.}
  \end{scriptsize}
\end{table}

%% file: ManuscriptMay2022/Conclusion.tex
\section{Conclusion}\label{sec:conclusion}

In this paper, we investigate a robust optimization variant of the kidney exchange problem. In this setting, donors and/or recipients might decide to drop out of the program after the proposal of exchanges and before the actual transplants take place, due to a variety of reasons. 
This is a practical problem since donors and pairs leaving a kidney exchange program lead to broken cycles and chains of initially proposed transplants. 

We propose a cutting plane algorithm for solving the attacker-defender subproblem based on the cycle-chain formulation and the position-indexed chain-edge formulation. The algorithms are based on a reformulation of the attacker-defender subproblem with constraints for any feasible kidney exchange solution.
Cuts are separated by solving the problem of finding the optimal recourse solution given an initial kidney exchange solution and attack.
These algorithms are implemented for two recourse policies, the Full Recourse policy introduced in~\cite{Glorie20} and the Fix Successful Exchanges policy.
The latter aims to give additional guarantees to participating recipients in case none of the involved people of an exchange drop out of the kidney exchange program. 
Furthermore, a lifting technique is deployed to obtain solutions on the entire compatibility graph that still have maximum recourse value on the non-attacked cycles and chains. 

Our computational results show that both the CC- and PICEF-based implementations (cut-CC and cut-PICEF) of our cutting plane algorithm outperform the state-of-the-art in the Full Recourse setting, which is a branch-and-bound type algorithm proposed in~\cite{Glorie20} when we consider benchmark instances with $|V| \ge 50$. 
We observe a transition phase between cut-CC and cut-PICEF.
Whenever we allow only small chains, cut-CC outperforms cut-PICEF.
However, when larger chains are allowed, cut-PICEF becomes the stronger method and the benefit of having a polynomial-size formulation in terms of chains becomes strongly apparent.
In the FSE setting, cut-CC and cut-PICEF can solve most instances with up to 50 instances and small attack budgets $B$, rendering the problem tractable for small kidney exchanges but challenging for mid-size to large pools in practice.

Enabling the cut lifting technique leads, under most choices for parameter settings, to a significant reduction in the running time under most choices for parameter settings.
The most apparent reason for this is that fewer iterations of the separation problem are needed to solve the attacker-defender subproblem.
Since they give better feedback on the recourse value that can be realized given a set of attacked vertices.
As a result, the generated attacks are more effective, and typically, we need fewer attacks to solve the entire robust optimization model to optimality.

The success of cut lifting highlights an interesting aspect of cutting plane approaches to multi-level problems. 
Whenever multiple valid cuts are violated in an iteration, the decision of which cut to add can have a significant impact on runtime. 
The variability in the number of iterations required between cut-CC and cut-PICEF, where different but in essence arbitrarily chosen alternate optimal solutions to the recourse problem lead to large discrepancies in the number of iterations required. 
Studying these differences may lead to insights and tie-breaking rules to further speed up computation.
One possible direction for future research might be to exploit the hierarchical structure of the set of cycles and chains in a directed graph.
This means that instead of one term per cycle or chain in an interdiction cut, we might also consider terms corresponding to subcycles or subchains. This leads to optimal solutions in which the interdiction budget is spent more efficiently.

A second avenue for further work is to apply our techniques to a setting with a more realistic uncertainty set.
A more sophisticated data analysis of the physiological properties of donors and recipients may be necessary to find an uncertainty set that more accurately approximates the real set of attacks that may occur.

\section*{Acknowledgements}
We would like to thank Xenia Klimentova for sharing kidney exchange instance data.
Declarations of interest: none.

\section*{Authorship contribution}
{\bfseries Danny Blom:} Conceptualization, Methodology, Software, Writing.
{\bfseries Christopher Hojny:} Conceptualization, Methodology, Software, Writing.
{\bfseries Bart Smeulders:} Conceptualization, Methodology, Writing.

%% file: ManuscriptMay2022/Appendix.tex
\setcounter{table}{0}
\renewcommand{\thetable}{\Alph{section}.\arabic{table}}

\section{Implementation of the \cite{Glorie20} algorithm}\label{sec:appendix-bandb}
In this appendix, we briefly describe our implementation of the
branch-and-bound algorithm by \cite{Glorie20} used to solve the combination of stage 2 and~3 in their iterative solution framework.
In particular, in this appendix, we lay out the choices we made where decisions are not
fully specified in Algorithm~2 in their original paper.
We matched these decisions as close as possible to add details on the \cite{Glorie20} implementation received through personal communications.
References to lines are as they appear in Algorithm 2 suggested by \cite{Glorie20}.

\paragraph{Line 8: Node selection}
We deploy a depth-first search strategy to select the next subproblem with
priority to the attack branch.
In other words, if we generated two children nodes~$t_0$ and~$t_1$ of branch-and-bound
node~$t$, which extend the branching decisions at~$t$ by selecting a new
vertex to be attacked and not to be attacked, respectively, we first explore the subtree rooted at~$t_0$ before proceeding with the subtree rooted at~$t_1$.

\paragraph{Line 9: Fill for maximal attack}
When choosing the maximal attack, we add vertices to be attacked sequentially.
For each new node, we first identify the highest weight cycle or chain for which no vertices are attacked yet, and for which there is a vertex that is still unfixed.
Ties between cycles and chains are broken arbitrarily.
For this cycle or chain, we add the lowest index vertex to the attack.
If no such cycles or chains exist, i.e., at least one vertex is already attacked, or all vertices are fixed to be attacked or not attacked, we add the lowest index vertex that is not yet fixed to the attack.

\paragraph{Line 25: Branching}
We choose the first vertex we added to the attack to branch on.

\section{Computational results for $\cyclen = 4$}\label{sec:appendix-compresults-K=4}

\input{./Tables/MethodComparison/comparison-cyc4-policy1.tex}

\input{./Tables/MethodComparison/comparison-cyc4-policy2.tex}

\input{./Tables/LiftingTables/combined-lifting-table-cyc4.tex}

\input{./Tables/AggregatedLifting/combined_aggr_lifting_table-cyc4-policy1.tex}


%% file: Tables/MethodComparison/comparison-cyc4-policy1.tex
\begin{landscape}
    \begin{table}[htbp]
    \begin{tiny}
      
      \label{tab:cyc4-policy1}
      \begin{tabular*}{\linewidth}{@{}l@{\;\;\extracolsep{\fill}}rrrrrrrrrrrrrrrrr@{}}\toprule
           & \multicolumn{5}{c}{Branch-and-Bound} & \multicolumn{6}{c}{cut-CC} & \multicolumn{6}{c}{cut-PICEF}\\
          \cmidrule{2-6} \cmidrule{7-12} \cmidrule{13-18}
          & & \multicolumn{2}{c}{time} & & & & \multicolumn{3}{c}{time} & & & & \multicolumn{3}{c}{time} & & \\
          \cmidrule{3-4} \cmidrule{8-10} \cmidrule{14-16}
          \#vertices & \#opt & total & stage 2 & \#att. & \#B\&B-nodes & \#opt & total & stage 2 & stage 3 & \#att. & \#sub & \#opt & total & stage 2 & stage 3 & \#att. & \#sub \\
          \midrule
          \multicolumn{18}{@{}l}{chain length: 2}\\
          \multicolumn{18}{@{}l}{$B = 1$}\\
           20 &  \num{ 30} & \num{ 0.19} & \SI{27.39}{\percent} & \num{ 1.9} & \num{ 15.0}  &  \num{ 30} & \num{ 0.15} & \SI{24.96}{\percent} & \SI{17.75}{\percent} & \num{ 1.8} & \num{ 7.7}  &  \num{ 30} & \num{ 0.24} & \SI{16.77}{\percent} & \SI{24.50}{\percent} & \num{ 2.0} & \num{ 7.3} \\
           50 &  \num{ 30} & \num{ 18.41} & \SI{34.78}{\percent} & \num{ 3.0} & \num{ 55.5}  &  \num{ 30} & \num{ 16.03} & \SI{8.32}{\percent} & \SI{26.97}{\percent} & \num{ 2.9} & \num{ 22.7}  &  \num{ 30} & \num{ 18.13} & \SI{7.37}{\percent} & \SI{24.84}{\percent} & \num{ 3.0} & \num{ 21.2} \\
          100 &  \num{ 23} & \num{1078.88} & \SI{50.34}{\percent} & \num{ 3.3} & \num{ 115.0}  &  \num{ 26} & \num{ 893.61} & \SI{5.14}{\percent} & \SI{37.28}{\percent} & \num{ 3.5} & \num{ 45.3}  &  \num{ 26} & \num{ 918.91} & \SI{5.62}{\percent} & \SI{43.40}{\percent} & \num{ 3.5} & \num{ 42.1} \\
          \multicolumn{18}{@{}l}{$B = 2$}\\
           20 &  \num{ 30} & \num{ 0.61} & \SI{48.54}{\percent} & \num{ 3.2} & \num{ 97.2}  &  \num{ 30} & \num{ 0.40} & \SI{26.39}{\percent} & \SI{20.56}{\percent} & \num{ 2.9} & \num{ 15.2}  &  \num{ 30} & \num{ 0.44} & \SI{23.22}{\percent} & \SI{16.24}{\percent} & \num{ 3.0} & \num{ 14.5} \\
           50 &  \num{ 30} & \num{ 85.99} & \SI{71.79}{\percent} & \num{ 5.1} & \num{ 833.7}  &  \num{ 30} & \num{ 50.99} & \SI{25.07}{\percent} & \SI{27.72}{\percent} & \num{ 5.0} & \num{ 101.5}  &  \num{ 30} & \num{ 51.56} & \SI{22.87}{\percent} & \SI{25.41}{\percent} & \num{ 5.0} & \num{ 79.8} \\
          100 &  \num{ 4} & \num{3343.15} & \SI{64.84}{\percent} & \num{ 7.8} & \num{ 2139.2}  &  \num{ 15} & \num{1904.66} & \SI{7.79}{\percent} & \SI{41.12}{\percent} & \num{ 4.9} & \num{ 147.7}  &  \num{ 15} & \num{2154.89} & \SI{8.62}{\percent} & \SI{44.55}{\percent} & \num{ 5.3} & \num{ 142.7} \\
          \multicolumn{18}{@{}l}{$B = 3$}\\
           20 &  \num{ 30} & \num{ 1.33} & \SI{63.32}{\percent} & \num{ 3.8} & \num{ 324.0}  &  \num{ 30} & \num{ 0.55} & \SI{31.08}{\percent} & \SI{13.95}{\percent} & \num{ 3.6} & \num{ 16.1}  &  \num{ 30} & \num{ 0.45} & \SI{34.28}{\percent} & \SI{19.87}{\percent} & \num{ 3.2} & \num{ 14.8} \\
           50 &  \num{ 29} & \num{ 727.97} & \SI{90.48}{\percent} & \num{ 7.4} & \num{12045.6}  &  \num{ 30} & \num{ 82.93} & \SI{38.76}{\percent} & \SI{22.90}{\percent} & \num{ 6.1} & \num{ 207.3}  &  \num{ 30} & \num{ 139.06} & \SI{43.64}{\percent} & \SI{18.97}{\percent} & \num{ 6.4} & \num{ 240.4} \\
          100 &  \num{ 0} & \num{3600.00} & \SI{60.64}{\percent} & --- & ---  &  \num{ 8} & \num{2972.21} & \SI{13.30}{\percent} & \SI{30.27}{\percent} & \num{ 5.9} & \num{ 333.6}  &  \num{ 6} & \num{3254.15} & \SI{11.60}{\percent} & \SI{38.67}{\percent} & \num{ 7.8} & \num{ 312.8} \\
          \multicolumn{18}{@{}l}{$B = 4$}\\
           20 &  \num{ 30} & \num{ 2.82} & \SI{62.26}{\percent} & \num{ 4.8} & \num{ 964.1}  &  \num{ 30} & \num{ 0.38} & \SI{33.03}{\percent} & \SI{14.29}{\percent} & \num{ 3.3} & \num{ 12.9}  &  \num{ 30} & \num{ 0.38} & \SI{36.45}{\percent} & \SI{15.05}{\percent} & \num{ 3.3} & \num{ 13.3} \\
           50 &  \num{ 12} & \num{2412.60} & \SI{95.11}{\percent} & \num{12.7} & \num{63189.2}  &  \num{ 28} & \num{ 164.12} & \SI{50.80}{\percent} & \SI{15.87}{\percent} & \num{ 7.8} & \num{ 273.8}  &  \num{ 27} & \num{ 187.50} & \SI{51.39}{\percent} & \SI{13.60}{\percent} & \num{ 7.6} & \num{ 227.5} \\
          100 &  \num{ 0} & \num{3600.00} & \SI{60.86}{\percent} & --- & ---  &  \num{ 3} & \num{3410.61} & \SI{19.56}{\percent} & \SI{24.22}{\percent} & \num{ 9.0} & \num{ 534.7}  &  \num{ 1} & \num{3572.13} & \SI{16.26}{\percent} & \SI{25.13}{\percent} & \num{17.0} & \num{ 872.0} \\
          \midrule
          \multicolumn{18}{@{}l}{chain length: 3}\\
          \multicolumn{18}{@{}l}{$B = 1$}\\
           20 &  \num{ 30} & \num{ 0.34} & \SI{32.74}{\percent} & \num{ 1.9} & \num{ 15.8}  &  \num{ 30} & \num{ 0.26} & \SI{27.11}{\percent} & \SI{13.20}{\percent} & \num{ 1.9} & \num{ 8.2}  &  \num{ 30} & \num{ 0.35} & \SI{11.59}{\percent} & \SI{21.65}{\percent} & \num{ 1.9} & \num{ 7.4} \\
           50 &  \num{ 30} & \num{ 34.18} & \SI{30.98}{\percent} & \num{ 2.8} & \num{ 51.3}  &  \num{ 30} & \num{ 37.67} & \SI{7.10}{\percent} & \SI{23.35}{\percent} & \num{ 3.1} & \num{ 24.5}  &  \num{ 30} & \num{ 33.20} & \SI{7.30}{\percent} & \SI{35.49}{\percent} & \num{ 3.2} & \num{ 22.8} \\
          100 &  \num{ 19} & \num{1883.69} & \SI{48.53}{\percent} & \num{ 3.2} & \num{ 97.1}  &  \num{ 21} & \num{1708.52} & \SI{6.25}{\percent} & \SI{44.78}{\percent} & \num{ 3.0} & \num{ 45.8}  &  \num{ 27} & \num{1142.96} & \SI{6.47}{\percent} & \SI{52.07}{\percent} & \num{ 3.2} & \num{ 38.7} \\
          \multicolumn{18}{@{}l}{$B = 2$}\\
           20 &  \num{ 30} & \num{ 0.83} & \SI{54.29}{\percent} & \num{ 3.1} & \num{ 107.9}  &  \num{ 30} & \num{ 0.57} & \SI{32.87}{\percent} & \SI{13.25}{\percent} & \num{ 2.8} & \num{ 16.4}  &  \num{ 30} & \num{ 0.64} & \SI{27.73}{\percent} & \SI{17.01}{\percent} & \num{ 3.0} & \num{ 14.0} \\
           50 &  \num{ 30} & \num{ 186.74} & \SI{67.81}{\percent} & \num{ 5.0} & \num{ 925.4}  &  \num{ 30} & \num{ 100.92} & \SI{18.67}{\percent} & \SI{30.48}{\percent} & \num{ 4.8} & \num{ 114.1}  &  \num{ 30} & \num{ 98.34} & \SI{16.99}{\percent} & \SI{40.47}{\percent} & \num{ 5.0} & \num{ 90.1} \\
          100 &  \num{ 1} & \num{3536.53} & \SI{59.65}{\percent} & \num{10.0} & \num{ 2343.0}  &  \num{ 7} & \num{3107.82} & \SI{5.38}{\percent} & \SI{29.00}{\percent} & \num{ 5.1} & \num{ 113.4}  &  \num{ 10} & \num{2824.96} & \SI{6.82}{\percent} & \SI{48.93}{\percent} & \num{ 6.0} & \num{ 133.5} \\
          \multicolumn{18}{@{}l}{$B = 3$}\\
           20 &  \num{ 30} & \num{ 1.68} & \SI{57.09}{\percent} & \num{ 3.7} & \num{ 374.5}  &  \num{ 30} & \num{ 0.57} & \SI{28.76}{\percent} & \SI{15.65}{\percent} & \num{ 3.5} & \num{ 15.6}  &  \num{ 30} & \num{ 0.72} & \SI{28.46}{\percent} & \SI{16.55}{\percent} & \num{ 3.5} & \num{ 13.1} \\
           50 &  \num{ 21} & \num{1219.48} & \SI{89.27}{\percent} & \num{ 7.5} & \num{11698.7}  &  \num{ 30} & \num{ 180.85} & \SI{30.49}{\percent} & \SI{25.76}{\percent} & \num{ 6.3} & \num{ 225.0}  &  \num{ 29} & \num{ 169.80} & \SI{31.81}{\percent} & \SI{29.60}{\percent} & \num{ 6.3} & \num{ 195.4} \\
          100 &  \num{ 0} & \num{3600.00} & \SI{41.74}{\percent} & --- & ---  &  \num{ 1} & \num{3533.76} & \SI{5.03}{\percent} & \SI{22.44}{\percent} & \num{12.0} & \num{ 366.0}  &  \num{ 3} & \num{3453.67} & \SI{9.41}{\percent} & \SI{46.40}{\percent} & \num{ 9.0} & \num{ 215.3} \\
          \multicolumn{18}{@{}l}{$B = 4$}\\
           20 &  \num{ 30} & \num{ 3.59} & \SI{57.30}{\percent} & \num{ 4.9} & \num{ 1071.7}  &  \num{ 30} & \num{ 0.59} & \SI{26.99}{\percent} & \SI{11.21}{\percent} & \num{ 3.6} & \num{ 13.2}  &  \num{ 30} & \num{ 0.56} & \SI{31.99}{\percent} & \SI{16.11}{\percent} & \num{ 3.7} & \num{ 12.6} \\
           50 &  \num{ 8} & \num{2872.08} & \SI{92.19}{\percent} & \num{12.8} & \num{45939.2}  &  \num{ 27} & \num{ 228.29} & \SI{44.26}{\percent} & \SI{19.61}{\percent} & \num{ 6.9} & \num{ 286.0}  &  \num{ 27} & \num{ 261.47} & \SI{43.44}{\percent} & \SI{21.69}{\percent} & \num{ 7.4} & \num{ 246.6} \\
          100 &  \num{ 0} & \num{3600.00} & \SI{44.48}{\percent} & --- & ---  &  \num{ 1} & \num{3568.59} & \SI{6.95}{\percent} & \SI{17.52}{\percent} & \num{15.0} & \num{ 825.0}  &  \num{ 0} & \num{3600.00} & \SI{11.48}{\percent} & \SI{42.81}{\percent} & --- & --- \\
          \midrule
          \multicolumn{18}{@{}l}{chain length: 4}\\
          \multicolumn{18}{@{}l}{$B = 1$}\\
           20 &  \num{ 30} & \num{ 0.77} & \SI{24.06}{\percent} & \num{ 1.9} & \num{ 16.3}  &  \num{ 30} & \num{ 0.69} & \SI{15.61}{\percent} & \SI{19.42}{\percent} & \num{ 2.0} & \num{ 8.6}  &  \num{ 30} & \num{ 0.56} & \SI{22.35}{\percent} & \SI{17.07}{\percent} & \num{ 2.1} & \num{ 7.3} \\
           50 &  \num{ 30} & \num{ 209.70} & \SI{27.92}{\percent} & \num{ 3.2} & \num{ 59.2}  &  \num{ 30} & \num{ 229.91} & \SI{5.72}{\percent} & \SI{23.29}{\percent} & \num{ 3.4} & \num{ 27.5}  &  \num{ 30} & \num{ 46.88} & \SI{7.09}{\percent} & \SI{39.38}{\percent} & \num{ 2.9} & \num{ 22.4} \\
          100 &  \num{ 1} & \num{3545.41} & \SI{18.53}{\percent} & \num{ 4.0} & \num{ 162.0}  &  \num{ 2} & \num{3327.35} & \SI{9.59}{\percent} & \SI{11.08}{\percent} & \num{ 3.0} & \num{ 55.5}  &  \num{ 24} & \num{1399.71} & \SI{5.73}{\percent} & \SI{60.70}{\percent} & \num{ 2.9} & \num{ 42.4} \\
          \multicolumn{18}{@{}l}{$B = 2$}\\
           20 &  \num{ 30} & \num{ 1.34} & \SI{49.20}{\percent} & \num{ 3.1} & \num{ 117.3}  &  \num{ 30} & \num{ 1.13} & \SI{29.00}{\percent} & \SI{11.08}{\percent} & \num{ 2.9} & \num{ 16.4}  &  \num{ 30} & \num{ 0.89} & \SI{29.88}{\percent} & \SI{15.96}{\percent} & \num{ 3.1} & \num{ 13.8} \\
           50 &  \num{ 25} & \num{ 761.50} & \SI{62.48}{\percent} & \num{ 5.1} & \num{ 905.8}  &  \num{ 28} & \num{ 399.60} & \SI{11.82}{\percent} & \SI{33.90}{\percent} & \num{ 4.5} & \num{ 105.6}  &  \num{ 30} & \num{ 156.36} & \SI{13.13}{\percent} & \SI{43.83}{\percent} & \num{ 5.1} & \num{ 97.3} \\
          100 &  \num{ 0} & \num{3600.00} & \SI{18.82}{\percent} & --- & ---  &  \num{ 0} & \num{3600.00} & \SI{10.37}{\percent} & \SI{5.14}{\percent} & --- & ---  &  \num{ 7} & \num{2970.19} & \SI{6.55}{\percent} & \SI{52.23}{\percent} & \num{ 5.1} & \num{ 123.4} \\
          \multicolumn{18}{@{}l}{$B = 3$}\\
           20 &  \num{ 30} & \num{ 2.99} & \SI{56.37}{\percent} & \num{ 3.9} & \num{ 446.8}  &  \num{ 30} & \num{ 1.21} & \SI{24.97}{\percent} & \SI{12.50}{\percent} & \num{ 3.5} & \num{ 16.1}  &  \num{ 30} & \num{ 0.76} & \SI{25.09}{\percent} & \SI{13.27}{\percent} & \num{ 3.6} & \num{ 15.6} \\
           50 &  \num{ 12} & \num{2254.06} & \SI{78.10}{\percent} & \num{ 8.2} & \num{ 9819.8}  &  \num{ 24} & \num{ 686.29} & \SI{18.03}{\percent} & \SI{30.83}{\percent} & \num{ 6.2} & \num{ 197.5}  &  \num{ 29} & \num{ 295.14} & \SI{28.06}{\percent} & \SI{35.31}{\percent} & \num{ 6.3} & \num{ 199.8} \\
          100 &  \num{ 0} & \num{3600.00} & \SI{15.34}{\percent} & --- & ---  &  \num{ 0} & \num{3600.00} & \SI{9.87}{\percent} & \SI{1.76}{\percent} & --- & ---  &  \num{ 1} & \num{3575.11} & \SI{6.89}{\percent} & \SI{38.18}{\percent} & \num{ 9.0} & \num{ 307.0} \\
          \multicolumn{18}{@{}l}{$B = 4$}\\
           20 &  \num{ 30} & \num{ 5.43} & \SI{57.99}{\percent} & \num{ 5.2} & \num{ 1479.1}  &  \num{ 30} & \num{ 0.93} & \SI{23.59}{\percent} & \SI{9.04}{\percent} & \num{ 3.3} & \num{ 13.8}  &  \num{ 30} & \num{ 0.75} & \SI{32.61}{\percent} & \SI{10.70}{\percent} & \num{ 3.4} & \num{ 12.7} \\
           50 &  \num{ 3} & \num{3325.58} & \SI{82.22}{\percent} & \num{15.7} & \num{53241.3}  &  \num{ 21} & \num{ 779.12} & \SI{23.65}{\percent} & \SI{26.43}{\percent} & \num{ 6.8} & \num{ 249.0}  &  \num{ 27} & \num{ 370.75} & \SI{41.40}{\percent} & \SI{24.96}{\percent} & \num{ 7.3} & \num{ 279.9} \\
          100 &  \num{ 0} & \num{3600.00} & \SI{18.90}{\percent} & --- & ---  &  \num{ 0} & \num{3600.00} & \SI{5.48}{\percent} & \SI{6.40}{\percent} & --- & ---  &  \num{ 0} & \num{3600.00} & \SI{10.26}{\percent} & \SI{36.56}{\percent} & --- & --- \\
          \bottomrule
    \end{tabular*}
    \caption{Comparison of the different methods for maximum cycle length 4 and different chain lengths, for the Full Recourse policy.}
    \end{tiny}
  \end{table}
\end{landscape}

%% file: Tables/MethodComparison/comparison-cyc4-policy2.tex
  \begin{landscape}
    \begin{table}[htbp]
      \begin{tiny}
       
        \label{tab:cyc4-policy2}
        \begin{tabular*}{\linewidth}{@{}l@{\;\;\extracolsep{\fill}}rrrrrrrrrrrr@{}}\toprule
           & \multicolumn{6}{c}{cut-CC} & \multicolumn{6}{c}{cut-PICEF}\\
          \cmidrule{2-7} \cmidrule{8-13}
          & & \multicolumn{3}{c}{time} & & & & \multicolumn{3}{c}{time} & & \\
          \cmidrule{3-5} \cmidrule{9-11}
          \#vertices & \#opt & total & stage 2 & stage 3 & \#att. & \#sub & \#opt & total & stage 2 & stage 3 & \#att. & \#sub \\
          \midrule
          \multicolumn{13}{@{}l}{chain length: 2}\\
          \multicolumn{13}{@{}l}{$B = 1$}\\
           20 &  \num{ 30} & \num{ 0.99} & \SI{31.88}{\percent} & \SI{8.01}{\percent} & \num{ 3.0} & \num{ 7.9}  &  \num{ 30} & \num{ 1.52} & \SI{16.23}{\percent} & \SI{12.21}{\percent} & \num{ 2.8} & \num{ 7.2} \\
           50 &  \num{ 25} & \num{ 342.60} & \SI{11.85}{\percent} & \SI{0.51}{\percent} & \num{ 7.7} & \num{ 35.0}  &  \num{ 23} & \num{ 425.01} & \SI{0.50}{\percent} & \SI{0.20}{\percent} & \num{ 5.7} & \num{ 19.5} \\
          100 &  \num{ 2} & \num{3481.47} & \SI{21.97}{\percent} & \SI{0.36}{\percent} & \num{ 6.5} & \num{ 27.5}  &  \num{ 3} & \num{3494.44} & \SI{0.32}{\percent} & \SI{0.11}{\percent} & \num{ 6.3} & \num{ 29.7} \\
          \multicolumn{13}{@{}l}{$B = 2$}\\
           20 &  \num{ 30} & \num{ 2.83} & \SI{35.18}{\percent} & \SI{6.30}{\percent} & \num{ 4.3} & \num{ 21.4}  &  \num{ 30} & \num{ 2.43} & \SI{20.22}{\percent} & \SI{8.33}{\percent} & \num{ 4.4} & \num{ 21.7} \\
           50 &  \num{ 16} & \num{ 978.54} & \SI{14.67}{\percent} & \SI{0.83}{\percent} & \num{10.1} & \num{ 104.1}  &  \num{ 14} & \num{1235.50} & \SI{1.15}{\percent} & \SI{0.34}{\percent} & \num{ 9.1} & \num{ 89.9} \\
          100 &  \num{ 0} & \num{3600.00} & \SI{15.30}{\percent} & \SI{0.31}{\percent} & --- & ---  &  \num{ 0} & \num{3600.00} & \SI{0.34}{\percent} & \SI{0.10}{\percent} & --- & --- \\
          \multicolumn{13}{@{}l}{$B = 3$}\\
           20 &  \num{ 30} & \num{ 1.89} & \SI{48.14}{\percent} & \SI{6.34}{\percent} & \num{ 4.1} & \num{ 20.9}  &  \num{ 30} & \num{ 2.30} & \SI{19.62}{\percent} & \SI{16.53}{\percent} & \num{ 4.3} & \num{ 30.6} \\
           50 &  \num{ 9} & \num{1983.83} & \SI{24.70}{\percent} & \SI{0.84}{\percent} & \num{10.9} & \num{ 226.3}  &  \num{ 11} & \num{1669.92} & \SI{5.09}{\percent} & \SI{0.70}{\percent} & \num{10.4} & \num{ 274.7} \\
          100 &  \num{ 0} & \num{3600.00} & \SI{15.24}{\percent} & \SI{0.21}{\percent} & --- & ---  &  \num{ 0} & \num{3600.00} & \SI{0.29}{\percent} & \SI{0.07}{\percent} & --- & --- \\
          \multicolumn{13}{@{}l}{$B = 4$}\\
           20 &  \num{ 30} & \num{ 1.46} & \SI{42.06}{\percent} & \SI{5.83}{\percent} & \num{ 4.1} & \num{ 22.2}  &  \num{ 30} & \num{ 1.43} & \SI{21.99}{\percent} & \SI{9.85}{\percent} & \num{ 3.8} & \num{ 20.2} \\
           50 &  \num{ 7} & \num{2271.96} & \SI{40.37}{\percent} & \SI{0.91}{\percent} & \num{14.1} & \num{ 304.0}  &  \num{ 8} & \num{2039.06} & \SI{6.18}{\percent} & \SI{0.60}{\percent} & \num{13.4} & \num{ 339.6} \\
          100 &  \num{ 0} & \num{3600.00} & \SI{15.20}{\percent} & \SI{0.12}{\percent} & --- & ---  &  \num{ 0} & \num{3600.00} & \SI{0.33}{\percent} & \SI{0.10}{\percent} & --- & --- \\
          \midrule
          \multicolumn{13}{@{}l}{chain length: 3}\\
          \multicolumn{13}{@{}l}{$B = 1$}\\
           20 &  \num{ 30} & \num{ 2.24} & \SI{28.40}{\percent} & \SI{8.36}{\percent} & \num{ 3.3} & \num{ 9.9}  &  \num{ 30} & \num{ 1.90} & \SI{21.59}{\percent} & \SI{6.46}{\percent} & \num{ 2.8} & \num{ 7.8} \\
           50 &  \num{ 25} & \num{ 525.96} & \SI{14.69}{\percent} & \SI{0.48}{\percent} & \num{ 7.1} & \num{ 30.0}  &  \num{ 20} & \num{ 604.26} & \SI{0.34}{\percent} & \SI{0.13}{\percent} & \num{ 5.1} & \num{ 20.1} \\
          100 &  \num{ 0} & \num{3600.00} & \SI{32.01}{\percent} & \SI{0.38}{\percent} & --- & ---  &  \num{ 2} & \num{3444.94} & \SI{0.29}{\percent} & \SI{0.08}{\percent} & \num{ 4.5} & \num{ 19.0} \\
          \multicolumn{13}{@{}l}{$B = 2$}\\
           20 &  \num{ 30} & \num{ 4.10} & \SI{33.52}{\percent} & \SI{9.66}{\percent} & \num{ 4.9} & \num{ 21.5}  &  \num{ 30} & \num{ 4.54} & \SI{17.08}{\percent} & \SI{4.25}{\percent} & \num{ 4.7} & \num{ 19.9} \\
           50 &  \num{ 12} & \num{1738.12} & \SI{16.03}{\percent} & \SI{0.69}{\percent} & \num{10.2} & \num{ 103.0}  &  \num{ 11} & \num{2003.66} & \SI{0.37}{\percent} & \SI{0.11}{\percent} & \num{ 9.1} & \num{ 70.6} \\
          100 &  \num{ 0} & \num{3600.00} & \SI{19.59}{\percent} & \SI{0.20}{\percent} & --- & ---  &  \num{ 0} & \num{3600.00} & \SI{0.23}{\percent} & \SI{0.45}{\percent} & --- & --- \\
          \multicolumn{13}{@{}l}{$B = 3$}\\
           20 &  \num{ 30} & \num{ 3.42} & \SI{43.68}{\percent} & \SI{15.84}{\percent} & \num{ 4.1} & \num{ 24.5}  &  \num{ 30} & \num{ 3.61} & \SI{24.62}{\percent} & \SI{8.44}{\percent} & \num{ 3.8} & \num{ 25.4} \\
           50 &  \num{ 7} & \num{2464.55} & \SI{26.10}{\percent} & \SI{0.71}{\percent} & \num{11.1} & \num{ 183.3}  &  \num{ 7} & \num{2326.23} & \SI{1.61}{\percent} & \SI{0.30}{\percent} & \num{10.1} & \num{ 197.4} \\
          100 &  \num{ 0} & \num{3600.00} & \SI{21.03}{\percent} & \SI{0.16}{\percent} & --- & ---  &  \num{ 0} & \num{3600.00} & \SI{0.22}{\percent} & \SI{0.03}{\percent} & --- & --- \\
          \multicolumn{13}{@{}l}{$B = 4$}\\
           20 &  \num{ 30} & \num{ 2.27} & \SI{41.98}{\percent} & \SI{10.27}{\percent} & \num{ 4.1} & \num{ 21.7}  &  \num{ 30} & \num{ 2.83} & \SI{28.60}{\percent} & \SI{11.29}{\percent} & \num{ 4.0} & \num{ 20.7} \\
           50 &  \num{ 6} & \num{2507.32} & \SI{35.23}{\percent} & \SI{0.67}{\percent} & \num{10.8} & \num{ 156.8}  &  \num{ 6} & \num{2278.47} & \SI{1.28}{\percent} & \SI{0.43}{\percent} & \num{10.7} & \num{ 179.0} \\
          100 &  \num{ 0} & \num{3600.00} & \SI{20.95}{\percent} & \SI{0.14}{\percent} & --- & ---  &  \num{ 0} & \num{3600.00} & \SI{0.24}{\percent} & \SI{0.04}{\percent} & --- & --- \\
          \midrule
          \multicolumn{13}{@{}l}{chain length: 4}\\
          \multicolumn{13}{@{}l}{$B = 1$}\\
           20 &  \num{ 30} & \num{ 4.87} & \SI{31.53}{\percent} & \SI{4.27}{\percent} & \num{ 3.2} & \num{ 9.1}  &  \num{ 30} & \num{ 2.66} & \SI{18.13}{\percent} & \SI{6.05}{\percent} & \num{ 2.9} & \num{ 7.7} \\
           50 &  \num{ 13} & \num{1749.18} & \SI{19.61}{\percent} & \SI{0.62}{\percent} & \num{ 6.9} & \num{ 23.6}  &  \num{ 18} & \num{1105.75} & \SI{0.21}{\percent} & \SI{0.09}{\percent} & \num{ 5.6} & \num{ 20.4} \\
          100 &  \num{ 0} & \num{3600.00} & \SI{25.56}{\percent} & \SI{0.18}{\percent} & --- & ---  &  \num{ 1} & \num{3554.88} & \SI{0.24}{\percent} & \SI{0.05}{\percent} & \num{ 5.0} & \num{ 37.0} \\
          \multicolumn{13}{@{}l}{$B = 2$}\\
           20 &  \num{ 30} & \num{ 8.64} & \SI{35.55}{\percent} & \SI{9.41}{\percent} & \num{ 5.2} & \num{ 25.0}  &  \num{ 30} & \num{ 5.49} & \SI{18.03}{\percent} & \SI{8.23}{\percent} & \num{ 4.6} & \num{ 24.2} \\
           50 &  \num{ 3} & \num{3134.54} & \SI{17.63}{\percent} & \SI{0.55}{\percent} & \num{ 9.0} & \num{ 49.7}  &  \num{ 7} & \num{2561.70} & \SI{0.30}{\percent} & \SI{0.14}{\percent} & \num{ 8.0} & \num{ 82.9} \\
          100 &  \num{ 0} & \num{3600.00} & \SI{21.42}{\percent} & \SI{0.06}{\percent} & --- & ---  &  \num{ 0} & \num{3600.00} & \SI{0.28}{\percent} & \SI{0.10}{\percent} & --- & --- \\
          \multicolumn{13}{@{}l}{$B = 3$}\\
           20 &  \num{ 30} & \num{ 6.97} & \SI{43.25}{\percent} & \SI{9.94}{\percent} & \num{ 4.4} & \num{ 30.4}  &  \num{ 30} & \num{ 5.01} & \SI{23.38}{\percent} & \SI{15.73}{\percent} & \num{ 4.6} & \num{ 36.2} \\
           50 &  \num{ 2} & \num{3431.51} & \SI{22.22}{\percent} & \SI{0.39}{\percent} & \num{11.5} & \num{ 113.5}  &  \num{ 6} & \num{2337.49} & \SI{0.46}{\percent} & \SI{0.24}{\percent} & \num{ 8.0} & \num{ 79.3} \\
          100 &  \num{ 0} & \num{3600.00} & \SI{22.95}{\percent} & \SI{0.06}{\percent} & --- & ---  &  \num{ 0} & \num{3600.00} & \SI{0.18}{\percent} & \SI{0.03}{\percent} & --- & --- \\
          \multicolumn{13}{@{}l}{$B = 4$}\\
           20 &  \num{ 30} & \num{ 3.77} & \SI{44.70}{\percent} & \SI{5.09}{\percent} & \num{ 4.0} & \num{ 21.3}  &  \num{ 30} & \num{ 2.48} & \SI{23.50}{\percent} & \SI{9.39}{\percent} & \num{ 3.9} & \num{ 19.4} \\
           50 &  \num{ 1} & \num{3500.53} & \SI{32.45}{\percent} & \SI{0.40}{\percent} & \num{11.0} & \num{ 142.0}  &  \num{ 5} & \num{2530.14} & \SI{2.42}{\percent} & \SI{0.76}{\percent} & \num{12.0} & \num{ 226.0} \\
          100 &  \num{ 0} & \num{3600.00} & \SI{24.38}{\percent} & \SI{0.06}{\percent} & --- & ---  &  \num{ 0} & \num{3600.00} & \SI{0.19}{\percent} & \SI{0.03}{\percent} & --- & --- \\
          \bottomrule
      \end{tabular*}
      \caption{Comparison of the different methods for maximum cycle length 4 and different chain lengths, for the FSE policy.}
    \end{tiny}
  \end{table}
\end{landscape}

%% file: Tables/LiftingTables/combined-lifting-table-cyc4.tex
\begin{table}[htbp]
  \begin{tiny}
    \label{tab:lifting-cyc4}
    \begin{tabular*}{\linewidth}{@{}l@{\;\;\extracolsep{\fill}}rrrrrrrrrrrr@{}}\toprule
       & \multicolumn{6}{c}{with lifting} & \multicolumn{6}{c}{without lifting} \\
      \cmidrule{2-7} \cmidrule{8-13}
       {\bfseries cut-CC: FR} & & \multicolumn{3}{c}{time} & & & & \multicolumn{3}{c}{time} & & \\
      \cmidrule{3-5} \cmidrule{9-11}
       & \#opt & total & stage 2 & stage 3 & \#att. & \#sub & \#opt & total & stage 2 & stage 3 & \#att. & \#sub \\
      \midrule
      \multicolumn{13}{@{}l}{chain length: 2}\\
      $B = 1$ &  \num{ 28} & \num{ 591.86} & \SI{5.19}{\percent} & \SI{24.36}{\percent} & \num{ 3.4} & \num{ 25.0}  &  \num{ 26} & \num{ 893.61} & \SI{5.14}{\percent} & \SI{37.28}{\percent} & \num{ 3.5} & \num{ 45.3} \\
      $B = 2$ &  \num{ 23} & \num{1052.83} & \SI{5.88}{\percent} & \SI{25.35}{\percent} & \num{ 4.6} & \num{ 64.8}  &  \num{ 15} & \num{1904.66} & \SI{7.79}{\percent} & \SI{41.12}{\percent} & \num{ 4.9} & \num{ 147.7} \\
      $B = 3$ &  \num{ 15} & \num{1873.45} & \SI{5.75}{\percent} & \SI{20.20}{\percent} & \num{ 5.3} & \num{ 93.6}  &  \num{ 8} & \num{2972.21} & \SI{13.30}{\percent} & \SI{30.27}{\percent} & \num{ 5.9} & \num{ 333.6} \\
      $B = 4$ &  \num{ 17} & \num{1883.98} & \SI{7.87}{\percent} & \SI{21.47}{\percent} & \num{ 6.4} & \num{ 127.2}  &  \num{ 3} & \num{3410.61} & \SI{19.56}{\percent} & \SI{24.22}{\percent} & \num{ 9.0} & \num{ 534.7} \\
      \midrule
      \multicolumn{13}{@{}l}{chain length: 3}\\
      $B = 1$ &  \num{ 24} & \num{1182.61} & \SI{6.08}{\percent} & \SI{28.47}{\percent} & \num{ 3.0} & \num{ 27.2}  &  \num{ 21} & \num{1708.52} & \SI{6.25}{\percent} & \SI{44.78}{\percent} & \num{ 3.0} & \num{ 45.8} \\
      $B = 2$ &  \num{ 14} & \num{2350.91} & \SI{5.12}{\percent} & \SI{27.02}{\percent} & \num{ 4.5} & \num{ 68.4}  &  \num{ 7} & \num{3107.82} & \SI{5.38}{\percent} & \SI{29.00}{\percent} & \num{ 5.1} & \num{ 113.4} \\
      $B = 3$ &  \num{ 8} & \num{2934.43} & \SI{4.39}{\percent} & \SI{15.07}{\percent} & \num{ 5.4} & \num{ 81.6}  &  \num{ 1} & \num{3533.76} & \SI{5.03}{\percent} & \SI{22.44}{\percent} & \num{12.0} & \num{ 366.0} \\
      $B = 4$ &  \num{ 8} & \num{3092.19} & \SI{4.83}{\percent} & \SI{14.73}{\percent} & \num{ 6.4} & \num{ 120.6}  &  \num{ 1} & \num{3568.59} & \SI{6.95}{\percent} & \SI{17.52}{\percent} & \num{15.0} & \num{ 825.0} \\
      \midrule
      \multicolumn{13}{@{}l}{chain length: 4}\\
      $B = 1$ &  \num{ 3} & \num{3234.14} & \SI{10.12}{\percent} & \SI{8.90}{\percent} & \num{ 3.7} & \num{ 31.7}  &  \num{ 2} & \num{3327.35} & \SI{9.59}{\percent} & \SI{11.08}{\percent} & \num{ 3.0} & \num{ 55.5} \\
      $B = 2$ &  \num{ 3} & \num{3488.09} & \SI{9.02}{\percent} & \SI{8.46}{\percent} & \num{ 4.7} & \num{ 82.7}  &  \num{ 0} & \num{3600.00} & \SI{10.37}{\percent} & \SI{5.14}{\percent} & --- & --- \\
      $B = 3$ &  \num{ 1} & \num{3541.42} & \SI{8.00}{\percent} & \SI{3.20}{\percent} & \num{ 7.0} & \num{ 122.0}  &  \num{ 0} & \num{3600.00} & \SI{9.87}{\percent} & \SI{1.76}{\percent} & --- & --- \\
      $B = 4$ &  \num{ 0} & \num{3600.00} & \SI{7.04}{\percent} & \SI{1.37}{\percent} & --- & ---  &  \num{ 0} & \num{3600.00} & \SI{5.48}{\percent} & \SI{6.40}{\percent} & --- & --- \\
      \midrule
      {\bfseries cut-CC: FSE} & & \multicolumn{3}{c}{time} & & & & \multicolumn{3}{c}{time} & & \\
      \cmidrule{3-5} \cmidrule{9-11}
       & \#opt & total & stage 2 & stage 3 & \#att. & \#sub & \#opt & total & stage 2 & stage 3 & \#att. & \#sub \\
      \midrule
     \multicolumn{13}{@{}l}{chain length: 2}\\
      $B = 1$ &  \num{ 3} & \num{3394.74} & \SI{21.92}{\percent} & \SI{0.37}{\percent} & \num{ 5.7} & \num{ 24.0}  &  \num{ 2} & \num{3481.47} & \SI{21.97}{\percent} & \SI{0.36}{\percent} & \num{ 6.5} & \num{ 27.5} \\
      $B = 2$ &  \num{ 0} & \num{3600.00} & \SI{13.46}{\percent} & \SI{0.27}{\percent} & --- & ---  &  \num{ 0} & \num{3600.00} & \SI{15.30}{\percent} & \SI{0.31}{\percent} & --- & --- \\
      $B = 3$ &  \num{ 0} & \num{3600.00} & \SI{13.77}{\percent} & \SI{0.18}{\percent} & --- & ---  &  \num{ 0} & \num{3600.00} & \SI{15.24}{\percent} & \SI{0.21}{\percent} & --- & --- \\
      $B = 4$ &  \num{ 0} & \num{3600.00} & \SI{13.49}{\percent} & \SI{0.11}{\percent} & --- & ---  &  \num{ 0} & \num{3600.00} & \SI{15.20}{\percent} & \SI{0.12}{\percent} & --- & --- \\
      \midrule
      \multicolumn{13}{@{}l}{chain length: 3}\\
      $B = 1$ &  \num{ 0} & \num{3600.00} & \SI{32.83}{\percent} & \SI{0.45}{\percent} & --- & ---  &  \num{ 0} & \num{3600.00} & \SI{32.01}{\percent} & \SI{0.38}{\percent} & --- & --- \\
      $B = 2$ &  \num{ 0} & \num{3600.00} & \SI{17.82}{\percent} & \SI{0.20}{\percent} & --- & ---  &  \num{ 0} & \num{3600.00} & \SI{19.59}{\percent} & \SI{0.20}{\percent} & --- & --- \\
      $B = 3$ &  \num{ 0} & \num{3600.00} & \SI{19.97}{\percent} & \SI{0.14}{\percent} & --- & ---  &  \num{ 0} & \num{3600.00} & \SI{21.03}{\percent} & \SI{0.16}{\percent} & --- & --- \\
      $B = 4$ &  \num{ 0} & \num{3600.00} & \SI{24.40}{\percent} & \SI{0.18}{\percent} & --- & ---  &  \num{ 0} & \num{3600.00} & \SI{20.95}{\percent} & \SI{0.14}{\percent} & --- & --- \\
      \midrule
      \multicolumn{13}{@{}l}{chain length: 4}\\
      $B = 1$ &  \num{ 0} & \num{3600.00} & \SI{27.36}{\percent} & \SI{0.17}{\percent} & --- & ---  &  \num{ 0} & \num{3600.00} & \SI{25.56}{\percent} & \SI{0.18}{\percent} & --- & --- \\
      $B = 2$ &  \num{ 0} & \num{3600.00} & \SI{21.07}{\percent} & \SI{0.05}{\percent} & --- & ---  &  \num{ 0} & \num{3600.00} & \SI{21.42}{\percent} & \SI{0.06}{\percent} & --- & --- \\
      $B = 3$ &  \num{ 0} & \num{3600.00} & \SI{23.64}{\percent} & \SI{0.06}{\percent} & --- & ---  &  \num{ 0} & \num{3600.00} & \SI{22.95}{\percent} & \SI{0.06}{\percent} & --- & --- \\
      $B = 4$ &  \num{ 0} & \num{3600.00} & \SI{24.25}{\percent} & \SI{0.06}{\percent} & --- & ---  &  \num{ 0} & \num{3600.00} & \SI{24.38}{\percent} & \SI{0.06}{\percent} & --- & --- \\
      \midrule
      {\bfseries cut-PICEF: FR} & & \multicolumn{3}{c}{time} & & & & \multicolumn{3}{c}{time} & & \\
      \cmidrule{3-5} \cmidrule{9-11}
       & \#opt & total & stage 2 & stage 3 & \#att. & \#sub & \#opt & total & stage 2 & stage 3 & \#att. & \#sub \\
      \midrule
      \multicolumn{13}{@{}l}{chain length: 2}\\
      $B = 1$ &  \num{ 26} & \num{ 771.57} & \SI{7.97}{\percent} & \SI{35.34}{\percent} & \num{ 3.3} & \num{ 22.6}  &  \num{ 26} & \num{ 918.91} & \SI{5.62}{\percent} & \SI{43.40}{\percent} & \num{ 3.5} & \num{ 42.1} \\
      $B = 2$ &  \num{ 22} & \num{1284.22} & \SI{6.97}{\percent} & \SI{43.76}{\percent} & \num{ 4.5} & \num{ 50.4}  &  \num{ 15} & \num{2154.89} & \SI{8.62}{\percent} & \SI{44.55}{\percent} & \num{ 5.3} & \num{ 142.7} \\
      $B = 3$ &  \num{ 12} & \num{2539.34} & \SI{6.39}{\percent} & \SI{33.10}{\percent} & \num{ 6.1} & \num{ 62.3}  &  \num{ 6} & \num{3254.15} & \SI{11.60}{\percent} & \SI{38.67}{\percent} & \num{ 7.8} & \num{ 312.8} \\
      $B = 4$ &  \num{ 13} & \num{2464.68} & \SI{7.09}{\percent} & \SI{31.69}{\percent} & \num{ 6.5} & \num{ 81.8}  &  \num{ 1} & \num{3572.13} & \SI{16.26}{\percent} & \SI{25.13}{\percent} & \num{17.0} & \num{ 872.0} \\
      \midrule
      \multicolumn{13}{@{}l}{chain length: 3}\\
      $B = 1$ &  \num{ 25} & \num{ 785.72} & \SI{9.38}{\percent} & \SI{41.37}{\percent} & \num{ 3.1} & \num{ 23.0}  &  \num{ 27} & \num{1142.96} & \SI{6.47}{\percent} & \SI{52.07}{\percent} & \num{ 3.2} & \num{ 38.7} \\
      $B = 2$ &  \num{ 21} & \num{1847.92} & \SI{6.14}{\percent} & \SI{34.68}{\percent} & \num{ 5.3} & \num{ 50.9}  &  \num{ 10} & \num{2824.96} & \SI{6.82}{\percent} & \SI{48.93}{\percent} & \num{ 6.0} & \num{ 133.5} \\
      $B = 3$ &  \num{ 11} & \num{2722.26} & \SI{6.79}{\percent} & \SI{36.65}{\percent} & \num{ 6.5} & \num{ 60.5}  &  \num{ 3} & \num{3453.67} & \SI{9.41}{\percent} & \SI{46.40}{\percent} & \num{ 9.0} & \num{ 215.3} \\
      $B = 4$ &  \num{ 11} & \num{2741.17} & \SI{6.12}{\percent} & \SI{39.37}{\percent} & \num{ 6.2} & \num{ 86.4}  &  \num{ 0} & \num{3600.00} & \SI{11.48}{\percent} & \SI{42.81}{\percent} & --- & --- \\
      \midrule
      \multicolumn{13}{@{}l}{chain length: 4}\\
      $B = 1$ &  \num{ 26} & \num{ 951.26} & \SI{7.48}{\percent} & \SI{42.90}{\percent} & \num{ 2.6} & \num{ 23.2}  &  \num{ 24} & \num{1399.71} & \SI{5.73}{\percent} & \SI{60.70}{\percent} & \num{ 2.9} & \num{ 42.4} \\
      $B = 2$ &  \num{ 22} & \num{1770.23} & \SI{6.38}{\percent} & \SI{44.15}{\percent} & \num{ 4.4} & \num{ 50.1}  &  \num{ 7} & \num{2970.19} & \SI{6.55}{\percent} & \SI{52.23}{\percent} & \num{ 5.1} & \num{ 123.4} \\
      $B = 3$ &  \num{ 8} & \num{2954.78} & \SI{6.64}{\percent} & \SI{37.47}{\percent} & \num{ 6.1} & \num{ 60.9}  &  \num{ 1} & \num{3575.11} & \SI{6.89}{\percent} & \SI{38.18}{\percent} & \num{ 9.0} & \num{ 307.0} \\
      $B = 4$ &  \num{ 10} & \num{2997.36} & \SI{6.43}{\percent} & \SI{32.87}{\percent} & \num{ 6.8} & \num{ 93.5}  &  \num{ 0} & \num{3600.00} & \SI{10.26}{\percent} & \SI{36.56}{\percent} & --- & --- \\
       \midrule
      {\bfseries cut-PICEF: FSE} & & \multicolumn{3}{c}{time} & & & & \multicolumn{3}{c}{time} & & \\
      \cmidrule{3-5} \cmidrule{9-11}
       & \#opt & total & stage 2 & stage 3 & \#att. & \#sub & \#opt & total & stage 2 & stage 3 & \#att. & \#sub \\
      \midrule
     \multicolumn{13}{@{}l}{chain length: 2}\\
      $B = 1$ &  \num{ 2} & \num{3202.94} & \SI{0.35}{\percent} & \SI{0.11}{\percent} & \num{ 4.0} & \num{ 14.5}  &  \num{ 3} & \num{3494.44} & \SI{0.32}{\percent} & \SI{0.11}{\percent} & \num{ 6.3} & \num{ 29.7} \\
      $B = 2$ &  \num{ 0} & \num{3600.00} & \SI{0.36}{\percent} & \SI{0.13}{\percent} & --- & ---  &  \num{ 0} & \num{3600.00} & \SI{0.34}{\percent} & \SI{0.10}{\percent} & --- & --- \\
      $B = 3$ &  \num{ 0} & \num{3600.00} & \SI{0.27}{\percent} & \SI{0.07}{\percent} & --- & ---  &  \num{ 0} & \num{3600.00} & \SI{0.29}{\percent} & \SI{0.07}{\percent} & --- & --- \\
      $B = 4$ &  \num{ 0} & \num{3600.00} & \SI{0.29}{\percent} & \SI{0.07}{\percent} & --- & ---  &  \num{ 0} & \num{3600.00} & \SI{0.33}{\percent} & \SI{0.10}{\percent} & --- & --- \\
      \midrule
      \multicolumn{13}{@{}l}{chain length: 3}\\
      $B = 1$ &  \num{ 1} & \num{3477.97} & \SI{0.30}{\percent} & \SI{0.09}{\percent} & \num{ 4.0} & \num{ 22.0}  &  \num{ 2} & \num{3444.94} & \SI{0.29}{\percent} & \SI{0.08}{\percent} & \num{ 4.5} & \num{ 19.0} \\
      $B = 2$ &  \num{ 0} & \num{3600.00} & \SI{0.21}{\percent} & \SI{0.03}{\percent} & --- & ---  &  \num{ 0} & \num{3600.00} & \SI{0.23}{\percent} & \SI{0.45}{\percent} & --- & --- \\
      $B = 3$ &  \num{ 0} & \num{3600.00} & \SI{0.23}{\percent} & \SI{0.05}{\percent} & --- & ---  &  \num{ 0} & \num{3600.00} & \SI{0.22}{\percent} & \SI{0.03}{\percent} & --- & --- \\
      $B = 4$ &  \num{ 0} & \num{3600.00} & \SI{0.23}{\percent} & \SI{0.05}{\percent} & --- & ---  &  \num{ 0} & \num{3600.00} & \SI{0.24}{\percent} & \SI{0.04}{\percent} & --- & --- \\
      \midrule
      \multicolumn{13}{@{}l}{chain length: 4}\\
      $B = 1$ &  \num{ 1} & \num{3570.22} & \SI{0.26}{\percent} & \SI{0.07}{\percent} & \num{ 5.0} & \num{ 32.0}  &  \num{ 1} & \num{3554.88} & \SI{0.24}{\percent} & \SI{0.05}{\percent} & \num{ 5.0} & \num{ 37.0} \\
      $B = 2$ &  \num{ 0} & \num{3600.00} & \SI{0.29}{\percent} & \SI{0.15}{\percent} & --- & ---  &  \num{ 0} & \num{3600.00} & \SI{0.28}{\percent} & \SI{0.10}{\percent} & --- & --- \\
      $B = 3$ &  \num{ 0} & \num{3600.00} & \SI{0.21}{\percent} & \SI{0.04}{\percent} & --- & ---  &  \num{ 0} & \num{3600.00} & \SI{0.18}{\percent} & \SI{0.03}{\percent} & --- & --- \\
      $B = 4$ &  \num{ 0} & \num{3600.00} & \SI{0.28}{\percent} & \SI{0.07}{\percent} & --- & ---  &  \num{ 0} & \num{3600.00} & \SI{0.19}{\percent} & \SI{0.03}{\percent} & --- & --- \\
      \bottomrule
    \end{tabular*}
    \caption{Impact of cut lifting restricted to instances with 100 pairs / NDDs, for cut-CC and cut-PICEF, given maximum cycle length $K = 4$ and different chain lengths, for Full Recourse (FR) and FSE}
  \end{tiny}
\end{table}

%% file: Tables/AggregatedLifting/combined_aggr_lifting_table-cyc4-policy1.tex
\begin{table}[htbp]
  \begin{scriptsize}
    \begin{tabular*}{\textwidth}{@{}l@{\;\;\extracolsep{\fill}}rrrrrrrr@{}}\toprule
      {\bfseries Cut-CC} & \multicolumn{4}{c}{with lifting} & \multicolumn{4}{c}{without lifting}\\
      \cmidrule{2-5} \cmidrule{6-9}
      \#vertices & \#opt & time & \#att. & \#sub & \#opt & time & \#att. & \#sub \\
      \midrule
      \multicolumn{9}{@{}l}{chain length: 2}\\
       20 &  120 & 0.3 & 2.8 & 8.8 &  120 & 0.4 & 2.9 & 13.0\\
       50 &  120 & 33.4 & 5.3 & 45.0 &  118 & 78.5 & 5.5 & 151.3\\
      100 &  83 & 1350.5 & 4.9 & 77.7 &  52 & 2295.3 & 5.8 & 265.3\\
      \midrule
      \multicolumn{9}{@{}l}{chain length: 3}\\
       20 &  120 & 0.4 & 2.9 & 9.9 &  120 & 0.5 & 3.0 & 13.4\\
       50 &  120 & 71.2 & 5.2 & 51.9 &  117 & 136.9 & 5.3 & 162.4\\
      100 &  54 & 2390.0 & 4.8 & 74.5 &  30 & 2979.7 & 8.8 & 337.6\\
      \midrule
      \multicolumn{9}{@{}l}{chain length: 4}\\
       20 &  120 & 0.9 & 2.9 & 9.5 &  120 & 1.0 & 2.9 & 13.7\\
       50 &  115 & 325.7 & 5.0 & 48.4 &  103 & 523.7 & 5.2 & 144.9\\
      100 &  7 & 3465.9 & 3.8 & 59.1 &  2 & 3531.8 & 0.8 & 13.9\\
      \midrule
       {\bfseries Cut-PICEF} & \multicolumn{4}{c}{with lifting} & \multicolumn{4}{c}{without lifting}\\
      \cmidrule{2-5} \cmidrule{6-9}
      \#vertices & \#opt & time & \#att. & \#sub & \#opt & time & \#att. & \#sub \\
      \midrule
      \multicolumn{9}{@{}l}{chain length: 2}\\
       20 &  120 & 0.3 & 2.8 & 8.8 &  120 & 0.4 & 2.9 & 12.5\\
       50 &  120 & 43.0 & 5.3 & 41.3 &  117 & 99.1 & 5.5 & 142.2\\
      100 &  73 & 1765.0 & 5.1 & 54.3 &  48 & 2475.0 & 8.4 & 342.4\\
      \midrule
      \multicolumn{9}{@{}l}{chain length: 3}\\
       20 &  120 & 0.6 & 2.9 & 8.7 &  120 & 0.6 & 3.0 & 11.8\\
       50 &  120 & 78.0 & 5.2 & 42.9 &  116 & 140.7 & 5.5 & 138.7\\
      100 &  68 & 2024.3 & 5.3 & 55.2 &  40 & 2755.4 & 4.6 & 96.9\\
      \midrule
      \multicolumn{9}{@{}l}{chain length: 4}\\
       20 &  120 & 0.8 & 3.0 & 8.9 &  120 & 0.7 & 3.0 & 12.3\\
       50 &  120 & 120.3 & 5.3 & 45.8 &  116 & 217.3 & 5.4 & 149.8\\
      100 &  66 & 2168.4 & 5.0 & 56.9 &  32 & 2886.3 & 4.3 & 118.2\\
      \bottomrule
    \end{tabular*}
    \caption{Comparison of the cut-CC method with and without lifting for maximum cycle length 4 and different chain lengths, for the Full Recourse policy.}
    \label{tab:aggrlifting_cyc4-method1-policy1}
  \end{scriptsize}
\end{table}